\documentclass[preprint]{imsart}
\RequirePackage[OT1]{fontenc}
\RequirePackage{amsthm,amsmath, amssymb, latexsym}
\RequirePackage[numbers, sort&compress]{natbib}
\RequirePackage[colorlinks,citecolor=blue,urlcolor=blue]{hyperref}
\usepackage{comment}
\usepackage{graphicx}
\usepackage{bbm}
\usepackage{color}
\usepackage{relsize}
\usepackage{csquotes}
\usepackage{mathabx}

\arxiv{arXiv:0000.0000}

\newcommand{\mycaption}[1]{
\begin{quote}
\centering
  {{\bf Figure} \arabic{figure}: #1}
  \end{quote}
 \vspace{0.01cm}
 \stepcounter{figure}}

\usepackage[justification=centering]{caption}
\usepackage{geometry}
\newcommand{\bN}{\mathbb{N}}

\newcommand{\bR}{\mathbb{R}}
\newcommand{\bT}{\mathbb{T}}

\newcommand{\cF}{\mathcal{F}}

\newcommand{\cP}{\mathcal{P}}
\newcommand{\cT}{\mathcal{T}}
\newcommand{\cR}{\mathcal{R}}

\startlocaldefs
\numberwithin{equation}{section}
\theoremstyle{plain}
\newtheorem{theorem}{Theorem}[section]
\theoremstyle{remark}
\newtheorem{remark}[theorem]{Remark}
\theoremstyle{definition}
\newtheorem{definition}[theorem]{Definition}
\newtheorem{algorithm}[theorem]{Algorithm}

\theoremstyle{plain}
\newtheorem{corollary}[theorem]{Corollary}
\newtheorem{lemma}[theorem]{Lemma}
\newtheorem{prop}[theorem]{Proposition}
\endlocaldefs

\begin{document}

\begin{frontmatter}

\title{A binary embedding \\ of the stable line-breaking construction}

\runtitle{Binary embedding of the stable line-breaking construction}

\begin{aug}

\author{\fnms{Franz} \snm{Rembart}\thanksref{T1}\ead[label=e1]{rembart@stats.ox.ac.uk}}
\and
\author{\fnms{Matthias} \snm{Winkel}\thanksref{T2}\ead[label=e2]{winkel@stats.ox.ac.uk}}
\address{Department of Statistics, University of Oxford, 24--29 St Giles, Oxford OX1 3LB \\ \printead{e1,e2}}
\today
\runauthor{F. Rembart and M. Winkel}

\affiliation{University of Oxford}
\thankstext{T1}{Supported by EPSRC grant EP/P505666/1}
\thankstext{T2}{Supported by EPSRC grant EP/K029797/1}

\end{aug}

\begin{abstract} We embed Duquesne and Le Gall's stable tree into a binary compact continuum random tree (CRT) in a way that solves an open problem posed by Goldschmidt and Haas. This CRT can be obtained by applying a recursive construction method of compact CRTs as presented in earlier work to a specific distribution of a random string of beads, i.e.\ a random interval equipped with a random discrete measure. We also express this CRT as a tree built by replacing all branch points of a stable tree by rescaled i.i.d. copies of a Ford CRT. Some of these developments are carried out in a space of $\infty$-marked metric spaces generalising Miermont's notion of a $k$-marked metric space.

\end{abstract}



\end{frontmatter}


\section{Introduction}\label{fig1}
\label{Intro}

Stable trees were introduced by Duquesne and Le Gall 
\cite{17} as a family of continuum random trees (CRTs) parametrised by a self-similarity parameter $\alpha \in (1,2]$ to describe the genealogical structure of continuous-state branching processes with branching mechanism $\lambda \mapsto \lambda^{\alpha}$. As such they form a 
subclass of L\'evy trees \cite{23} and contain Aldous's Brownian CRT \cite{6,7,8} as a special case ($\alpha=2$). They were studied by Miermont and others  
\cite{2, 16, 17, 23,33,35,36,38} in the context of self-similar fragmentations and by several authors to establish invariance principles \cite{24,Kor12,CP09,HM10,Die13} 
and other properties \cite{11,CK14}. Furthermore, they have deeper connections to random maps and Liouville quantum gravity \cite{Subord,DMS14,MS15}. 

We represent trees as \textit{$\mathbb R$-trees}, i.e.\ compact metric spaces $(\mathcal T, d)$ such that any two points $x,y \in \mathcal T$ are 
connected by a unique path $[[x,y]]$ in $\mathcal{T}$, which is furthermore required to have length $d(x,y)$. All our $\mathbb{R}$-trees are \em rooted \em at a distinguished $\rho\in\mathcal{T}$. We refer to a 
rooted $\mathbb R$-tree $(\mathcal T, d, \rho)$ equipped with a probability measure $\mu$ as a 
\textit{weighted} $\mathbb R$-tree $(\mathcal T, d, \rho, \mu)$, and equip sets of isometry classes of $\mathbb{R}$-trees and weighted $\mathbb R$-trees with the 
Gromov-Hausdorff and the Gromov-Hausdorff-Prokhorov topology, respectively.

Ever since Aldous \cite{6}, such trees have been built sequentially from a single branch $[[\rho,\Sigma_0]]$, grafting further branches
(line segments) $]]J_{k-1},\Sigma_k]]$ to build trees $\mathcal T_k$ spanned by a growing finite number of points $\rho,\Sigma_0,\ldots,\Sigma_k$, $k\ge 1$, finally
passing to the closure/completion $\mathcal T$ of $\bigcup_{k\ge 0}\mathcal T_k$. In a given weighted $\mathbb R$-tree $(\mathcal T,d,\rho,\mu)$, a natural sequence
$(\Sigma_k,k\ge 0)$ may be obtained as an independent sample from $\mu$. 
For the Brownian CRT, Aldous \cite{6} gave an autonomous description of the resulting tree-growth process $(\mathcal T_k,k\ge 0)$ by breaking the half-line $[0,\infty)$ 
at the 
points $(S_k,k\ge 0)$ of an inhomogeneous Poisson process with linearly growing intensity $tdt$ on $[0,\infty)$, each segment $]S_k,S_{k+1}]$ grafted in a point $J_k$ chosen 
uniformly from the length measure on the structure $\mathcal T_k$ already built, with $\mathcal T_0=[0,S_0]$. 

In Aldous's construction, the \em branch points \em $J_k$, $k\ge 0$, are distinct, the trees \em binary\em. This construction reveals some of the 
local complexity of the limiting tree, since elementary thinning of Poisson processes shows that every branch receives a dense set of branch points. Goldschmidt and Haas 
\cite{12} generalised this line-breaking construction to all stable trees $(\mathcal T,d,\rho,\mu)$, which are not binary for $\alpha\in(1,2)$. They describe
\begin{equation}
\mathcal T_k=\bigcup_{i=0}^k[[\rho,\Sigma_i]],\quad k \geq  0,\qquad\mbox{ for a sample }\Sigma_i\sim\mu,i\ge 0, \label{findimmarg1}
\end{equation}
not quite autonomously, as Aldous does in the special case $\alpha=2$, but by assigning weights 
\begin{equation} W^{(i)}_{k}, \quad i \in [b_k], \quad k \geq 0, \label{weights} 
\end{equation} 
to each branch point $v_i$ of $\mathcal T_k$, where $(v_i, i \geq 1)$ is the sequence of \em distinct branch points \em in their order of appearance in  
$(\mathcal T_k, k \geq 0)$, and $b_k\ge 0$ is the number of branch points of $\mathcal T_k$. Here, $[b]:=\{1,\ldots,b\}$. 

Specifically, it will be convenient to change the usual parametrisation of the stable trees from a parameter $\alpha \in (1,2]$ to an \em index \em  
$\beta =1-1/\alpha \in (0,1/2]$. For $k \geq 0$, the sum of the branch point weights $(W_k^{(i)}, i \in[b_k])$ and the total length $L_k={\rm Leb}(\mathcal T_k)$ of 
$\mathcal T_k$ is given by $S_k$, where $(S_k, k \geq 0)$ is the \em Mittag-Leffler Markov chain \em \cite{2,12,39} with parameter $\beta$,  
starting from $S_0 \sim \text{ML}(\beta, \beta)$ with transition density \vspace{-0.2cm}
\begin{equation*} 
  f_{S_{k+1}\lvert S_k=z} \left( y \right)=f(z,y)=\frac{1-\beta}{\Gamma(1/\beta)}(y-z)^{1/\beta-2}\frac{yg_{\beta}(y)}{g_\beta(z)}, \quad 0 < z <y, \quad k \geq 0. \end{equation*} 
Here, $\text{ML}(\alpha, \theta)$ denotes the Mittag-Leffler distribution with parameters $ 0 < \alpha < 1$ and $\theta > -\alpha$ (cf. Section \ref{secml}), and 
$g_\beta(\cdot)$ is the density of ML$(\beta, 0)$. Then
$S_k=W_k^{(1)}+\cdots+W_k^{(b_k)} + L_k \sim {\rm{ML}}(\beta, \beta+k)$.


\begin{algorithm}[Goldschmidt-Haas \cite{12}] \label{GH}
Let $\beta\in(0,1/2]$. We grow discrete $\mathbb R$-trees $\mathcal T_k$ with weights $W_k^{(i)}$ in the branch points $v_i$, $i\in[b_k]$, of $\mathcal T_k$, and edge lengths between vertices, as follows.
  \begin{enumerate}\item[0.] Let $(\mathcal T_0,\rho)$ be isometric to $([0,S_0],0)$, where $S_0\sim{\rm ML}(\beta,\beta)$; let $b_0=0$ and $W_0^{(i)}=0$, $i\ge 1$.
  \end{enumerate}
  Given $(\mathcal T_j,(v_i,i\in[b_j]),(W_j^{(i)},i\in[b_j]))$, $0\le j\le k$, and $S_k=L_k+W_k^{(1)}+\cdots+W_k^{(b_k)}$, where $L_k={\rm Leb}(\mathcal T_k)$,   
  \begin{enumerate}\item select $I_k=i$ for each branch point $v_i$ of $\mathcal T_k$ with probability proportional to $W_k^{(i)}$, $i\in[b_k]$; 
    or select an edge $E_k\subset\mathcal T_k$ with probability proportional to its length and let $b_{k+1}=b_k+1$, $I_k=b_{k+1}$;
    \item if an edge $E_k$ is selected, sample $v_{b_{k+1}}$ from the normalised length measure on $E_k$; 
    \item sample $S_{k+1}$ with density $f(S_k,\cdot)$ and an independent $B_k\sim{\rm Beta}(1,1/\beta-2)$; attach to $\mathcal T_k$ at $J_k:=v_{I_k}$ a new branch of
      length $(S_{k+1}-S_k)B_k$ to form $\mathcal T_{k+1}$; increase the weight of $J_k=v_{I_k}$ to 
	  $$W_{k+1}^{(I_k)}=W_k^{(I_k)}+(S_{k+1}-S_k)(1-B_k),\qquad \text{and set}\quad W_{k+1}^{(j)}=W_k^{(j)},\quad j \neq I_k.$$ 
  \end{enumerate}
\end{algorithm}

When $\beta=1/2$, we understand $B_k=1$, so $W_k^{(i)}=0$ for all $i \geq 1$, $k \geq 0$, and $L_k=S_k$ for all $k \geq 0$. We obtain a sequence of compact binary 
$\mathbb R$-trees whose evolution is determined by attachment points chosen uniformly at random according to the length measure, and the total length given by the 
Mittag-Leffler Markov chain of parameter $\beta=1/2$, which can be seen to correspond to an inhomogeneous Poisson process of rate $\frac{1}{2}tdt$. Hence, this reduces to Aldous's 
line-breaking construction of the Brownian CRT \cite{6}.

It was shown in \cite{12} that the sequence of trees $(\mathcal T_k, k \geq 0)$ in Algorithm \ref{GH} has the same distribution as the sequence of trees from 
\eqref{findimmarg1}, i.e. we can formally define the \em stable tree of index $\beta \in (0,1/2]$ \em as the (Gromov-Hausdorff) limit $\mathcal T$ of $\mathcal T_k$, as 
$k\rightarrow\infty$. See also \cite{12} for an alternative line-breaking construction of the sequence $(\mathcal T_k, k \geq 0)$, where branch point selection is based 
on vertex degrees instead of weights. 

Goldschmidt and Haas \cite{12} asked if there was a sensible way to associate a notion of length with the branch point weights in Algorithm \ref{GH}. We answer this 
question by using the branch point weights to build rescaled Ford trees whose lengths correspond to these weights. Ford trees arise in the scaling limit of Ford's alpha 
model studied in \cite{9,2} and in the context of the alpha-gamma-model \cite{10} for $\gamma=\alpha$, which is also related to the stable tree in the case when 
$\gamma=1-\alpha$. Ford trees are examples of binary self-similar CRTs and have also been constructed via line-breaking:


\begin{algorithm}[Haas-Miermont-Pitman-Winkel \cite {1,2}] \label{Ford} 
Let $\beta^\prime\in(0,1)$. We grow $\bR$-trees $\cF_m$, $m\ge 1$:
  \begin{enumerate}\setcounter{enumi}{-1}\item Let $(\cF_1,\rho)$ be isometric to $([0,S_1^\prime],0)$, where $S_1^\prime\sim{\rm ML}(\beta^\prime,1-\beta^\prime)$.
  \end{enumerate}
  Given $\cF_j$, $1\le j\le m$, let $S_m^\prime={\rm Leb}(\cF_m)$ denote the length of $\cF_m$;
  \begin{enumerate}\item select an edge $E_m\subset\cF_m$ with probability proportional to its length; 
    \item if $E_m$ is external, sample $D_m\sim{\rm Beta}(1,1/\beta^\prime-1)$ and place $J_m\in E_m$ to split $E_m$ into length proportions $D_m$ and $1-D_m$; 
    otherwise, sample $J_m$ from the normalised length measure on $E_m$; 
    \item sample $S_{m+1}^\prime$ with density $f(S_m^\prime,\cdot)$; attach to $\cF_m$ at $J_m$ an edge of length $S_{m+1}^\prime\!-\!S_m^\prime$ to form
      $\cF_{m+1}$. 
  \end{enumerate}
\end{algorithm}

The sequence of trees $(\mathcal F_m, m \geq 1)$ has as its (Gromov-Hausdorff) limit a CRT $\mathcal F$ as $k \rightarrow \infty$, a so-called \textit{Ford CRT} of index 
$\beta' \in (0,1)$, see \cite{1,2}. We refer to the trees $\mathcal F_m$, $m \geq 1$, as \em Ford trees\em. In the case when $\beta'=1/2$, Algorithm \ref{Ford} corresponds to 
Aldous's construction of the Brownian CRT. 

We combine the line-breaking constructions given in Algorithms \ref{GH} and \ref{Ford} in the framework of $\infty$-marked $\mathbb R$-trees, which we introduce in 
Section \ref{IMRT} as a natural extension of Miermont's notion of $k$-marked trees \cite{22}. An $\infty$-marked $\mathbb R$-tree 
$(\mathcal T, (\mathcal R^{(i)}, i \geq 1))$ is an $\mathbb R$-tree $(\mathcal T, d, \rho)$ with non-empty closed connected subsets 
$\mathcal R^{(i)} \subset \mathcal T$, $i \geq 1$. We will refer to this setting as a \textit{two-colour} framework, meaning that the \em marked \em set 
$\bigcup_{i \geq 1} \mathcal R^{(i)}$ and the \em unmarked \em remainder $\mathcal T \setminus \bigcup_{i \geq 1} \mathcal R^{(i)}$ are associated with two different colours. 
The marked components in the line-breaking construction below correspond to rescaled Ford trees with lengths equal to the branch point weights in Algorithm \ref{GH} and 
the unmarked remainder gives rise to a stable tree. Selection of a branch point in Algorithm \ref{GH} corresponds to an insertion into the respective marked component in 
the enhanced line-breaking construction given by Algorithm \ref{twocolour}.  \pagebreak


\begin{algorithm}[Two-colour line-breaking construction] \label{twocolour}  Let $\beta\in(0,1/2]$. We grow $\infty$-marked $\bR$-trees $(\cT_k^*,(\cR_k^{(i)},i\ge 1))$, $k\ge 0$, as follows.
  \begin{enumerate}\setcounter{enumi}{-1}\item Let $(\cT_0^*,\rho)$ be isometric to $([0,S_0],0)$, where $S_0\sim{\rm ML}(\beta,\beta)$; let $r_0=0$ and $\cR_0^{(i)}=\{\rho\}$, $i\ge 1$.
  \end{enumerate}
  Given $(\cT_j^*,(\cR_j^{(i)},i\ge 1))$, $0\le j\le k$, let $S_k={\rm Leb}(\mathcal T_k^*)$ be the length of $\cT_k^*$ and $r_k=\#\{i\ge 1\colon\cR_k^{(i)}\neq\{\rho\}\}$;   
  \begin{enumerate}\item select an edge $E_k^*\subset\cT_k^*$ with probability proportional to its length; if $E_k^*\subset\cR_k^{(i)}$ for some $i\in[r_k]$, let $I_k=i$; otherwise, i.e.\ if $E_k^*\subset\cT_k^*\setminus\bigcup_{i\in[r_k]}\cR_k^{(i)}$, let $r_{k+1}\!=\!r_k\!+\!1$, $I_k\!=\!r_{k+1}$; 
    \item if $E_k^*$ is an
      external edge of $\cR_k^{(i)}$, sample $D_k\sim{\rm Beta}(1,1/\beta-2)$ and place $J_k^*$ to split $E_k^*$ into length proportions $D_k$ and
      $1-D_k$;
      otherwise, i.e.\ if $E_k^*\subset\cT_k^*\setminus\bigcup_{i\in[r_k]}\cR_k^{(i)}$ or if $E_k^*$ is an internal edge of $\cR_k^{(i)}$, sample $J_k^*$ from the normalised length measure on $E_k^*$;
    \item sample $S_{k+1}$ with density $f(S_k,\cdot)$ and an independent $B_k\sim{\rm Beta}(1,1/\beta-2)$; attach to $\cT_k^*$ at $J_k^*$ a new  
      branch of length $S_{k+1}-S_k$ to form $\cT_{k+1}^*$, and add to $\cR_k^{(I_k)}$ the part of length $(S_{k+1}-S_k)(1-B_k)$ closest to the root
      to form $\cR_{k+1}^{(I_k)}$; set $\cR_{k+1}^{(j)}=\cR_k^{(j)}$, $j\neq I_k$. 
  \end{enumerate}
\end{algorithm}

Indeed, we obtain the correspondence of the branch point weights in Algorithm \ref{GH} and the lengths of the marked subtrees in Algorithm \ref{twocolour}, as well as
marked subtrees as in Algorithm \ref{Ford}, up to scaling:



\begin{theorem}[Weight-length representation]\label{Mainresult1}  Let  $({\mathcal T}_k,({W}_k^{(i)}, i \geq 1), k \geq 0)$ be as in Algorithm \ref{GH}. Let $(\mathcal T_k^*,(\mathcal R_k^{(i)}, i \geq 1), k \geq 0)$ be the sequence of $\infty$-marked $\mathbb R$-trees constructed in Algorithm \ref{twocolour}, and let $\widetilde{W}_k^{(i)}={\rm Leb}(\mathcal R_k^{(i)})$ denote the length of $\mathcal R_k^{(i)}$, $i \geq 1$, respectively. For $k \geq 1$, contract each component $\mathcal R_k^{(i)}$ to a single branch point $\widetilde{v}_i$ by using an equivalence relation, and denote the resulting tree by $\widetilde{\mathcal T}_k$. Then
\begin{equation}\label{weighteq}
\left(\widetilde{\mathcal T}_k, \left(\widetilde{W}_k^{(i)}, i \geq 1 \right), k \geq 0\right) \,{\buildrel d \over =}\, \left({\mathcal T}_k, \left({W}_k^{(i)}, i \geq 1 \right), k \geq 0\right).
\end{equation}
See Figure \ref{fig1}. Furthermore, there exist positive random variables $C^{(i)}$ and subsequences $(k_m^{(i)}, m \geq 1)$, $i \geq 1$, such that the rescaled marked subtrees grow like Ford trees of index $\beta'=\beta/(1-\beta)$, i.e.
\begin{equation}\label{subfordgrowth}
\left(C^{(i)} \mathcal R^{(i)}_{k_m^{(i)}}, m \geq 1\right)\,{\buildrel d \over =}\,  \left(\mathcal F_m, m \geq 1 \right), \end{equation}
for all $i \geq 1$ where $(C^{(i)} \mathcal R^{(i)}_{k_m^{(i)}}, m \geq 1)$, $i \geq 1$, are independent of each other.
\end{theorem}

\begin{figure}[t]   \vspace{-0.3cm}\centering 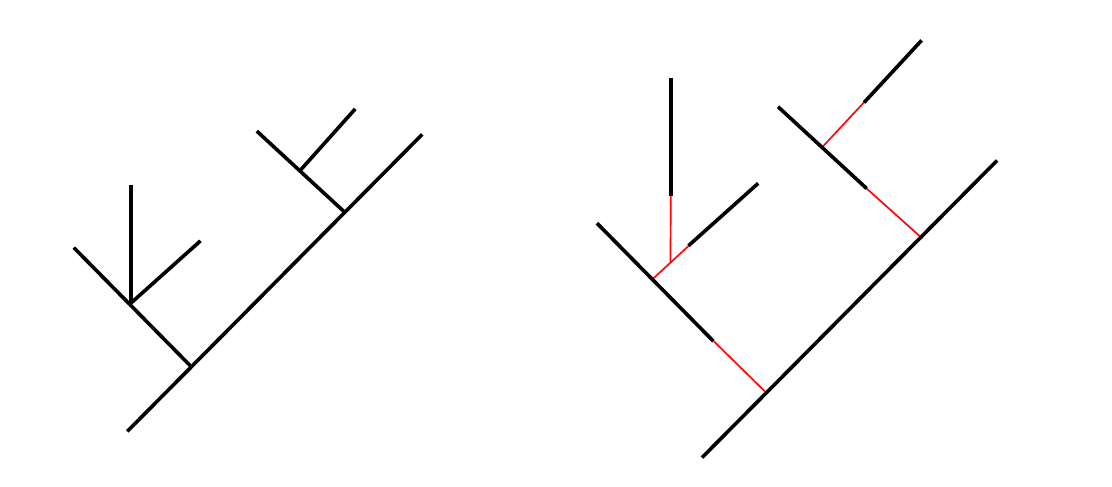\vspace{-0.2cm}  \mycaption{Example of $\widetilde{\mathcal T}_5$ with four branch points $v_1, \ldots, v_4$. \\ Branch point weights $W_5^{(1)}, \ldots, W_5^{(4)}$ are represented as lengths of marked subtrees $\mathcal R_5^{(1)}, \ldots, \mathcal R_5^{(4)}$.}\vspace{-0.2cm} \end{figure}



To obtain limiting $\infty$-marked CRTs, we introduce a suitable metric $d_{\rm GH}^\infty$ in Section \ref{MMS}.


\begin{theorem}[Convergence of two-colour trees] \label{Mainresult2} Let $(\mathcal T_k^*, (\mathcal R_k^{(i)}, i \geq 1), k \geq 0)$ be as above. Then 
\begin{equation}
\lim \limits_{k \rightarrow \infty} \left(\mathcal T_k^*,  \left(\mathcal R_k^{(i)}, i \geq 1 \right)\right) = \left(\mathcal T^*, \left( \mathcal R^{(i)}, i \geq 1 \right) \right) \quad \text{a.s.} \end{equation}
with respect to $d_{\rm GH}^\infty$, where $(\mathcal T^*, (\mathcal R^{(i)}, i \geq 1))$ is a compact $\infty$-marked $\mathbb R$-tree. Furthermore, \begin{itemize} \item the tree $\widetilde{\mathcal T}$, obtained from $\mathcal T^*$ by contracting each component $\mathcal R^{(i)}$ to a single branch point $\widetilde{v}_i$, is a stable tree of parameter $\beta$;  \item there exist scaling factors $(C^{(i)}, i \geq 1)$ such that the trees $C^{(i)} \mathcal R^{(i)}$, $i \geq 1$, are i.i.d. copies of a Ford CRT $\mathcal F$ of index $\beta'=\beta/(1-\beta)$, and the trees $C^{(i)} \mathcal R^{(i)}$, $i \geq 1$, are independent of $\widetilde{\mathcal T}$. \end{itemize} \pagebreak
\end{theorem}

The scaling factors $C^{(i)}$ can be given explicitly in terms of the masses of the subtrees of the stable tree $\widetilde{\mathcal T}$ above the branch 
point $\widetilde{v}_i$.
We can in fact use this, with the ingredients listed in Theorem \ref{Mainresult2}, to construct the two-colour tree $(\mathcal T^*,( \mathcal R^{(i)}, i \geq 1 ))$ from a stable 
tree $(\mathcal T, \mu)$ by replacing each branch point by a rescaled independent copy of a Ford CRT: 


\begin{theorem}[Branch point replacement in a stable tree] \label{branchrepl}
Let $(\mathcal T, d, \rho, \mu)$ be a stable tree of index $\beta \in (0,1/2]$ equipped with an i.i.d. sequence of labelled leaves $(\Sigma_k, k \geq 0)$ sampled from $\mu$. Consider the reduced trees $(\mathcal T_k, k \geq 0)$ as in \eqref{findimmarg1} with branch points $(v_i, i \geq 1)$ in order of appearance. For each $i \geq 1$, consider the path from the root to the leaf with the smallest label above $v_i$ and the following variables:
\begin{itemize}
\item the total mass ${P}^{(i)}=\sum_{j \geq 1} {P}_j^{(i)}$ of the subtrees rooted at $v_i$ on this path with masses $({P}_j^{(i)}, j \geq 1)$, in the
order of their smallest labels;
\item the random variable ${D}^{(i)}=\lim_{n \rightarrow \infty}\big(1-\sum_{j\in[n]}P_j^{(i)}/P^{(i)}\big)^{1-\beta}(1-\beta)^{\beta-1}n^\beta$ derived from $({P}^{(i)}_j, j \geq 1)$.
\end{itemize}
For $i \geq 1$, replace $v_i$ by an independent Ford tree $\mathcal F^{(i)}$ of index $\beta'=\beta/(1-\beta)$ with distances rescaled by 
$\left(C^{(i)}\right)^{-1}=\left({P}^{(i)}\right)^{\beta} \cdot \left({D}^{(i)}\right)^{\beta/(1-\beta)}
                                         =\lim_{n \rightarrow \infty}\big(P^{(i)}-\sum_{j\in[n]}P_j^{(i)}\big)^{\beta}(1-\beta)^{-\beta}n^{\beta^2/(1-\beta)}$. Specifically,  
the root of $\mathcal F^{(i)}$ is identified with $v_i$ and the subtrees rooted at $v_i$ are attached to leaves of $\mathcal F^{(i)}$ in the order of their appearance in 
Algorithm \ref{Ford}. Then the tree $\mathcal T^*$ obtained here in the limit after all replacements has the same distribution as the tree $\mathcal T^*$ in Theorem \ref{Mainresult2}.
\end{theorem}


We will formalise this construction in Section \ref{bprepl}. The random variable ${D}^{(i)}$ is the so-called \em $(1-\beta)$-diversity \em of the mass partition $({P}^{(i)}_j/{P}^{(i)}, j \geq 1) \sim {\rm GEM}(1-\beta, -\beta)$, where GEM$(\alpha, \theta)$ denotes the Griffiths-Engen-McCloskey distribution with parameters $\alpha\in[0,1)$, $\theta > -\alpha$, whose ranked version is the Poisson-Dirichlet distribution ${\rm PD}(\alpha,\theta)$. Note that, when $\beta=1/3$, we have $\beta'=1/2$, which means that we 
replace the branch points of the stable tree by rescaled i.i.d.\ Brownian CRTs. This should be compared with Le Gall \cite{Subord}, who effectively contracts 
subtrees in the middle of a Brownian CRT to obtain a stable tree of parameter 3/2. Neither his subtrees nor our $\mathcal{T}^*$ appear to be rescaled Brownian CRTs.

The proofs of Theorems \ref{Mainresult2} and \ref{branchrepl}, in particular the compactness of $\mathcal T^*$, are based on an embedding of 
$(\mathcal T_k^*, k \geq 0)$ into a compact CRT whose existence follows from earlier work \cite{RW} where we constructed CRTs via i.i.d.\ copies of a 
random string of beads, i.e.\ any random interval equipped with a random discrete probability measure, see Section \ref{Embedding} here for details. The distribution 
$\nu$ of the string of beads needed to obtain this compact CRT combines two $(\beta, \theta)$-strings of beads (for $\theta=\beta$ and $\theta=1-2\beta$), 
which arise in the framework 
of ordered $(\beta, \theta)$-Chinese restaurant processes as introduced in \cite{1}. A \em $(\beta, \theta)$-string of beads \em is 
an interval of length $K \sim {\rm{ML}}(\beta, \theta)$ equipped with a discrete probability measure whose atom sizes are PD$(\beta, \theta)$, arranged in a random 
order that yields a regenerative property. 
It is crucial for our argument to equip each reduced tree with a mass measure which effectively captures projected subtree masses. This naturally leads to a new line-breaking construction of the stable tree where the selection of the attachment point $J_k$ is based on masses rather than lengths, and where a proportion of 
the mass in $J_k$ is spread over the new branch, depending on the degree ${\rm deg}(J_k,\cT_k)$ of $J_k$ in $\cT_k$.

\begin{algorithm}[Line-breaking construction of the stable tree with masses] \label{masssel}
Let $\beta\in(0,1/2]$. We grow weighted $\bR$-trees $(\cT_k,\mu_k)$, $k\ge 0$, as follows.
  \begin{enumerate}\setcounter{enumi}{-1}\item Let $(\cT_0,\mu_0)$ be isometric to a $(\beta,\beta)$-string of beads.
  \end{enumerate}
  Given $(\cT_j,\mu_j)$ with $\mu_j=\sum_{x\in\cT_j}\mu_j(x)\delta_x$, $0\le j\le k$,    
  \begin{enumerate}\stepcounter{enumi}
    \item[1.-2.] sample $J_k$ from $\mu_k$; 
    \item[3.] given ${\rm deg}(J_k,\cT_k)=d\ge 2$, let $Q_k\sim{\rm Beta}(\beta,(d-2)(1-\beta)+1-2\beta)$, and let $\xi_k$ be an
      independent $(\beta,\beta)$-string of beads; to form $(\cT_{k+1},\mu_{k+1})$, remove $Q_k\mu_k(J_k)\delta_{J_k}$ from $\mu_k$ and attach
      to $\cT_k$ at $J_k$ an isometric copy of $\xi_k$ with measure rescaled by $Q_k\mu_k(J_k)$ and metric rescaled by $(Q_k\mu_k(J_k))^\beta$.
  \end{enumerate}
\end{algorithm}

\begin{theorem} \label{samestable} In Algorithm \ref{masssel}, $(\mathcal T_k, k \geq 0)$ has the same distribution as the sequence of trees in \eqref{findimmarg1} (and as in Algorithm \ref{GH}). In particular, $\lim_{k \rightarrow \infty} \mathcal T_k = \mathcal T$ a.s. in the Gromov-Hausdorff topology for a stable tree $\mathcal T$. Furthermore, $\lim_{k \rightarrow \infty} (\mathcal T_k, \mu_k)= (\mathcal T, \mu)$ a.s. in the Gromov-Hausdorff-Prokhorov topology. 
\end{theorem}
The proof of Theorem \ref{samestable} is based on the following well-known property of the stable tree, phrased in different terminology in \cite[Corollary 10(3)]{35},  and \cite[discussion after Corollary 8]{1}, where the link between $(\beta, \beta)$-strings of beads and a Bessel bridge of dimension $2\beta$ was established.

\begin{prop} \label{betastring} Let $(\mathcal T, \mu)$ be a stable tree of parameter $\beta \in (0,1/2]$, and let $\Sigma_0 \sim \mu$. Consider the \em spine \em  
  $\mathcal T_0=[[\rho,\Sigma_0]]$, and equip $\mathcal T_0$ with the mass measure $\mu_0$, capturing the masses of the connected components of 
  $\mathcal T \setminus \mathcal T_0$ projected onto $\mathcal T_0$. Then $(\mathcal T_0, \mu_0)$ is a $(\beta, \beta)$-string of beads.

\end{prop}

This paper is structured as follows. We introduce the framework of $\infty$-marked $\mathbb R$-trees in Section \ref{MMS}, and collect some preliminary results in 
Section \ref{prelim}. Section \ref{sec4} is devoted to the study of the two-colour line-breaking construction, while Section \ref{sec5} deals with its convergence to  
compact CRTs, as well as the branch point replacement. In Section \ref{Sec6}, we study a discrete two-colour tree-growth process whose two-step scaling 
limit is the two-colour CRT. An appendix includes some proofs postponed from earlier sections.

\section{$\mathbb R$-trees and marked metric spaces} \label{MMS}

\subsection{$\mathbb R$-trees and the Gromov-Hausdorff topology} \label{RT}

A compact metric space $(\mathcal T, d)$ is called an \textit{$\mathbb R$-tree} \cite{20,26} if for each $x,y \in \mathcal T$ the following holds.
\begin{enumerate}
\item[(i)] There is an isometry $f_{x,y} \colon [0, d(x,y)] \rightarrow \mathcal T$ such that $f_{x,y}(0)=x$ and $f_{x,y}(d(x,y))=y$.
\item[(ii)] For all injective paths $g \colon [0,1] \rightarrow \mathcal T$ with $g(0)=x$ and $g(1)=y$, we have $g([0,1])=f_{x,y}([0,d(x,y)])$.
\end{enumerate}
We denote the range of $f_{x,y}$ by $[[x,y]]:=f_{x,y}\left([0,d(x,y)]\right)$. 
All our $\mathbb R$-trees will be \em rooted \em at a distinguished element $\rho$, the \textit{root} of $\mathcal T$.
We call two $\mathbb R$-trees $(\mathcal T, d, \rho)$ and $(\mathcal T', d', \rho')$ 
\textit{equivalent} if there is an isometry from $\mathcal T$ to $\mathcal T'$ that maps $\rho$ onto $\rho'$. We denote by $\mathbb T$ the set of equivalence classes of 
rooted $\mathbb R$-trees, which we equip with the {Gromov-Hausdorff distance} $d_{\rm GH}$ \cite{21}  to obtain the Polish space $(\mathbb T, d_{\rm GH})$. The \textit{Gromov-Hausdorff distance} between two $\mathbb{R}$-trees $(\mathcal T, d, \rho)$ and $(\mathcal T',d', \rho')$ is defined as \begin{equation}  d_{\rm GH} \left(\left(\mathcal T, d,\rho\right), \left(\mathcal T', d',\rho'\right)\right) := \inf \limits_{\varphi, \varphi'} \left\{\max \left\{\delta \left(\varphi\left(\rho\right), \varphi'\left(\rho'\right)\right), \delta_{\rm H}\left(\varphi \left(\mathcal T\right), \varphi'\left(\mathcal T'\right)\right)\right\}\right\}, \label{GHdist} \end{equation}
where the infimum is taken over all metric spaces $(\mathcal M, \delta)$ and all isometric embeddings $\varphi\colon \mathcal T \rightarrow \mathcal M$, $\varphi' \colon \mathcal T' \rightarrow \mathcal M$ into the common metric space $(\mathcal M, \delta)$, and $\delta_{\rm H}$ is the Hausdorff distance between compact subsets of $(\mathcal M, \delta)$.
It is well-known that the Gromov-Hausdorff distance only depends on equivalence classes of rooted $\mathbb R$-trees, and we equip $\mathbb T$ with the Borel $\sigma$-algebra $\mathcal B(\mathbb T)$ induced by $d_{\rm GH}$. 

We can enhance a rooted $\mathbb R$-tree by considering a probability measure $\mu$ on its Borel sets $\mathcal B(\mathcal T)$, and call $(\mathcal T, d, \rho, \mu)$ a \textit{weighted} $\mathbb R$-tree. We call $(\mathcal T, d, \rho, \mu)$ and $(\mathcal T', d', \rho', \mu')$ \textit{equivalent} if there is an isometry from $\mathcal T$ to $\mathcal T'$ such that $\rho$ is mapped onto $\rho'$ and $\mu'$ is the push-forward of $\mu$ under this isometry. We let $\mathbb{T}_{\rm w}$ denote the set of equivalence classes of compact weighted $\mathbb R$-trees. Then $\mathbb{T}_{\rm w}$ is Polish when equipped with the \textit{Gromov-Hausdorff-Prokhorov distance} $d_{\rm GHP}$ induced by 
\begin{equation} d_{\rm GHP} \left(\left(\mathcal T, d,\rho, \mu\right), \left(\mathcal T', d',\rho', \mu'\right)\right) := \inf \limits_{\varphi, \varphi'} \left\{\max\left\{\delta\left(\varphi\left(\rho\right), \varphi'\left(\rho'\right)\right), \delta_{\rm H}\left(\varphi\left(\mathcal T\right), \varphi'\left(\mathcal T'\right)\right), \delta_{\text P}\left(\varphi_*\mu, \varphi'_*\mu'\right) \right\} \right\} \label{GHPdist} \end{equation}
for weighted $\mathbb{R}$-trees $(\mathcal T, d, \rho, \mu)$ and $(\mathcal T',d', \rho', \mu')$, where $\varphi, \varphi', \delta_{\text H}$ are as in \eqref{GHdist}, $\varphi_*\mu$, $\varphi'_*\mu$ are the push-forwards of $\mu$, $\mu'$ via $\varphi, \varphi'$, respectively, and $\delta_{\text P}$ is the Prokhorov distance on the space of Borel probability measures on $(\mathcal M, \delta)$ given by 
$\delta_{\rm P}\left(\mu, \mu' \right)=\inf \left\{ \epsilon >0\colon \mu(D) \leq \mu'( D^\epsilon)+\epsilon \quad \forall D \subset \mathcal M \text{ closed}\right\}$, 
where $D^\epsilon:=\{x \in \mathcal M\colon \inf_{y \in D} \delta(x,y) \leq\epsilon\}$ denotes the \textit{$\epsilon$-thickening} of $D$. 


While some of our developments are more easily stated in $(\mathbb T,d_{\rm GH})$ or $(\mathbb T_{\rm w},d_{\rm GHP})$, others benefit from more explicit
embeddings into a particular metric space $(\mathcal{M},\delta)$, which we will mostly choose as 
$$\mathcal{M}=l^1(\mathbb{N}_0^2):=\left\{(s_{i,j})_{i,j\in\mathbb N_0}\in[0,\infty)^{\mathbb N_0^2}\colon\sum_{i,j\in\mathbb N_0}s_{i,j}<\infty\right\}$$
with the metric induced by the $l^1$-norm. This is a variation of Aldous's \cite{8,7,6} choice $\mathcal{M}=l^1(\mathbb{N})$.
We denote by $\mathbb{T}^{\rm emb}$ the space of all compact $\mathbb{R}$-trees $\mathcal T\subset l^1(\mathbb{N}_0^2)$ with root $0\in\mathcal T$, which we equip with the Hausdorff metric $\delta_{\rm H}$, and by $\mathbb{T}^{\rm emb}_{\rm w}$ the space of all weighted compact $\mathbb{R}$-trees $(\mathcal T,\mu)$ with $\mathcal T\in\mathbb{T}^{\rm emb}$, which we equip with the metric
$\delta_{\rm HP}((\mathcal T,\mu),(\mathcal T^\prime,\mu^\prime))=\max\{\delta_{\rm H}(\mathcal T,\mathcal T^\prime),\delta_{\rm P}(\mu,\mu^\prime)\}$.

\begin{prop}[e.g. \cite{RW}]\label{embprop}\begin{enumerate}\item[\rm(i)] $(\mathbb{T}^{\rm emb},\delta_{\rm H})$ and $(\mathbb{T}^{\rm emb}_{\rm w},\delta_{\rm HP})$ are separable and complete.
  \item[\rm(ii)] For all $\mathcal T,\mathcal T^\prime\in\mathbb{T}^{\rm emb}$ we have $d_{\rm GH}(\mathcal T,\mathcal T^\prime)\le\delta_{\rm H}(\mathcal T,\mathcal T^\prime)$, and for all $(\mathcal T,\mu),(\mathcal T^\prime,\mu^\prime)\in\mathbb{T}^{\rm emb}_{\rm w}$, we have $d_{\rm GHP}((\mathcal T,\mu),(\mathcal T^\prime,\mu^\prime))\le\delta_{\rm HP}((\mathcal T,\mu),(\mathcal T^\prime,\mu^\prime))$. 
  \item[\rm(iii)] Every rooted compact $\mathbb{R}$-tree is equivalent to an element of $\mathbb{T}^{\rm emb}$, and every rooted weighted compact $\mathbb{R}$-tree is
    equivalent to an element of $\mathbb{T}^{\rm emb}_{\rm w}$.
  \item[\rm(iv)] For $\mathcal T_n\in\mathbb{T}^{\rm emb}$ with $\mathcal T_n\subseteq \mathcal T_{n+1}$, $n\ge 1$, and the closure $\mathcal T:=\overline{\bigcup_{n\ge 1}\mathcal T_n}$, we
    have $(\mathcal T_n,n\ge 1)$ convergent in $(\mathbb{T},d_{\rm GH})$ if and only if $\lim_{n\rightarrow\infty}\delta_{\rm H}(\mathcal{T}_n,\mathcal{T})=0$. In particular, in this
    case $\mathcal{T}$ is compact.
  \end{enumerate}  
\end{prop}  

For $\mathcal T\in\mathbb T^{\rm emb}$ and $c>0$, we define $c\mathcal T:=\{cx\colon x\in\mathcal T\}$. More generally for any $\mathbb{R}$-tree $(\mathcal T,d)$, we 
slightly abuse notation and denote by $c\mathcal T$ the metric space $(\mathcal T, cd)$ obtained when all distances are multiplied by $c$. We consider random 
$\mathbb R$-trees whose equivalence class in $\mathbb T$ has the distribution of a stable or Ford tree, and also refer to these trees as stable or Ford trees, and to the associated law on $\mathbb T$ as their \em distribution\em.

If $x \in \mathcal T\setminus \{\rho\}$ is such that $\mathcal T \setminus \{x\}$ is connected, we call $x$ a $\textit{leaf}$ of $\mathcal T$. A \textit{branch point} is 
an element $x \in \mathcal T$ such that $\mathcal T \setminus \{x\}$ has at least three connected components. We refer to the number of these components as the 
\textit{degree} {{deg}}$(x, \mathcal T)$ of $x$. We denote the 
sets of all leaves and branch points by ${\rm Lf}(\mathcal T)$ and ${\rm Br}(\mathcal T)$. If $\mathcal T \setminus {\rm Br}(\mathcal T)$
has only finitely many connected components, we call $\mathcal T$ a \em discrete \em $\mathbb{R}$-tree and these components (with or without one or both endpoints) \textit{edges}. We denote 
the set of edges by  Edg($\mathcal T$), and call  
$\#{\rm Lf}(\mathcal T)$ the \em size \em of $\mathcal T$. Also, $|\mathcal T|:=\#{\rm Edg}(\mathcal T)$. We call the discrete graph with edge set Edg($\mathcal T$) the \textit{shape} of $\mathcal T$. 

In the case of discrete weighted $\mathbb R$-trees it will often be of interest how the total mass of $1$ is distributed between the edges, with possibly some mass in
branch points, which for convenience we will also write in the form $E=\{v\}$. For any weighted $\mathbb R$-tree $(\mathcal T, \mu)$ with $n$ edges/branch points 
$E_1, \ldots, E_n$, the vector $(X_1, \ldots, X_n)$ with $X_i:=\mu(E_i)$, $i \in [n]$, is called the \textit{mass split} in $\mathcal T$. We will also 
consider mass splits in subtrees $\mathcal R\subset \mathcal T$, i.e.\ mass splits in $(\mathcal R, \mu(\mathcal R)^{-1} \mu \restriction_{\mathcal R})$. To distinguish 
mass splits in the \enquote{big} tree $\mathcal T$ and in \enquote{small} subtrees, we will speak of the \textit{total} and 
\textit{internal} (or \textit{relative}) mass splits, respectively. 

The limiting trees of the weighted $\mathbb R$-trees in our constructions will be \textit{continuum trees}, i.e.\ weighted $\mathbb R$-trees $(\mathcal T, d, \mu)$ such that the probability measure $\mu$ on $\mathcal T$ satisfies the following three properties.
(i) $\mu$ is supported by the set Lf($\mathcal T)$ of leaves of $\mathcal T$. (ii) $\mu$ is non-atomic, i.e. for any $x \in {\rm Lf}(\mathcal T)$, $\mu(x)=0$. 
(iii) For any $x \in \mathcal T \setminus {\rm Lf}(\mathcal T)$ and $\mathcal T_x := \{ \sigma \in \mathcal T\colon x \in [[\rho, \sigma]]\}$, we have $\mu(\mathcal T_x)>0$.

It is an immediate consequence of (i)-(iii) that, for any continuum tree $(\mathcal T, d)$, the set of leaves ${\rm Lf} (\mathcal T)$ is uncountable and that it has no isolated points. Finally, we introduce the notion of a \textit{reduced} subtree 
\begin{equation} \mathcal R(\mathcal T, x_1,\ldots,x_n):= \bigcup \limits_{i\in[n]} [[\rho, x_i]] 
\end{equation} 
of an $\mathbb R$-tree $\mathcal T$ spanned by the root and $x_1,x_2,\ldots,x_n \in {\rm Lf}(\mathcal T)$. Note that $\mathcal R(\mathcal T, x_1,\ldots,x_n)$ is a discrete $\mathbb R$-tree with root $\rho$ and leaves $x_1,\ldots,x_n$. We further consider the projection map
\begin{equation} \pi_k\colon \mathcal T \rightarrow \mathcal R(\mathcal T, x_1, \ldots, x_k) , \quad y \mapsto f_{\rho, y}\left(\sup\{t \geq 0\colon f_{\rho, y}(t) \in \mathcal R(\mathcal T,x_1,\ldots,x_k) \}\right),\label{pm}
\end{equation}
where $f_{\rho,y} \colon [0, d(\rho,y)] \rightarrow \mathcal T$ is the unique isometry  with $f_{\rho,y}(0)=\rho$ and $f_{\rho,y}(d(\rho,y))=y$ from the definition of an $\mathbb R$-tree. The push-forward of a probability measure $\mu$ on $\mathcal T$ via this projection map is denoted by $(\pi_k)_*\mu$, i.e. 
\begin{equation} (\pi_k)_*\mu\left(D\right)=\mu\left(\pi_k^{-1}\left(D\right)\right), \qquad D \subset \mathcal R(\mathcal T, x_1, \ldots, x_k) \text{ Borel measurable}.  \label{pfpm}
\end{equation}
More details on $\mathbb R$-trees and proofs for the statements made in this section can be found in  \cite{20,25,26}.

\subsection{$\infty$-marked $\mathbb R$-trees} \label{IMRT}

We introduce $\infty$-marked $\mathbb R$-trees to capture the framework of an $\mathbb R$-tree with infinitely many marked components. This is a generalisation of Miermont's concept of a $k$-marked metric space, \cite[Section 6.4]{22}. In the context of the two-colour line-breaking construction, the marked components correspond to the rescaled Ford trees by which we replace the branch points in the stable line-breaking construction. Each Ford tree, i.e. each connected red component, is related to a new marked subset of the $\infty$-marked $\mathbb R$-tree.

A \textit{k-marked} $\mathbb R$-tree $(\mathcal T,d, \rho, (\mathcal R^{(1)}, \ldots, \mathcal R^{(k)}))$, $k \geq 1$, is a rooted $\mathbb R$-tree $(\mathcal T, d, \rho)$ with non-empty closed connected subsets $\mathcal R^{(1)}, \ldots, \mathcal R^{(k)} \subset \mathcal T$. 
We call two $k$-marked $\mathbb R$-trees $(\mathcal T,d, \rho, (\mathcal R^{(1)}, \ldots, \mathcal R^{(k)}))$ and $(\mathcal T^\prime,d^\prime, \rho, (\mathcal R'^{(1)}, \ldots, \mathcal R^{\prime(k)}) )$ \textit{equivalent} if there exists an isometry from $\mathcal T$ to $\mathcal T^\prime$ such that each $\mathcal R^{(i)}$ is mapped onto $\mathcal R^{\prime(i)}$, $i \in [k]$, respectively, and $\rho$ is mapped onto $\rho'$. If $\mathcal T$ and $\mathcal T^\prime$ are equipped with mass measures $\mu$ and $\mu^\prime$, we speak of \textit{weighted} $k$-marked $\mathbb R$-trees, and we call them \textit{equivalent} if there is an isometry from $\mathcal T$ to $\mathcal T^\prime$ such that each $\mathcal R^{(i)}$ is mapped onto $\mathcal R^{\prime(i)}$, $i \in [k]$, $\rho$ is mapped to $\rho'$ and $\mu^\prime$ is the push-forward of $\mu$ under this isometry. The set of equivalence classes of $k$-marked $\mathbb R$-trees is denoted by $\mathbb T^{[k]}$, and $\mathbb T_{\rm w}^{[k]}$ is the set of equivalence classes of weighted $k$-marked $\mathbb R$-trees.

For $k$-marked $\mathbb{R}$-trees $(\mathcal T,d, \rho, (\mathcal R^{(1)}, \ldots, \mathcal R^{(k)}))$, $(\mathcal T^\prime,d^\prime, \rho^\prime, (\mathcal R^{\prime(1)}, \ldots, \mathcal R^{\prime(k)})) \in \mathbb T^{[k]}$, define
\begin{align} d_{\rm{GH}}^{[k]} &\left( \left(\mathcal T,d, \rho, \left(\mathcal R^{(1)}, \ldots, \mathcal R^{(k)} \right)\right), \left(\mathcal T^\prime, d^\prime, \rho^\prime, \left(\mathcal R^{\prime(1)}, \ldots, \mathcal R^{\prime(k)}\right) \right) \right) \nonumber \\
&:=\inf_{\varphi, \varphi^\prime} \left\{  \max \left\{ \delta_{\rm{H}} \left( \varphi\left(\mathcal T\right), \varphi^\prime\left(\mathcal T^\prime\right) \right), \max \limits_{1 \leq i \leq k} \delta_{\rm H} \left( \varphi \left(\mathcal R^{(i)}\right), \varphi'\left( \mathcal R^{\prime(i)} \right)\right), \delta_{} \left( \varphi\left(\rho\right), \varphi \left(\rho^\prime\right) \right)  \right\} \right\} \label{GHK} \end{align} where the infimum is taken over all isometric embeddings $\varphi$, $\varphi^\prime$ of $\mathcal T$, $\mathcal T^\prime$ into a common metric space $(\mathcal M, \delta)$, and $\delta_{\text H}$ is the Hausdorff distance on $(\mathcal M, \delta)$. It was shown in \cite{22} that $d^{[k]}_{\rm{GH}}$ is a metric on $\mathbb T^{[k]}$.

\begin{lemma}[{\cite[Proposition 9(ii)]{22}}] \label{sepcomp} The space $(\mathbb T^{[k]}, d_{\rm GH}^{[k]})$ is separable and complete. \end{lemma}\pagebreak

We extend the notion of a $k$-marked $\mathbb R$-tree to an \textit{$\infty$-marked} $\mathbb R$-tree $(\mathcal T, d, \rho, (\mathcal R^{(i)}, i \geq 1))$. 
The marked components $\mathcal R^{(i)}, i \geq 1$, of an $\infty$-marked $\mathbb R$-tree $(\mathcal T, ( \mathcal R^{(i)}, i \geq 1))$ are themselves $\mathbb R$-trees when equipped with the metric restricted to $\mathcal R^{(i)}$, and rooted at the point of $\mathcal R^{(i)}$ closest to the root of $\mathcal T$, $i\geq 1$.  
We will consider $\infty$-marked $\mathbb R$-trees $(\mathcal T, d, \rho, (\mathcal R^{(i)}, i \geq 1))$ with a discrete branching structure, and distinguish between \textit{internal} and \textit{external} edges of $\mathcal R^{(i)}$. 
External edges of $\mathcal R^{(i)}$ are edges connecting a branch point/root and a leaf of $\mathcal R^{(i)}$, internal edges connect two branch points or the root and a branch point. 

 As in the $k$-marked case, \textit{$\infty$-marked} $\mathbb R$-trees $(\mathcal T, d, \rho, (\mathcal R^{(i)}, i \geq 1))$, $(\mathcal T^\prime, d^\prime, \rho^\prime, (\mathcal R^{\prime(i)}, i \geq 1))$ are \textit{equivalent} if there is an isometry from $\mathcal T$ to $\mathcal T^\prime$ such that $\rho$ is mapped onto $\rho^\prime$, and each $\mathcal R^{(i)}$ is mapped onto $\mathcal R^{\prime(i)}$, $i \geq 1$, respectively. We write $\mathbb T^{\infty}$ for the set of equivalence classes of compact $\infty$-marked $\mathbb R$-trees, and equip it with the metric $d^{\infty}_{\rm GH}:=\sum_{k \geq 1} 2^{-k}d_{\rm GH}^{[k]}$, i.e. for $(\mathcal T, d, \rho, (\mathcal R^{(i)}, i \geq 1))$, $(\mathcal T^\prime, d^\prime, \rho^\prime, (\mathcal R^{\prime(i)}, i \geq 1)) \in \mathbb T^{\infty}$,
\begin{align}
d^{\infty}_{\rm GH}&\left(\left(\mathcal T, d, \rho,\left(\mathcal R^{(i)},i \geq 1\right)\right),\left(\mathcal T', d^\prime, \rho^\prime, \left(\mathcal R'^{(i)}, i \geq 1 \right)\right) \right)\nonumber \\ & \hspace{3cm}:= \sum_{k \geq 1} 2^{-k} d_{\rm GH}^{[k]} \left(\left(\mathcal T, \left(\mathcal R^{(i)}, \ldots, \mathcal R^{(k)} \right)\right), \left(\mathcal T', \left(\mathcal R'^{(1)}, \ldots, \mathcal R'^{(k)}\right)\right)\right). \label{GHKmarked}
\end{align}

\begin{corollary} \label{33} The space $(\mathbb T^{\infty}, d_{\rm GH}^{\infty})$ is separable and complete. \end{corollary}

\begin{proof}
By Lemma \ref{sepcomp}, for any $k \geq 1$, there exists a
countable dense subset $\mathbb C_k \subset \mathbb T_k$ such that for any $\epsilon >0$ and any $(\mathcal T, d, \rho, (\mathcal R^{(1)}, \ldots, \mathcal R^{(k)})) \in \mathbb T^{[k]}$, there is $(\mathcal T', d', \rho', (\mathcal R'^{(1)}, \ldots, \mathcal R'^{(k)}))\in \mathbb C_k$ with
\begin{equation}
d_{\rm GH}^{[k]}\left(\left(\mathcal T, \left(\mathcal R^{(1)}, \ldots, \mathcal R^{(k)}\right)\right), \left(\mathcal T', \left(\mathcal R'^{(1)}, \ldots, \mathcal R'^{(k)}\right)\right)\right) < \epsilon. \label{sepink}
\end{equation}
For any $k \geq 1$, the set 
$\mathbb C_k':=\left\{ \left(\mathcal T', \left(\mathcal R^{(1')}, \ldots, \mathcal R^{(k')}, \{\rho'\}, \{\rho'\}, \ldots \right)\right)\colon \left(\mathcal T', \left(\mathcal R'^{(1)}, \ldots, \mathcal R'^{(k)}\right)\right) \in \mathbb C_k\right\}$ 
is countable. For any $(\mathcal T, (\mathcal R^{(1)}, \mathcal R^{(2)}, \ldots)) \in \mathbb T^{\infty}$ there is $(\mathcal T', (\mathcal R'^{(1)}, \ldots, \mathcal R'^{(k)}, \{\rho'\}, \{\rho'\}, \ldots)) \in \mathbb C_k'$ such that \eqref{sepink} holds when we restrict to the first $k$ marked components, i.e. for ${\rm diam}(\mathcal T)=\sup\{d(x,y)\colon x,y\in\mathcal T\}$, 
\begin{equation*}
d_{\rm GH}^{\infty}\left(\left(\mathcal T, \left(\mathcal R^{(i)}, i \geq 1 \right)\right), \left(\mathcal T', \left(\mathcal R'^{(i)}, i \geq 1 \right)\right)\right) \leq 2^{-k} \epsilon \left( 1+2+ \cdots +2^{k-1}\right)+ \sum_{n \geq k+1} 2^{-n} {\rm diam}\left(\mathcal T\right) < \epsilon  \end{equation*}
for $k$ large enough. Therefore, $\mathbb C_\infty:={\bigcup_{k \geq 1} \mathbb C_k'}$ is countable and dense in $\mathbb T^\infty$, i.e. $(\mathbb T^\infty,d_{\rm GH}^{\infty})$ is separable.

To see that $(\mathbb T^\infty,d_{\text GH}^{\infty})$ is complete, consider a Cauchy sequence $(\mathcal T_n, (\mathcal R_n^{(i)}, i \geq 1), n \geq 1)$ in $\mathbb T^\infty$. By definition of $d_{\rm GH}^{\infty}$, for any $k \geq 1$, the sequence $(\mathcal T_n, (\mathcal R_n^{(1)}, \ldots, \mathcal R_n^{(k)}), n \geq 1)$ is Cauchy in $\mathbb T^{[k]}$. By Lemma \ref{sepcomp}, $(\mathbb T^{[k]}, d_{\rm GH}^{[k]})$ is complete, and we conclude that there is $(\mathcal T, (\mathcal R^{(1)}, \ldots, \mathcal R^{(k)})) \in \mathbb T^{[k]}$ such that 
\begin{equation}
\lim_{n \rightarrow \infty} \left(\mathcal T, \left(\mathcal R_n^{{(1)}}, \ldots, \mathcal R_n^{(k)}\right)\right) = \left(\mathcal T, \left(\mathcal R^{(1)}, \ldots, \mathcal R^{(k)}\right)\right). \label{kvaries}
\end{equation}
Furthermore, the limits \eqref{kvaries} are easily seen to be consistent as $k$ varies, i.e. given $\eqref{kvaries}$ for some $k \geq 2$,
\begin{equation}
\lim_{n \rightarrow \infty} \left(\mathcal T, \left(\mathcal R_n^{(1)}, \ldots, \mathcal R_n^{(k-1)}\right)\right) = \left(\mathcal T, \left(\mathcal R^{(1)}, \ldots, \mathcal R^{(k-1)}\right)\right) 
\end{equation}
in $(\mathbb T^{[k-1]}, d_{\rm GH}^{[k-1]})$. We conclude that there is $(\mathcal T, (\mathcal R^{(1)}, \mathcal R^{(2)}, \ldots))$ such that
\begin{equation*}
\lim \limits_{n \rightarrow \infty} d_{\rm GH}^{\infty}\left(\left(\mathcal T_n, \left(\mathcal R_n^{(i)}, i \geq 1\right)\right),\left(\mathcal T, \left(\mathcal R^{(i)}, i \geq 1\right)\right)\right)=0.\vspace{-0.5cm}
\end{equation*}
\end{proof}

We can extend $d_{\rm GH}^{[k]}$ to a metric on $\mathbb T_{\rm w}^{[k]}$ by adding a Prokhorov component to $d_{\rm GH}^{[k]}$. For any $k \in \{1,2, \ldots\}$ and $(\mathcal T, (\mathcal R^{(1)}, \ldots, \mathcal R^{(k)}), \mu)$, $(\mathcal T', (\mathcal R'^{(1)}, \ldots, \mathcal R'^{(k)}), \mu') \in \mathbb T^{[k]}$, we define
\begin{align*} d_{\rm GHP}^{[k]}&\left( \left(\mathcal T, \left(\mathcal R^{(1)}, \ldots, \mathcal R^{(k)}\right), \mu\right), \left(\mathcal T', \left(\mathcal R'^{(1)}, \ldots, \mathcal R'^{(k)}\right), \mu'\right)\right) \\ & \hspace{-0.8cm} :=\inf_{\varphi, \varphi'} \left\{  \max \left\{ \delta_{\rm{H}} \left( \varphi\left(\mathcal T\right), \varphi'\left(\mathcal T'\right) \right), \max \limits_{1 \leq i \leq k} \delta_{\rm H} \left( \varphi \left(\mathcal R^{(i)}\right), \varphi'\left( \mathcal R'^{(i)} \right)\right), \delta_{} \left( \varphi\left(\rho\right), \varphi' \left(\rho'\right) \right), \delta_{\rm P} \left( \varphi_*\mu, \varphi'_*\mu'\right) \right\} \right\} \end{align*}
where $\varphi, \varphi'$ and $\varphi_*\mu$, $\varphi'_*\mu'$ are as in \eqref{GHPdist} and \eqref{GHK}. In the spirit of \eqref{GHKmarked}, we define
\begin{align}
d_{\rm GHP}^\infty & \left( \left( \mathcal T, \left( \mathcal R^{(i)},i \geq 1 \right), \mu \right), \left( \mathcal T', \left(\mathcal R'^{(i)}, i \geq 1 \right), \mu' \right) \right) \nonumber \\ &\hspace{2cm}= \sum_{k \geq 1} 2^{-k} d_{\rm GHP}^{[k]} \left( \left(\mathcal T, \left(\mathcal R^{(1)}, \ldots, \mathcal R^{(k)}\right), \mu\right), \left(\mathcal T', \left(\mathcal R'^{(1)}, \ldots, \mathcal R'^{(k)}\right), \mu'\right) \right) \end{align}
for two weighted $\infty$-marked $\mathbb R$-trees $(\mathcal T, (\mathcal R^{(i)},i \geq 1), \mu)$ and  $( \mathcal T', (\mathcal R'^{(i)}, i \geq 1), \mu')$.

\begin{lemma} The function $d_{\rm GHP}^{[k]}$ defines a distance on $\mathbb T_{\rm w}^{[k]}$, and the space $(\mathbb T_{\rm w}^{[k]}, d_{\rm GHP}^{[k]})$ is separable and complete, for any $k \in \{0,1,2, \ldots; \infty\}$.
\end{lemma}

\begin{proof} For $k\in\{0,1,2,\ldots\}$, the proof is a direct generalisation of the proof of Lemma \ref{sepcomp}. In particular, it is straightforward to generalise the results about $d_{\rm GHP}$ in \cite[Section 6.2/6.3]{22} to $d_{\rm GHP}^{[k]}$.
For $k=\infty$, the claim can then be deduced as in the proof of Corollary \ref{33}. \end{proof}

\begin{remark} Miermont \cite{22} introduced the more general concept of a $k$-marked metric space, and studied the space $\mathbb M^{[k]}$ of equivalence classes of $k$-marked metric spaces. $\mathbb T^{[k]}$ is a closed subset of $\mathbb M^{[k+1]}$ (\cite[Lemma 2.1]{21}), i.e. the results on $(\mathbb T^{[k]}, d_{\rm GH}^{[k]})$ presented here follow from his study of $(\mathbb M^{[k]}, d_{\rm GH}^{[k]})$, $k\ge 0$.
\end{remark}

\section{Mittag-Leffler distributions, strings of beads and stable trees} \label{prelim}

\subsection{Dirichlet and Mittag-Leffler distributions} \label{secml}

In this section, we present the distributional relationships that are key for our constructions. A random variable $L$ follows a (generalised) Mittag-Leffler distribution with parameters $(\alpha, \theta)$ for $\alpha >0$ and $\theta > -\alpha$ if its $p$th moment is given by\vspace{-0.1cm}
\begin{equation}
\mathbb E\left[ L^p\right]=\frac{\Gamma(\theta+1)\Gamma(\theta/\alpha+1+p)}{\Gamma(\theta/\alpha+1)\Gamma(\theta+p\alpha+1)}, \quad p\geq 1, \label{mlmoms}
\end{equation}  
for short $L \sim \rm{ML}(\alpha, \theta)$. The moments \eqref{mlmoms} uniquely characterise ML$(\alpha, \theta)$, cf. \cite{5}. 
%
%

The Mittag-Leffler distribution naturally appears when we study lengths in the trees considered in this paper. To analyse mass and length splits across the branches of these trees we have to consider Dirichlet distributions. We will be able to relate mass and length splits on the edges using the following result.

\begin{prop}[\cite{12} Proposition 4.2] \label{DIRML} Let $\beta \in (0,1)$. For $n \geq 2$, let $\theta_1, \ldots, \theta_n >0$ and let $\theta:=\sum_{i\in[n]} \theta_i$. Consider $S \sim {\rm{ML}}(\beta, \theta)$ and an independent vector $(Y_1, \ldots, Y_n) \sim {\rm{Dirichlet}}(\theta_1/\beta, \ldots, \theta_n/\beta)$. Then,
\begin{equation}
S \cdot \left(Y_1, \ldots, Y_n\right)\,{\buildrel d \over =}\, \left(X_1^{\beta}S^{(1)}, \ldots, X_n^{\beta}S^{(n)}\right)
\end{equation}
where $(X_1, \ldots, X_n) \sim {\rm{Dirichlet}}(\theta_1, \ldots, \theta_n)$ and $S^{(i)} \sim {\rm{ML}}(\beta, \theta_i)$, $i \in [n]$, are independent. 
\end{prop}

We will also need some standard properties of the Dirichlet distribution.

\begin{prop} \label{Diri} Let $n \in \mathbb N$, $\theta_1, \ldots, \theta_n>0$ and $X:=(X_1, \ldots, X_n) \sim {\rm{Dirichlet}}(\theta_1, \ldots, \theta_n)$. 
\begin{enumerate}

\item[{\rm (i)}] \textit{Symmetry}. For any permutation $\sigma\colon [n] \rightarrow [n]$, 
$\left(X_{\sigma(1)}, \ldots, X_{\sigma(n)}\right) \sim {\rm{Dirichlet}}\left(\theta_{\sigma(1)}, \ldots, \theta_{\sigma(n)}\right)$.

\item[{\rm (ii)}] \textit{Aggregation and deletion}. Let $X':=(\sum_{i\in[m]}X_i, X_{m+1},\ldots, X_n)$ for some $m \in [n-1]$. Then the vectors 
$
X' \sim {\rm{Dirichlet}}\big(\sum_{i\in[m]}\theta_i, \theta_{m+1}, \ldots, \theta_n \big)
$
and $X^*:=( {X_1}/{\sum_{i\in[m]} X_i}, \ldots, {X_m}/{\sum_{i\in[m]} X_i}) \sim {\rm Dirichlet}\left(\theta_1, \ldots, \theta_m\right)$ are independent.

\item[{\rm (iii)}] \textit{Decimation}. Let $i \in [n]$, $m \in \mathbb N$, and let $\theta_{i,1}, \ldots, \theta_{i,m}>0$ be such that $\sum_{j\in[m]} \theta_{i,j}=\theta_i$. Consider an independent random vector 
$
\left(P_1, \ldots, P_m\right) \sim {\rm{Dirichlet}}\left(\theta_{i,1}, \ldots, \theta_{i,m}\right).
$ 
Then we have $X'':=(X_1, \ldots, X_{i-1}, P_1X_i, \ldots, P_mX_i, X_{i+1}, \ldots, X_n)\sim {\rm{Dirichlet}}\left(\theta_1, \ldots, \theta_{i-1}, \theta_{i,1}, \ldots, \theta_{i,m}, \theta_{i+1}, \ldots, \theta_n\right).
$

\item[{\rm (iv)}] \textit{Size-bias}. Let $I \in [n]$ be a random index such that $\mathbb P(I=i|X_1, \ldots, X_n)=X_i$ a.s.\ for $i \in [n]$. Then, conditionally given $I=i$, we have 
$X \sim {\rm{Dirichlet}}\left(\theta_1, \ldots, \theta_{i-1}, \theta_{i}+1, \theta_{i+1},\ldots, \theta_n\right)$ for any $i \in [n]$.
Furthermore, we have $\mathbb P(I=i)=\theta_i\big/\sum_{j\in[n]} \theta_j$.

\end{enumerate}
\end{prop}
\begin{proof} We refer to \cite[Propositions 13-14, Remark 15]{29}, and the Gamma variable representation for the Dirichlet distribution.
\end{proof}


\subsection{Chinese restaurant processes and strings of beads} \label{Sec22}

We consider $(\alpha, \theta)$-strings of beads for $\alpha \in (0,1), \theta >0$, arising in the scaling limit of ordered $(\alpha, \theta)$-Chinese restaurant processes (CRPs), cf. \cite{james2006,1,5}.
Consider customers labelled by $[n]:=\{1,\ldots,n\}$ sitting at a random number of tables as follows. Let customer $1$ sit at the first table. At step $n+1$, conditionally given that we have $k$ tables with $n_1, \ldots, n_k$ customers, the next customer labelled by $n+1$ 
\begin{itemize}
\item sits at the $i$th occupied table with probability $(n_i-\alpha)/(n+\theta)$, $i \in [k]$;
\item opens a new table to the left of the first table, or between any two tables with probability $\alpha/(k\alpha+\theta)$;
\item opens a new table to the right of the last table with probability $\theta/(k\alpha+\theta)$.
\end{itemize}
%
This induces the \textit{ordered} $(\alpha, \theta)$-CRP $(\widetilde{\Pi}_n, n\geq 1)$. The classical \em unordered $(\alpha,\theta)$-CRP \em $(\Pi_n, n \geq 1)$ is 
obtained from $(\widetilde{\Pi}_n, n\geq 1)$ by ordering the  \pagebreak blocks by least labels. For $n \in \mathbb N$, we write 
$\Pi_n=(\Pi_{n,1}, \ldots, \Pi_{n,K_n})$ and $\widetilde{\Pi}_n=(\widetilde{\Pi}_{n,1}, \ldots, \widetilde{\Pi}_{n,K_n})$
for the blocks of the two partitions of $[n]$, where $K_n$ denotes the number of tables at step $n$. The block sizes at step $n$ form 
random \textit{compositions} of $n$, $n \geq 1$, i.e.\ sequences of positive integers $(n_1,\ldots,n_k)$ with sum $n=\sum_{j\in[k]} n_j$. The composition related to $\widetilde{\Pi}_n$, $n \geq 1$, can be shown to be \textit{regenerative} in the sense of Gnedin and Pitman \cite{3}.

The number of tables $K_n$ at step $n$, rescaled by $n^\alpha$, converges a.s., i.e. there is $L_{\alpha,\theta} > 0$ a.s. such that \begin{equation} L_{\alpha,\theta} = \lim \limits_{n \rightarrow \infty}{n^{-\alpha}}{K_n} \quad \text{a.s.}. \label{tables} \end{equation} The distribution of $L_{\alpha, \theta}$ can be identified as $\rm{ML}(\alpha, \theta)$. Furthermore, there are limiting proportions $(P_1, P_2, \ldots)$ of the relative table sizes $n^{-1}\#\Pi_{n,i}$, $i \in [K_n]$, as $n \rightarrow \infty$ in order of least labels, i.e. 
\begin{equation} \lim_{n \rightarrow \infty}\left( n^{-1} \#\Pi_{n,1}, \ldots, n^{-1} \#\Pi_{n,K_n} \right) =\left(P_1, P_2, P_3, \ldots \right) = \left(V_1, \overline{V}_1 V_2, \overline{V}_1 \overline{V}_2 V_3, \ldots\right) \quad \text{a.s.} \label{v1v2} \end{equation} where $(V_i, i \geq 1)$ are independent with $V_i \sim \text{Beta}(1-\alpha, \theta+i\alpha)$, and $\overline{V}_i:=1-V_i$. The distribution of the vector $(P_1, P_2, \ldots)$ is a \textit{Griffiths-Engen-McCloskey distribution} GEM$(\alpha, \theta)$. Ranking $(P_i, i \geq 1)$ in decreasing order we obtain a Poisson-Dirichlet distribution $(P_i^{\downarrow}, i \geq 1):=(P_i, i \geq 1)^{\downarrow} \sim \text{PD}(\alpha, \theta)$. Each $P_i, i \geq 1,$ is further associated with a position on the limiting interval $[0,L_{\alpha,\theta}]$ induced by the table order. 

\begin{lemma}[{\cite[Proposition 6]{1}}]\label{cvgcs}
Consider an \textit{ordered} $(\alpha, \theta)$-CRP $(\widetilde{\Pi}_n=(\widetilde{\Pi}_{n,1},\ldots, \widetilde{\Pi}_{n,K_n}), n\geq 1)$ for $\alpha \in (0,1)$, $\theta >0$. Let $N_{n,j}:=\sum _{i\in[j]} \# \widetilde{\Pi}_{n,i}$, $j \in [n]$, be the number of customers at the first $j$ tables from the left. Then,
\begin{equation}
\lim \limits_{n \rightarrow \infty} \left\{n^{-1}{N_{n,j}}, j \geq 0\right\}= \mathcal N_{\alpha, \theta}:=\left\{1-e^{-G_t}, t \geq 0\right\}^{\text{\rm cl}} \quad \text{a.s.} \label{lefttables}
\end{equation}
with respect to the Hausdorff metric on closed subsets of $[0,1]$, where ${\rm cl}$ denotes the closure in $[0,1]$, and $(G_t, t\geq 0)$ is a subordinator with Laplace exponent $\Phi_{\alpha, \theta}(s)={s\Gamma(s+\theta)\Gamma(1-\alpha)}/{\Gamma(s+\theta+1-\alpha)}.$

There is a continuous local time process $\mathcal L=(\mathcal L(u), u \in [0,1])$ for $\mathcal L_n(u):=\#\{j \in [K_n]\colon n^{-1} N_{n,j} \leq u\}$, $u \in [0,1]$, such that
\begin{equation*}
\lim \limits_{n \rightarrow \infty} \sup \limits_{u \in [0,1]} \lvert n^{-\alpha} \mathcal L_n(u)- \mathcal L(u)\rvert =0 \qquad \text{a.s.}\end{equation*}
where $\mathcal N_{\alpha, \theta}$ is the set of points at which $\mathcal L$ increases a.s..
\end{lemma}

We refer to the collection of open intervals in $[0,1]\setminus \mathcal N_{\alpha, \theta}$ as the \textit{($\alpha,\theta$)-regenerative interval partition} associated with the local time process $\mathcal L$, where $\mathcal L(1)=L_{\alpha,\theta}$ a.s.. Note that the joint law of ranked lengths of components of this interval partition is PD$(\alpha,\theta)$.
The inverse local time $\mathcal L^{-1}$ defined by
\begin{equation}
\mathcal L^{-1}:[0,L_{\alpha, \theta}) \rightarrow [0,1), \qquad \mathcal L^{-1}(x):=\inf\{u \in [0,1]\colon \mathcal L(u) > x\},
\end{equation}
is right-continuous increasing. We equip the random interval $[0, L_{\alpha, \theta}]$ with the Stieltjes measure $d \mathcal L^{-1}$.

\begin{definition}[String of beads]\rm
A \textit{string of beads} $(I, \lambda)$ is an interval $I$ equipped with a discrete mass measure $\lambda$. A measure-preserving isometric copy of $([0,L_{\alpha, \theta}], d \mathcal L^{-1})$ associated as above with an $(\alpha, \theta)$-regenerative interval partition $[0,1] \setminus \mathcal N_{\alpha, \theta}$ is called
an \textit{$(\alpha, \theta)$-string of beads}, for $\alpha \in (0,1), \theta > 0$.
\end{definition}

We can view a string of beads $([0,K], \lambda)$ as a weighted $\mathbb R$-tree consisting of one single branch connecting the root $0$ with a leaf at distance $K$. 

Since the lengths of the interval components of an $(\alpha, \theta)$-regenerative interval partition $[0,1] \setminus \mathcal N_{\alpha, \theta}$ are the masses of the atoms of the associated $(\alpha, \theta)$-string of beads, we conclude that the joint law of the masses $(P_i^\downarrow, i \geq 1)$ of the atoms of an $(\alpha, \theta)$-string of beads ranked in decreasing order is PD$(\alpha, \theta)$. It is well-known that the length $L_{\alpha, \theta}\sim \text{ML}(\alpha, \theta)$ of an $(\alpha, \theta)$-string of beads can be recovered from the ranked atom masses $(P_i^{\downarrow}, i \geq 1)$ or the vector $(P_i, i \geq 1)$ of the stick-breaking representation \eqref{v1v2} via
\begin{equation} L_{\alpha, \theta}=\lim \limits_{i \rightarrow \infty} i \Gamma(1-\alpha) (P_i^{\downarrow})^{\alpha}=\lim_{k \rightarrow \infty} \left(1-\sum_{i\in[k]} P_i\right)^{\alpha} \alpha^{-\alpha}k^{1-\alpha}, \label{alphadiv}
\end{equation}
which is the so-called \textit{$\alpha$-diversity} of $(P_i^\downarrow, i \geq 1) \sim {\rm{PD}}(\alpha, \theta)$, cf. \cite[Lemma 3.11]{5}.

One of the key properties of $(\alpha, \theta)$-strings of beads is the regenerative nature inherited from the underlying regenerative interval partition, cf. \cite{3}. Pitman and Winkel \cite{1} developed a method (``$(\alpha,\theta)$-coin-tossing sampling'') to sample an atom of an $(\alpha, \theta)$-string of beads such that the two strings of beads obtained in this way are rescaled independent $(\alpha, \alpha)$- and $(\alpha, \theta)$-strings of beads (the first one being the one closer to the origin). The mass split between the two induced interval components and the selected atom is Dirichlet$(\alpha, 1-\alpha, \theta)$, with parameters assigned in their order on the interval $[0, L_{\alpha, \theta}]$. When $\theta=\alpha$, the special sampling reduces to uniform sampling from the mass measure $d\mathcal L^{-1}$.

\begin{prop}[{\cite[Proposition 10/14(b), Corollary 15]{1}}] \label{cointoss} 
Let $(I, \lambda):=([0, L_{\alpha, \theta}], d \mathcal L^{-1})$ be an $(\alpha, \theta)$-string of beads for some $\alpha \in (0,1), \theta >0$.
Then there is a random variable $J\in(0,L_{\alpha,\theta})$ on a suitably enlarged probability space such that the following are independent.
\begin{itemize}
\item The mass split $( \lambda([0, J)), \lambda(J), \lambda((J, L_{\alpha, \theta}])) \sim \text{\rm Dirichlet}(\alpha, 1- \alpha, \theta)$;
\item (the isometry class of) the $(\alpha, \alpha)$-string of beads $( \lambda([0,J))^{-\alpha} [0,J), \lambda([0,J))^{-1} \lambda \restriction_{[0,J)} )$; 
\item (the isometry class of) the $(\alpha, \theta)$-string of beads $( \lambda((J, L_{\alpha, \theta}])^{-\alpha} (J, L_{\alpha, \theta}], \lambda((J, L_{\alpha, \theta}])^{-1} \lambda \restriction_{(J, L_{\alpha, \theta}]} ).$
\end{itemize}
\end{prop}


In Section \ref{sec4} we will formulate the algorithms of the introduction based on masses rather than lengths. In particular, the attachment points in the update step will be mass-sampled, not length-sampled. The following lemma will imply that that the algorithms based on masses induce the length versions.

\begin{lemma} \label{equivalence} Let $(X_1, \ldots, X_n) \sim {\rm Dirichlet}(\theta_1, \ldots, \theta_n)$ for some $\theta_1, \ldots, \theta_n >0$ and $n \in \mathbb N$, and let $([0,L_i], \lambda_i)$ be independent $(\alpha, \theta_i)$-strings of beads, respectively, $i\in[n]$.

\begin{itemize}
\item Select $I'=j \in [n]$ with probability $X_j$ and, conditionally given $I'=j$, select $L' \in [0, L_j]$ via $(\alpha, \theta_j)$-coin tossing sampling on $([0,L_j], \lambda_j)$.
\item Select $I''=j \in [n]$ with probability proportional to $X_j^\alpha L_j$ and, conditionally given $I''=j$, select $L''=B L_j$ where $B \sim {\rm Beta}(1, \theta_j/\alpha)$ is independent.
\end{itemize}
Then $
\left(I',L_1, \ldots, L_{I'-1}, L',  L_{I'}-L', L_{I'+1}, \ldots L_n\right) \,{\buildrel d \over =}\, \left(I'',L_1, \ldots, L_{I''-1}, L'', L_{I''}-L'', L_{I''+1}, \ldots L_n\right).
$
\end{lemma}

\begin{proof} We need to show that, for any bounded and continuous function $f\colon \mathbb R^{n+2} \rightarrow \mathbb R$
\begin{equation}
\mathbb E \left[ f \left(I',L_1, \ldots, L_{I'-1}, L',  L_{I'}\!-\!L', L_{I'+1}, \ldots, L_n \right)\! \right] = \mathbb E \left[ f \left(I'',L_1, \ldots, L_{I''-1}, L'', L_{I''}\!-\!L'', L_{I''+1}, \ldots L_n \right)\! \right] . \label{massselislengthsel}\end{equation}

Conditioning on $I'=j$, and using Proposition \ref{Diri}(iv), the LHS of \eqref{massselislengthsel} is 
\begin{align*}
\sum_{j\in[n]} \mathbb E \left[ f \left(I',L_1, \ldots, L_{I'-1}, L',  L_{I'}-L', L_{I'+1}, \ldots L_n\right) \mid I'=j \right] \left({\theta_j}\bigg/{\sum_{i\in[n]} \theta_i}\right).\end{align*}
Conditionally given $I'=j$, we select an atom of the $(\alpha, \theta_j)$-string of beads via $(\alpha, \theta_j)$-coin tossing sampling. By Proposition \ref{cointoss} and Proposition \ref{Diri}(ii), the mass split $(1-\lambda_j(L'))^{-1}\left(\lambda_j\left(\left[0, L'\right)\right), \lambda_j\left(\left(L', L_j\right]\right)\right) \sim {\rm Dirichlet}\left(\alpha, \theta_j\right)$
and the $(\alpha, \alpha)$- and the $(\alpha, \theta)$-strings of beads given by 
\begin{equation*} \left( \lambda\left(\left[0, L'\right)\right)^{-\alpha} [0, L'), \lambda\left(\left[0, L'\right)\right)^{-1}\lambda\restriction_{\left[0, L'\right)} \right), \quad \left( \lambda\left(\left( L', L_j\right]\right)^{-\alpha} \left(L', L_j\right], \lambda\left(\left( L',L_j\right]\right)^{-1}\lambda\restriction_{\left(L',L_j\right]} \right), 
\end{equation*}
respectively, are independent. By Proposition \ref{DIRML}, we conclude that the relative length split on $[0,L_j]$ is $L'/L_j\sim {\rm Beta}(1, \theta_j/\alpha)$. To see \eqref{massselislengthsel}, proceed likewise with the RHS of \eqref{massselislengthsel}, using that, by Proposition \ref{DIRML}, $(L_1, \ldots, L_n)\sim {\rm Dirichlet}(\theta_1/\alpha, \ldots, \theta_n/\alpha)$. More precisely, note that $\mathbb P(I''=j)=(\theta_j/\alpha)/(\sum_{i\in[n]} \theta_i/\alpha)=\theta_j/\sum_{i\in[n]} \theta_i$, and that, conditionally given $I''=j$, we have $L''/L_j\sim {\rm Beta}(1,\theta_j/\alpha)$, as before.
\end{proof}


We will also need the following statement about sampling from Poisson-Dirichlet distributions.

\begin{prop}[Sampling from PD$(\alpha, \theta)$, {\cite[Proposition 34]{30}}] \label{usefulPD}
Let $(P_i, i \geq 1) \sim \rm{PD}(\alpha, \theta)$ for some $0 \leq \alpha < 1$ and $\theta > -\alpha$, and let $N$ be an index such that 
\begin{equation*} \mathbb P\left(N=i \mid P_i, i \geq 1 \right)=P_i, \quad i \geq 1. \end{equation*}
Let $(P'_i, i \geq 1)$ be obtained from $P$ by deleting $P_N$, and set $P''_i:=P_i'/(1-P_N)$ for $i \geq 1$. Then, $P_N \sim \rm{Beta}(1-\alpha, \alpha + \theta)$, and $(P_i'', i \geq 1) \sim \rm{PD}(\alpha, \alpha + \theta)$ is independent of $P_N$.
\end{prop}

\subsection{Line-breaking constructions of the stable tree, and the proof of Theorem \ref{samestable}} \label{lconst}

In this section, we collect some preliminary results on stable trees and prove Theorem \ref{samestable}. Recall the line-breaking construction of the stable tree given by Algorithm \ref{GH} yielding the sequence of compact $\mathbb R$-trees $(\mathcal T_k, k \geq 0)$. Leaves and branch points have a natural order induced by the time of appearance in the sequence $(\mathcal T_k, k \geq 0)$, i.e. we can write $(v_i, i \geq 1)$ for the branch points, and $W_k^{(i)}$ for the branch point weight of $v_i$ in $\mathcal T_k$ (if $v_i \notin {\rm Br}(\mathcal T_k)$ or $i > b_k$, set $W_k^{(i)}=0$). We will list the edges $E_k^{(1)},\ldots,E_k^{(|\mathcal T_k|)}$ of $\mathcal T_k$ and their 
lengths $L^{(i)}_k={\rm Leb}(E^{(i)}_k)$, $i\in[|\mathcal T_k|]$, in the order encountered on a depth-first search directed by
least labels.

\begin{lemma}[{\cite[Proposition 3.2]{12}}] \label{GH1} For $k \geq 1$, given the shapes of $\mathcal T_0, \ldots, \mathcal T_k$, and $\lvert\mathcal T_k\vert -(k+1)=\ell$, i.e. conditionally given that the tree $\mathcal T_k$ has $k+1+\ell$ edges and $\ell$ branch points $(v_i, i \in [\ell])$, 
\begin{equation}
\left( L_k^{(1)}, \ldots, L_k^{(k+1+\ell)}, W_k^{(1)}, \ldots, W_k^{(\ell)}\right)=S_k \cdot \left(Z_k^{(1)}, \ldots, Z_k^{(k+\ell+1)}, Z_k^{(k+\ell+2)}, \ldots Z_{k}^{(k+1+2\ell)}\right)
\end{equation}
where $(Z_k^{(1)}, \ldots, Z_k^{(k+\ell+1)}, Z_k^{(k+\ell+2)}, \ldots Z_{k}^{(k+1+2\ell)}) \sim {\rm{Dirichlet}}\left( 1, \ldots, 1, w(d_1)/{\beta}, \ldots, w(d_\ell)/{\beta}\right)$ and $S_k \sim {\rm{ML}}(\beta, \beta+k)$ are independent, $w(d_i)=(d_i-3)(1-\beta)+1-2\beta$ and $d_i={\rm deg}(v_i, \mathcal T_k)$ is the degree of $v_i$.
\end{lemma}


\begin{corollary}[Masses as lengths] \label{Massesaslengths} For $k \geq 1$, given the shapes of $\mathcal T_0, \ldots, \mathcal T_k$, and $\lvert\mathcal T_k\vert -(k+1)=\ell$,
\begin{equation}
\left( L_k^{(1)}, \ldots, L_k^{(k+1+\ell)}, W_k^{(1)}, \ldots, W_k^{(\ell)}\right)= \left(X_{1}^\beta M_k^{(1)}, \ldots, X_{k+1+2\ell}^\beta M_{k}^{(k+1+2\ell)}\right) \label{masl}
\end{equation}
where the random variables $M_k^{(i)} \sim {\rm ML}(\beta, \beta)$, $i \in [k+1+\ell]$, $M_k^{(k+1+\ell+i)} \sim {\rm ML}(\beta, w(d_i))$, $i \in [\ell]$, and 
$X=\left(X_1, \ldots, X_{k+1+\ell}, X_{k+2+\ell}, \ldots X_{k+1+2\ell}\right)\sim {\rm{Dirichlet}}\left( \beta, \ldots, \beta, w(d_1), \ldots,w(d_\ell)\right)$
are independent, and $w(d_i)=(d_i-3)(1-\beta)+1-2\beta$ with $d_i={\rm deg}(v_i, \mathcal T_k)$.
\end{corollary} 

\begin{proof} We apply Lemma \ref{GH1}, and Proposition \ref{DIRML} with $n=k+1+2\ell$, $\theta_i=\beta, i \in [k+1+\ell]$, and $\theta_{k+1+\ell+i}=w(d_i)$, 
  $i \in [\ell]$. It remains to check that $\theta=\sum_{i\in[n]} \theta_i=\beta+k$, i.e. that 
  \begin{equation*} {(\beta+k})/{\beta}= k+\ell+1+\sum_{i\in[\ell]} \left({(d_i-3)(1/\beta-1)+(1/\beta-2)} \right). 
  \end{equation*} 
  This follows from the fact that the sum of the vertex degrees in a tree with $m$ edges is $2m$, i.e. 
  $\sum_{i\in[\ell]} d_i=2(k+1+\ell)-(k+1)-1$, since $\mathcal T_k$ has $k+1+\ell$ edges and $(k+1)+1$ degree-1 vertices. 
\end{proof}

Haas et al. \cite{35} analysed the stable tree as an example of a self-similar CRT. Let $(\mathcal T, d, \rho)$ with mass measure $\mu$ be the stable tree of parameter 
$\beta \in (0,1/2]$, and let $\Sigma \sim \mu$ be a leaf sampled from $\mu$. Consider the \textit{spine}, i.e.\ the path $[[\rho, \Sigma]]$ from the root to this leaf. 
Remove all vertices of degree one or two from this path. This yields a sequence of connected components that can a.s. be ranked in decreasing order of mass, and which
we denote by $(\overline{\mathcal S}^{(i)}, i \ge 1)$, rooted at vertices $\rho_i \in [[\rho, \Sigma]]$ of a.s. infinite degree, $i\ge 1$, respectively. Each 
$\overline{\mathcal S}^{(i)}$ further separates into a sequence $(\overline{\mathcal S}^{(i)\downarrow}_j,j\ge 1)$ when removing $\rho_i$. 
\begin{itemize} 
\item  The \textit{coarse spinal mass partition} is $\big(\overline{P}^{(i)}, i \geq 1\big):=\big(\mu(\overline{\mathcal S}^{(i)}), i \geq 1\big)$,
\item The \textit{fine spinal mass partition} is the sequence  $\big(\overline{P}_{j}^{(i)\downarrow}, j\geq 1,i\geq 1\big)^{\downarrow}:=\big(\mu \big(\overline{\mathcal S}_{j}^{(i)\downarrow}\big), j \ge 1, i \ge 1\big)^{\downarrow}$, i.e.\ the ranked sequence of masses of connected components obtained after removal of the whole spine.
\end{itemize}

\begin{theorem}[Mass partition in the stable tree, {\cite[Corollary 10]{35}}] \label{masspart}

Let $\beta \in (0,1/2]$, and let $\mathcal T$ be the stable tree of parameter $\beta$. Then the following statements hold.
\begin{enumerate}
\item[{\rm (i)}] The coarse spinal mass partition has a Poisson-Dirichlet distribution with parameters $(\beta, \beta)$, i.e.  \begin{equation*} \left(\overline{P}^{(i)}, i \geq 1\right) = \left(\mu\left(\overline{\mathcal S}^{(i)}\right), i \ge 1\right) \sim {\rm{PD}}\left(\beta, \beta\right). \end{equation*}
\item[{\rm (ii)}] The fine spinal mass partition is a $(1-\beta,-\beta)$-fragmentation of the coarse spinal mass partition, i.e. for each block 
  $\mu(\overline{\mathcal S}^{(i)})$ of the coarse partition, the relative part sizes 
  $(\mu(\overline{\mathcal S}^{(i)\downarrow}_j)/\mu(\overline{\mathcal S}^{(i)}),j\ge 1)$ are independent with distribution {\rm PD}$(1-\beta, -\beta)$, $i\ge 1$.
\item[{\rm (iii)}]Conditionally given the fine spinal mass partition $(\mu(\overline{\mathcal S}_{j}^{(i)\downarrow}), j \ge 1, i\ge 1)^{\downarrow}$, the rescaled trees equipped with restricted mass measures
\begin{equation}
\left(\mu\left(\overline{\mathcal S}_{j}^{(i)\downarrow}\right)^{-\beta} \overline{\mathcal  S}_{j}^{(i)\downarrow}, \mu(\overline{\mathcal  S}_{j}^{(i)\downarrow})^{-1} \mu \restriction_{\overline{\mathcal  S}_{j}^{(i)\downarrow}} \right),  \quad j \ge 1, i \ge 1, \end{equation}
 are i.i.d. copies of $(\mathcal T, \mu)$.
\end{enumerate}
\end{theorem}

The $\alpha$-diversities of PD$(\alpha, \theta)$ partitions can naturally be interpreted as lengths in trees. In 
particular the $\beta$-diversity of the coarse spinal mass partition has distribution $S_0 \sim \text{ML}(\beta, \beta)$, which is the starting point of 
Goldschmidt-Haas' line-breaking constructions. The fragmenting PD$(1-\beta, -\beta)$ random partitions for each block of the coarse spinal 
mass partition capture important information about the branch points that we relate to sizes of the Ford CRTs by which we 
replace them in Theorem \ref{branchrepl}. Specifically, the independence of these PD$(1-\beta, -\beta)$ vectors relates to the independence of the Ford trees. 
Sampling i.i.d. leaves $(\Sigma_k, k \geq 0)$ from the measure $\mu$ of the stable tree yields a natural random order of 
$(\overline{\mathcal S}_j^{(i)\downarrow}, j \ge 1)$, in terms of smallest leaf labels of the subtrees, which we write as $(\overline{\mathcal S}_j^{(i)}, j \geq 1)$, for each $i \geq 1$.

\begin{corollary} \label{GEM} Let $(\mathcal T, \mu)$ be a stable tree of index $\beta \in (0,1/2]$ with associated reduced tree sequence $(\mathcal T_k, k \geq 0)$. Let 
  $\overline{\mathcal S}^{(i)}$ be the subtree rooted at $\rho_i \in  [[\rho, \Sigma_0]]$, $i \ge 1$, related to the coarse spinal mass partition $(\mu(\overline{\mathcal {S}}^{(i)}), i \ge 1)$.
  For each $i \ge 1$, let $(\overline{\mathcal S}_j^{(i)}, j \geq 1)$ denote the connected components of $\overline{\mathcal S}^{(i)} \setminus \rho_i$, ordered in increasing order of least 
  leaf labels. Then $ ( \mu(\overline{\mathcal S}_j^{(i)})^{-1} \mu(\overline{\mathcal S}_j^{(i)}), j \geq 1) \sim {\rm GEM}\left(1-\beta, -\beta\right).$ 
\end{corollary}

\begin{proof} This is a direct consequence of Theorem \ref{masspart}(ii) in combination with results on sampling from PD$(\alpha, \theta)$, cf. Theorem \ref{usefulPD}, and the construction \eqref{v1v2} of GEM$(\alpha, \theta)$. 
\end{proof}


We now show that the line-breaking construction of the stable tree based on masses (Algorithm \ref{masssel}), yields trees $(\mathcal T_k, k \geq 0)$ as in \eqref{findimmarg1} and Algorithm \ref{GH}. The following result will prove Theorem \ref{samestable}.

\begin{prop} \label{masssel2} The sequence of weighted $\mathbb R$-trees $(\mathcal T_k, \mu_k, k \geq 0)$ from Algorithm \ref{masssel} has the same distribution as the sequence of trees in \eqref{findimmarg1} equipped with projected subtree masses, i.e. with the mass measures $(\pi_k)_*\mu, k \geq 1,$ as in \eqref{pm}-\eqref{pfpm}.
Furthermore, conditionally given $\lvert \mathcal T_k \rvert =k+1+\ell$, the edges of $\mathcal T_k$ equipped with the mass measure $\mu_k$ restricted to each edge, are rescaled independent $(\beta, \beta)$-strings of beads given via
\begin{equation} \left( \mu_k\left(E_k^{(i)}\right)^{-\beta} E_k^{(i)}, \mu_k\left(E_k^{(i)}\right)^{-1} \mu_k \restriction_{E_k^{(i)}}\right), \quad i \in [k+1+\ell], \end{equation} 
and the total mass distribution 
\begin{equation*} \left(\mu_k\left(E_k^{(1)}\right), \ldots, \mu_{k}\left(E_k^{(k+1+\ell)}\right), \mu_k\left(v_1\right), \ldots, \mu_k \left(v_\ell\right)\right) \sim {\rm{Dirichlet}}\left(\beta, \ldots, \beta, w\left(d_1\right), \ldots, w\left(d_\ell\right)\right) \end{equation*}
where $v_i, i \in [\ell],$ are the branch points of $\mathcal T_k$ of degrees $d_i={\rm deg}(v_i, \mathcal T_k)$, $i \in [\ell]$, respectively,  and $w(d_i)=(d_i-3)(1-\beta)+(1-2\beta)$, $i \in [\ell]$, and where we number the edges $E_k^{(i)}, i \in [k+1+\ell]$ by depth-first search.
\end{prop}

The proof of Proposition \ref{masssel2} is part of Appendix \ref{appen1}, where we collect several similar proofs. We also record the following consequence of Algorithm \ref{masssel} and Proposition \ref{masssel2}. 

\begin{corollary} \label{GEMS} Let $(\mathcal T, \mu)$ be a stable tree of index $\beta \in (0, 1/2]$, and let $(\mathcal T_k, k \geq 0)$ be as in \eqref{findimmarg1} with branch points $(v_i, i \geq 1)$ in order of appearance. Let $k_i:=\inf\left\{k \geq 0\colon [[\rho, \Sigma_k]] \cap [[\rho, v_i]]=[[\rho, v_i]] \right\}$ and let $(\mathcal S_j^{(i)}, j \geq 1)$ be the subtrees of $\mathcal T \setminus [[\rho, \Sigma_{k_i}]]$ rooted at $v_i$ in increasing order of smallest leaf labels, $i\ge 1$.  Set $P_j^{(i)}:=\mu(\mathcal S_j^{(i)})$ and $P^{(i)}=\sum_{j \geq 1} \mu(S_j^{(i)} )$, $i \geq 1$. Then the sequences $(P_j^{(i)}/P^{(i)}, j \geq 1), i \geq 1$, are i.i.d. with distribution ${\rm GEM}(1-\beta, -\beta)$.
\end{corollary}

\begin{proof} This is a direct consequence of the stick-breaking representation \eqref{v1v2} of GEM$(1-\beta, -\beta)$ and the random variables $(Q_k, k \geq 0)$ splitting branch point mass into subtrees from Algorithm \ref{masssel}. Specifically, conditionally given the branch point degrees in the sequence $(\mathcal T_k, k \geq 0)$, for each branch point $v_i$, we can find a sequence of random variables $(Q_{m}^{(i)}, m \geq 1)$ such that $$P_j^{(i)}= \mu_{k_1^{(i)}-1}\left(v_i\right)  Q_j^{(i)}    \prod_{m\in[j-1]} \left(1-Q_m^{(i)} \right), \quad j \geq 1,$$
where $Q_m^{(i)}:=Q_{k_m^{(i)}} \sim {\rm Beta}(\beta, m(1-\beta)-\beta))$ and $k_m^{(i)}=\inf\{k \geq 1\colon {\rm deg}(v_i, \mathcal T_k) =m+1\}$.
Note that, for $m_1, \ldots, m_i \geq 1$, the random variables $Q_j^{(i)}, j \in [m_i]$, $i \geq 1$, have conditional distributions given $k_j^{(i)}$, $j \in [m_i]$, $i \geq 1$, that do not depend on $k_j^{(i)}$, $j \in [m_i]$, $i \geq 1$, and are hence unconditionally independent.
\end{proof}

\section{The binary two-colour line-breaking construction with masses} \label{sec4}

We present an enhanced version of Algorithm \ref{twocolour}, which is based on sampling from the mass measure. We use this enhanced version to prove Theorem
\ref{Mainresult1}. 

The following (1-marked) string of beads will be at the centre of our construction. For $\beta \in (0,1/2]$, 
consider 
$([0,K_1], \lambda_1)$ and $([0,K_2], \lambda_2)$ two independent $(\beta, 1-2\beta)$- and $(\beta, \beta)$-strings of beads, respectively, and an independent $B \sim \rm{Beta}(1-2\beta, \beta)$. Then scale the two strings by $B$ and $1-B$, as follows: set
\begin{equation} K:=B^\beta K_1+(1-B)^\beta K_2, \qquad K':=B^\beta K_1 \label{twocolstr1}\end{equation}
and consider the mass measure $\lambda$ on $[0,K]$ given by
\begin{equation}\label{twocolstr2} 
\lambda\left(\left[0,x\right]\right)=\begin{cases} B \lambda_1\left(\left[0,B^{-\beta}x\right]\right) &\text{ if } x \in \left[0,K'\right], \\
B+ \left(1-B\right) \lambda_2\left(\left[0,\left(1-B\right)^{-\beta}\left(x-K'\right)\right]\right) &\text{ if } x \in \left[K', K\right]. \end{cases}
\end{equation}
The string of beads $([0,K],\lambda)$ 
is called a \textit{$\beta$-mixed} string of beads \cite{RW}. We denote the distributions of $([0,K], \lambda)$ and $([0,K],[0,K'], \lambda)$ 
on $\mathbb{T}_{\rm w}$ and $\mathbb{T}_{\rm w}^{[1]}$ 
by $\nu_\beta$ and $\nu_\beta^{[1]}$, respectively.

\begin{remark} \label{betamixedlength} 
By Proposition \ref{DIRML} with $\theta_1=1-2\beta, \theta_2=\beta$, noting that $(B, 1-B) \sim \text{Dirichlet}(1-2\beta, \beta)$, we have 
\begin{equation} \left(B^\beta K_1, (1-B)^\beta K_2\right)\,{\buildrel d \over =}\, L  \left({B}', 1-{B}'\right)\label{twoclb} \end{equation}
where ${B}' \sim \text{Beta}(1/\beta-2, 1)$ is independent of $L$, and $L \sim \text{ML}(\beta, 1-\beta)$. We conclude that for each $\beta$-mixed string of beads $\xi=([0,K], \lambda)$ we have
$ (\lambda(x)\colon x \in [0,K], \lambda(x) > 0)^{\downarrow} \sim \text{PD}(\beta, 1-\beta ),  $
cf. e.g. \cite[Corollary 1.2]{31}. Although the length of a $\beta$-mixed string of beads $\xi$ is ML$(\beta, 1-\beta)$ and the atom sizes are PD$(\beta,1-\beta)$, we cannot expect that $\xi$  is a $(\beta, 1-\beta)$-string of beads when $\beta \in (0,1/2)$. Specifically, at the junction point in a $(\beta, 1-\beta)$-string of beads, we would expect a Beta$(\beta, 1-2\beta)$ mass split into a rescaled $(\beta, \beta)$- and a rescaled $(\beta, 1-2\beta)$-string of beads in this order (and not vice versa).
\end{remark}
  
We will use the notation
$ \xi = \left( [0,K|, \sum_{i \geq 1}P_i \delta_{X_i}\right)$
for any $(\alpha, \theta)$- or $\beta$-mixed string of beads where $K$ is the length of the string of beads with ranked atomic masses of sizes $1 > P_1 > P_2 > \cdots > 0$, a.s., in the points $X_i \in [0,K]$, $i \geq 1$, respectively.

Let us now explain how to attach a weighted $\mathbb R$-tree onto another weighted $\mathbb R$-tree. This clarifies in particular how to construct weighted $\mathbb R$-trees by attaching strings of beads as a string of beads can be interpreted as a weighted $\mathbb R$-tree consisting of a single branch. For any weighted $\mathbb R$-tree $(\mathcal T, d, \rho,  \mu)$, a parameter $\beta \in (0,1/2]$, an element $J \in \mathcal T$ and another weighted $\mathbb R$-tree $(\mathcal T^+, d^+, \rho^+, \mu^+)$ with $\mathcal T \cap \mathcal T^+ = \emptyset$, the tree $(\mathcal T', d', \mu')$ created from $(\mathcal T, d, \mu)$ by \textit{attaching} to $J$ the tree $(\mathcal T^+, d^+, \rho^+, \mu^+)$ with mass measure $\mu^+$ rescaled by $\mu(J)$ and metric $d^+$ rescaled by $\mu(J)^\beta$ is defined as follows. Specifically, set
\begin{equation}
\mathcal T':= \mathcal T \setminus \{J\} \sqcup \mathcal T^+,  \label{attach1}\quad
\begin{aligned}
d'(x,y):=\begin{cases} d(x,y) & \text{ if } x,y \in \mathcal T, \\ d(x,J)+(\mu(J))^{\beta} d^+(\rho^+,y) & \text{ if } x \in \mathcal T, y \in \mathcal T^+, \\
(\mu(J))^{\beta} d^+(x, y) & \text{ if } x, y \in \mathcal T^+,\end{cases}\qquad\rho'=\rho,\quad
\end{aligned}
\end{equation}
and equip $(\mathcal T',d',\rho')$ with the mass measure $\mu'$ given by 
$\mu' \restriction_{\mathcal T \setminus \{J\}}=\mu\restriction_{\mathcal T \setminus \{J\}}, \  \mu'\left(J\right)=0, \ \mu' \restriction_{\mathcal T^+}=\mu\left(J\right) \mu^+.$

We are now ready to present the two-colour line-breaking construction with masses. 

\begin{algorithm}[Two-colour line-breaking construction with masses] \label{twocolourmass} \rm Let $\beta\in(0,1/2]$. We grow weighted $\infty$-marked $\bR$-trees $(\cT_k^*,(\cR_k^{(i)},i\ge 1),\mu_k^*)$, $k\ge 0$, as follows.
  \begin{enumerate}\setcounter{enumi}{-1}\item Let $(\cT_0^*,\mu_0^*)$ be isometric to a $(\beta,\beta)$-string of beads; let $r_0=0$ and $\cR_0^{(i)}=\{\rho\}$, $i\ge 1$.
  \end{enumerate}
  Given $(\cT_j^*,(\cR_j^{(i)},i\ge 1),\mu_j^*)$ with $\mu_j^*=\sum_{x\in\cT_j^*}\mu_j^*(x)\delta_x$, $0\le j\le k$, let $r_k=\#\{i\ge 1\colon\cR_k^{(i)}\neq\{\rho\}\}$;
  \begin{enumerate}\item select an edge $E_k^*\subset\cT_k^*$ with probability proportional to its mass $\mu_k^*(E_k^*)$; if 
      $E_k^*\subset\cR_k^{(i)}$ for some $i\in[r_k]$, let $I_k=i$; otherwise, i.e.\ if $E_k^*\subset\cT_k^*\setminus\bigcup_{i\in[r_k]}\cR_k^{(i)}$, let $r_{k+1}\!=\!r_k\!+\!1$, $I_k\!=\!r_{k+1}$;  
    \item if $E_k^*$ is an external edge of $\cR_k^{(i)}$, perform $(\beta,1-2\beta)$-coin tossing sampling on $E_k^*$ to determine $J_k^*\in E_k^*$ (cf. Proposition \ref{cointoss});  otherwise, i.e.\ if $E_k^*\subset\cT_k^*\setminus\bigcup_{i\in[r_k]}\cR_k^{(i)}$ or if $E_k^*$ is an internal edge of $\cR_k^{(i)}$, sample $J_k^*$ from the normalised mass measure on $E_k^*$;
    \item let $(E_k^+,R_k^+,\mu_k^+)$ be an independent $\beta$-mixed string of beads; to form $(\cT_{k+1}^*,\mu_{k+1}^*)$ remove $\mu_k^*(J_k^*)\delta_{J_k^*}$ from $\mu_k^*$ and attach to $\cT_k^*$ at $J_k^*$ an isometric copy of $(E_k^+,\mu_k^+)$ with measure rescaled by $\mu_k^*(J_k^*)$ and metric rescaled by $(\mu_k^*(J_k^*))^\beta$; add to $\cR_k^{(I_k)}$ the (image under the isometry of) $R_k^+$ to form $\cR_{k+1}^{(I_k)}$; set $\cR_{k+1}^{(i)}=\cR_k^{(i)}$, $i \neq I_k$. 
  \end{enumerate}
\end{algorithm}

\subsection{The distribution of two-colour trees}

To analyse Algorithm \ref{twocolourmass}, we will need some more notation, in particular with regard to the marked subtree growth processes $(\mathcal R_k^{(i)}, k \geq 0)$, $i \geq 1$. Define the random subsequences $(k_m^{(i)}, m \geq 1)$, $i \geq 1$, by
\begin{equation} 
k_1^{(i)}:=\inf \left\{n \geq 1\colon \mathcal R_{n}^{(i)} \neq \mathcal R_{0}^{(i)} \right\}=\inf \left\{n \geq 1\colon \mathcal R_{n}^{(i)} \neq \{\rho\}\right\},
\end{equation}
and, for $m \geq 1$, \vspace{-0.2cm}
\begin{equation}
k_{m+1}^{(i)}:=\inf \left\{n \geq k_m^{(i)}\colon \mathcal R_{n}^{(i)} \neq \mathcal R_{k_m^{(i)}}^{(i)}\right\}, \label{subsequences}
\end{equation}
i.e. there is a change in $(\mathcal R_k^{(i)}, k \geq 1)$ when $k=k_{m}^{(i)}$ for some $m \geq 1$. 
Note that $\bigcup_{i \geq 1} \{ k_m^{(i)}, m \geq 1\}=\{1,2, \ldots\}$ is a disjoint union, and that, for any $i \geq 1$,  $\mathcal R_k^{(i)}$ is a binary tree for any $k \geq 1$. We will also use the convention that $\rho\notin \mathcal R_k^{(i)}$ for $k \geq k_1^{(i)}$. For $k=k_m^{(i)}-1$, we write
 \begin{equation*} R_k^+ = [[J_k^*, \Omega_{m}^{(i)}]] \subset E_k^+ =[[J_k^*, \Sigma_{k+1}]], \qquad \text{i.e.} \quad ]]J_k^*, \Omega_{m}^{(i)}]]= \mathcal R_{k+1}^{(i)} \setminus \mathcal R_{k}^{(i)}. \end{equation*} 
In other words, at step $k=k_m^{(i)}-1$, $\Omega_{m}^{(i)}$ and $\Sigma_{k+1}$ denote the leaves added to $\mathcal R_{k}^{(i)}$ and $\mathcal T_k^*$, respectively. 

We write $\xi_k^{(1)}, \xi_k^{(2)}$ and $\gamma_k$ for the random variables inducing the $\beta$-mixed string of beads $(E_k^+, R_k^+, \mu_k^+)$, i.e. 
$\left(E_k^+, R_k^+, \mu_k^+\right)$ is built from independent $\xi_k^{(1)}, \xi_k^{(2)}$ and $\gamma_k$ in the same way as $([0,K],[0,K^\prime],\lambda)$ is built from independent $([0,K_1],\lambda_1)$, $([0,K_2],\lambda_2)$ and $B$ in \eqref{twocolstr1}-\eqref{twocolstr2}.

Furthermore, we use an equivalence relation $\sim$ on $(\mathcal T_k^*, (\mathcal R_k^{(i)}, i \geq 1))$ to contract each marked component $\mathcal R_k^{(i)}$, $i \geq 1$, of $\mathcal T_k^*$ to a single point, i.e. 
\begin{equation}
x \sim y\qquad :\Leftrightarrow \qquad x,y \in \mathcal R_k^{(i)} \quad \text{ for some } i \geq 1. \label{equrel}
\end{equation}
Note that, for all $i$ with $\mathcal R_k^{(i)} \neq \{\rho\}$, $x,y \in \mathcal R_k^{(i)}$ implies $x,y\in \mathcal R_{k'}^{(i)}$ for all $k' \geq k$, and hence the equivalence relation $\sim$ is consistent as $k$ varies. Denote the equivalence class related to $\mathcal R_k^{(i)}$ by $\widetilde{v}_i:=[\mathcal R_k^{(i)}]_{\sim}$, and let 
\begin{equation}
\widetilde{\mathcal T}_k:=\mathcal T_k^*/\sim \label{ttilde}
\end{equation}
denote the quotient space of $\mathcal T_k^*, k \geq 0$, with the canonical quotient metric. Furthermore, for $k \geq 0$, let $\widetilde{\mu}_k$ be the push-forward of $\mu_k^*$ under the projection map from $\mathcal T_k^*$ onto $\widetilde{\mathcal T}_k$. 

The following characterisation of Ford trees will be useful to obtain the distribution of $\mathcal T_k^*$. 

\begin{prop}[{\cite[Proposition 18]{2}}] \label{Forddis} Consider the tree growth process $(\mathcal F_m, m\geq 1)$ from Algorithm \ref{Ford} for some $\beta' \in (0,1)$. The distribution of $\mathcal F_m$ is given in terms of three independent random variables: its shape,  the total length $S_m' \sim {\rm ML}(\beta', m-\beta')$ and the length split between the edges of $\mathcal F_m$ which has a ${\rm Dirichlet}\left(1, \ldots, 1, (1-\beta')/\beta', \ldots, (1-\beta')/\beta'\right)$ distribution, where a parameter of $1$ is assigned to each of the $m-1$ internal edges, and  a parameter of $(1-\beta')/\beta'$ to each of the $m$ external edges of $\mathcal F_m$.
\end{prop}

We can describe the distribution of the tree $\mathcal T_k^*$ as follows.

\begin{prop}[Distribution of $\mathcal T_k^*$] \label{starstrings} Let $(\cT_k^*,(\cR_k^{(i)},i\ge 1),\mu_k^*, k \geq 0)$ be as in Algorithm \ref{twocolourmass} for some
  $\beta \in (0,1/2]$. The distribution of $\mathcal T_k^*$ is characterised by the following independent random variables:
  \begin{itemize} 
    \item the shape $T_k^*$ of $\mathcal T_k^*$ obtained from the shape $\widetilde{T}_k$ of $\widetilde{\mathcal T}_k$ and the shapes $R_k^{(i)}$ of 
      $\mathcal R_k^{(i)}$, $i \geq 1$, as follows; 
      \begin{itemize} 
        \item $\widetilde{T}_k$ has the distribution of the shape of a stable tree $\mathcal T_k$ reduced to the first $k$ leaves, and 
        \item conditionally given that $\widetilde{T}_k$ has $\ell$ branch points of degrees $d_1, \ldots, d_\ell$, the shapes $R_k^{(1)}, \ldots, R_k^{(\ell)}$ are the shapes 
          of Ford trees with $m_1:=d_1-2, \ldots, m_\ell:=d_\ell-2$ leaves, respectively; 
      \end{itemize}
    \item the total mass split between the $3k+1$ edges of $\mathcal T_k^*$ has a
      \begin{equation}
        {\rm Dirichlet}\left(\beta, \ldots, \beta, 1-2\beta, \ldots, 1-2\beta\right) \label{Es0} 
      \end{equation}
      distribution, with parameter $\beta$ for each internal marked and each unmarked edge, and parameter $1-2\beta$ for each external marked edge with edges ordered 
      according to depth-first search (first run for unmarked and internal marked edges, then for external marked edges); 
    \item the $3k+1$ independent $(\beta, \theta)$-strings of beads isometric to
      \begin{equation}
        \left(\mu_k^*\left(E\right)^{-\beta}E, \mu_k^*\left(E\right)^{-1} \mu_k^* \restriction_{E}\right), \quad E \in {\rm Edg}\left(\mathcal T_k^*\right), \label{Es}
      \end{equation} 
      where $\theta=1-2\beta$ if $E$ is an external marked edge of $\mathcal R_k^{(i)}$ for some $i \in [\ell]$, and $\theta=\beta$ otherwise, again listed according to 
      depth-first search. 
  \end{itemize}
\end{prop}

\begin{proof} This proof is mainly an application of the properties of the Dirichlet distribution, Proposition \ref{Diri}, and of  coin tossing sampling, Proposition \ref{cointoss}. We give a brief sketch of the proof via an induction on $k$. 

For $k=0$, the claim is trivial as $(\mathcal T_0^*, \mu_0^*)$ is a $(\beta, \beta)$-string of beads by definition. For the induction step, suppose that the claim holds for some $k \geq 0$.

We first consider the shape transition from $T_k^*$ to $T_{k+1}^*$. Observe that, given $\widetilde{T}_k$ has $\ell$ branch points of degrees $d_1, \ldots, d_\ell$, we have a $ {\rm{Dirichlet}}\left(\beta, \ldots, \beta, w(d_1), \ldots, w(d_\ell)\right)$ mass split in $\widetilde{\mathcal T}_k$ with weight $\beta$ for each edge and weight $w(d)=(d-2)(1-\beta)-\beta$ for each branch point of degree $d \geq 3$. Hence, by Proposition \ref{masssel2}, the overall edge selection is as in Algorithm \ref{masssel}. 

Conditionally given that the $i$th branch point of $\widetilde{T}_k$ is selected, an edge of $\mathcal R_k^{(i)}$ is chosen proportionally to the weights assigned by the relative $\rm{Dirichlet}\left(\beta, \ldots, \beta, 1-2\beta, \ldots, 1-2\beta\right)$ mass split in $\mathcal R_k^{(i)}$, so each internal edge is chosen with probability $\beta/((d_i-2)(1-2\beta)+(d_i-3)\beta)$, each external edge with probability $(1-2\beta)/((d_i-2)(1-2\beta)+(d_i-3)\beta)$. This corresponds to the shape growth rule in a Ford tree growth process of index $\beta/(1-\beta)$, using obvious cancellations, cf. Algorithm \ref{Ford} and Proposition \ref{Forddis}.

In the update step from $\mathcal T_k^*$ to $\mathcal T_{k+1}^*$, we first select an edge of $\mathcal T_k^*$ proportionally to mass. By Proposition \ref{Diri}(iv), the parameter for this edge in the Dirichlet split \eqref{Es0}, conditionally given that it has been selected, is then increased by $1$. We select an atom $J_k^*$ on this edge via $(\beta, \theta)$-coin tossing, where $\theta=1-2\beta$ for external marked edges, and $\theta=\beta$ otherwise, and, by Proposition \ref{cointoss}, the selected edge is split by $J_k^*$ into a rescaled independent $(\beta, \beta)$- and a rescaled independent $(\beta, \theta)$-string of beads where the relative mass split on this edge is ${\rm Dirichlet}(\beta, 1-\beta, \theta)$, which is conditionally independent of the total mass split. Furthermore, the mass $\mu_k^*(J_k^*)$ is split by the independent random variable $\gamma_k \sim {\rm Beta}(1-2\beta, \beta)$ into a marked $(\beta,1-2\beta)$-string of beads, and an unmarked $(\beta, \beta)$-string of beads, which are independent, i.e., by Proposition \ref{Diri}(iii), the claims \eqref{Es0} and \eqref{Es} follow, as statements conditionally given tree shapes.

Finally, these conditional distributions of the Dirichlet mass split \eqref{Es0} and the independent $(\beta, \theta)$-strings of beads \eqref{Es} do not depend on the shape $T_{k+1}^*$, and are hence unconditionally independent. \end{proof}

\begin{remark}\label{rem45} 
  By Proposition \ref{starstrings} and Lemma \ref{equivalence} we see that Algorithm \ref{twocolourmass} reduces to Algorithm \ref{twocolour}. 
\end{remark}

\subsection{Identification of the stable line-breaking constructions} \label{stgp}

We now turn to the trees $(\widetilde{\mathcal T}_k, k \geq 0)$ obtained from $(\mathcal T_k^*, (\mathcal R_k^{(i)}, i \geq 1), k \geq 0)$ by contracting all marked components to single branch points as in \eqref{equrel}-\eqref{ttilde}. This description yields another formulation of the atom selection procedure on $\mathcal T_k^*$ in Algorithm \ref{twocolourmass}. 

Given $(\mathcal T_j^*, (\mathcal R_j^{(i)}, i \geq 1), \mu_j^*)$, $0 \leq j \leq k$, and $r_k = \#\{i \geq 1\colon \mathcal R_k^{(i)} \neq \{\rho\}\} =\#\{i \geq 1\colon \widetilde{v}_i \neq \{\rho\}\}$,
 \begin{enumerate}\item[1.-2.] select $\widetilde{J}_k$ from $\widetilde{\mu}_k$; if $\widetilde{J}_k \neq \widetilde{v}_i$ for all $i \in [r_k]$, set $J_k^*=\widetilde{J}_k$; otherwise, if $\widetilde{J}_k=\widetilde{v}_i$ for some $i \in [r_k]$, sample an edge $E_k^*$ of $\mathcal R_k^{(i)}$ proportionally to its mass $\mu_k^*(E_k^*)$; if $E_k^*$ is an internal edge of $\mathcal R_k^{(i)}$, sample $J_k^*$ from the normalised mass measure on $E_k^*$; if $E_k^*$ is an external edge of $\mathcal R_k^{(i)}$, perform $(\beta, 1-2\beta)$-coin tossing sampling on $E_k^*$ to determine $J_k^* \in E_k^*$.
  \end{enumerate}

It is this view on Algorithm \ref{twocolourmass} that we pursue further now. The following theorem contains the desired weight-length transformation, i.e. the branch point weights in Goldschmidt-Haas' stable line-breaking construction (Algorithm \ref{GH}) are indeed as the lengths of the marked subtrees in the two-colour line-breaking construction (Algorithm \ref{twocolourmass}). Its proof is given with other similar proofs in the appendix.

\begin{theorem} \label{mainresult} Let the sequence $(\mathcal T_k^*, (\mathcal R_k^{(i)}, i \geq 1 ), \mu_k^*, k \geq 0)$ be as in Algorithm \ref{twocolourmass}, and 
  associate $( \widetilde{\mathcal T}_k, (\widetilde{v}_i=[\mathcal R_k^{(i)}]_{\sim}, i \geq 1 ), \widetilde{\mu}_k, k \geq 0)$ as in \eqref{ttilde}. Then the following hold.
\begin{itemize}
\item[\rm(i)] The sequence of trees with mass measures from Algorithm \ref{twocolourmass} and \eqref{ttilde} has the same distribution as the sequence in Algorithm \ref{masssel}, i.e. 
\begin{equation}
\left(\widetilde{\mathcal T}_k, \widetilde{\mu}_k, k \geq 0\right) \,{\buildrel d \over =}\, \left({\mathcal T}_k, \mu_k , k \geq 0\right). \label{emstlb1}
\end{equation}
\item[\rm(ii)] The sequence of trees with marked component lengths from Algorithm \ref{twocolourmass} and \eqref{ttilde} has the same distribution as the sequence of trees with weights from Algorithm \ref{GH}, i.e. 
\begin{equation}
\left(\widetilde{\mathcal T}_k, \left(\widetilde{W}_k^{(i)}, i \geq 1\right), k \geq 0\right) \,{\buildrel d \over =}\, \left({\mathcal T}_k,  \left({W}_k^{(i)}, i \geq 1\right), k \geq 0\right), \label{emstlb2} \end{equation}
where  $\widetilde{W}_k^{(i)}={\rm Leb}(\mathcal R_k^{(i)})$ is the length of $\mathcal R_k^{(i)}, i \geq 1$, respectively. In particular, letting ${S}^*_k={\rm Leb}(\mathcal T_k^*)$ denote the length of $\mathcal T_k^*$, 
the sequence $({S}^*_k, k\geq 0)$ is a Mittag-Leffler Markov chain starting from ${\rm ML}(\beta,\beta)$
, i.e. 
$
\left({S}^*_k, k \geq 0\right)  \,{\buildrel d \over =}\,  \left(S_k, k \geq 0\right).
$
\end{itemize}
\end{theorem}

Let us pull some threads together and deduce the first assertion of Theorem \ref{Mainresult1} and the limit of $\widetilde{\mathcal T}_k$.

\begin{proof}[Proof of \eqref{weighteq} in Theorem \ref{Mainresult1}] We noted in Remark \ref{rem45} that the sequence of two-colour trees of Algorithm 
  \ref{twocolourmass} without mass measures has the same joint distribution as the sequence of two-colour trees of Algorithm \ref{twocolour}. Hence, \eqref{emstlb2} is 
  precisely \eqref{weighteq}.
\end{proof}


\begin{corollary} \label{convstable} In the setting of Theorem \ref{mainresult}, 
$ \lim_{k \rightarrow \infty} (\widetilde{\mathcal T}_k, \widetilde{\mu}_k)=(\mathcal T, \mu)$ a.s.\ with respect to the Gromov-Hausdorff-Prokhorov distance, where $(\mathcal T, \mu)$ is a stable tree of index $\beta$.
\end{corollary}
\begin{proof} Goldschmidt and Haas \cite{12} showed this for the RHS of \eqref{emstlb2}, so it also holds for the LHS.
\end{proof}

\subsection{Identification of marked subtree growth processes, and the proof of Theorem \ref{Mainresult1}} \label{sec42}

The main aim of this section is to identify the marked tree growth processes $(\mathcal R_k^{(i)}, k \geq 1)$, $i \geq 1$, as rescaled i.i.d. Ford tree growth processes of index $\beta'=\beta/(1-\beta)$. We will show the following.

\begin{theorem} \label{fordembb} Let $(\mathcal T_k^*, (\mathcal R_k^{(i)}, k \geq 1), \mu_k^*, k \geq 0)$ be the weighted $\infty$-marked tree growth process of Algorithm \ref{twocolourmass} for some $\beta \in (0,1/2]$. Then there exists a sequence of scaling factors $(C^{(i)}, i \geq 1)$ such that for all $i \geq 1$
$$ \lim_{k \rightarrow \infty} \mathcal R_k^{(i)}=\mathcal R^{(i)} \quad \text{a.s.} $$ in the Gromov-Hausdorff topology where 
 $(C^{(i)} \mathcal R^{(i)}, i \geq 1)$ is a sequence of i.i.d. Ford CRTs of index $\beta'=\beta/(1-\beta)$. Furthermore, the sequence $(C^{(i)} \mathcal R^{(i)}, i \geq 1)$ is independent of the stable tree $(\widetilde{\mathcal T}, \widetilde{\mu})=\lim_{k \rightarrow \infty} (\widetilde{\mathcal T}, \widetilde{\mu})$ obtained from $(\mathcal T_k^*, (\mathcal R_k^{(i)}, i \geq 1), \mu_k^*, k \geq 0)$ as in Corollary \ref{convstable}.
\end{theorem}

We will prove this by carrying out the two-colour line-breaking construction using a given stable tree $(\mathcal T, \mu)$ equipped with a sequence 
of i.i.d. leaves $(\Sigma_k, k \geq 0)$ sampled from $\mu$, and i.i.d.\ sequences of i.i.d.\ ordered $(\beta', 1-\beta')$-Chinese restaurant processes 
$(\widetilde{\Pi}_n^{(i,m)}, n \geq 1)$, $i \geq 1$, $m \geq 1$, cf. Section \ref{Sec22}. 

\begin{definition}[Labelled bead tree/string of beads] A pair $(x, \Lambda)$ is called a \textit{labelled bead} if $\Lambda \subset \mathbb N$ is an infinite label set. A weighted $\mathbb R$-tree $(\mathcal R, \mu_{\mathcal R})$ equipped with a point process $\mathcal P_{\mathcal R}=\sum_{i \geq 1} \delta_{(x_i, \Lambda_{i})}$ on some countable subset $\{x_i, i \geq 1\} \subset \mathcal R$, $x_i \neq x_j, i \neq j$, is called a \textit{labelled bead tree} if $(x_i, \Lambda_i)$ is a labelled bead for every $ i \geq 1$. If $(\mathcal R, \mu_{\mathcal R})$ is a string of beads we call  $(\mathcal R, \mu_{\mathcal R}, \mathcal P_{\mathcal R})$ a \textit{labelled string of beads}.
\end{definition}

We will also speak of \textit{labelled $(\alpha, \theta)$-strings of beads} for $\alpha \in (0,1)$, $\theta >0$, as induced by an ordered $(\alpha, \theta)$-Chinese restaurant process. Specifically, the label sets are the blocks $\Pi_{\infty,i}$, $i\ge 1$, of the limiting partition of $\mathbb N$, which we relabel by
$\mathbb N\setminus\{1\}$ using the increasing bijection $\mathbb{N}\rightarrow\mathbb{N}\setminus\{1\}$. The locations $X_i$ are the locations of the corresponding
atom of size $P_i$ on the string, $i\ge 1$. A Ford tree growth process of index $\beta' \in (0,1)$ as in Algorithm \ref{Ford} can be represented in terms of labelled $(\beta', 1-\beta')$-strings of beads $\widehat{\xi}_{m}$, $m \geq 1$, as follows \cite[Corollary 16]{1}.

\begin{prop}[Ford tree growth via labelled strings of beads] \label{labelledbeadford} For $\beta' \in (0,1)$, construct a sequence of labelled bead trees $(\mathcal F_m, \nu_m, \mathcal P_{m}, m \geq 1)$ as follows. 

\begin{itemize} \item[\rm 0.] Let $\widehat{\xi}_0=(\mathcal F_1, \nu_1, \mathcal P_1)$ be a labelled $(\beta', 1-\beta')$-string of beads with label set $\mathbb{N}\setminus\{1\}$.
\end{itemize}
Given $(\mathcal F_j, \nu_j, \mathcal P_{j})$, $1 \leq j \leq m$, with $\mathcal P_{m}=\sum_{i\ge 1}\delta_{X_{m,i},\Lambda_{m,i}}$, to construct 
$(\mathcal F_{m+1}, \nu_{m+1}, \mathcal P_{m+1})$,
\begin{itemize} \item[\rm 1.-2.] select the unique $X_{m,i} \in \mathcal F_m$ such that $m+1 \in \Lambda_{m,i}$; 
  \item[\rm 3.] to obtain $(\mathcal F_{m+1},\nu_{m+1},\mathcal{P}_{m+1})$, remove $\nu_m(X_{m,i})\delta_{X_{m,i}}$ from $\nu_m$ and $\delta_{(X_{m,i},\Lambda_{m,i})}$ 
    from $\mathcal P_m$; attach to $\mathcal F_m$ at ${X_{m,i}}$ an independent copy $\widehat{\xi}_{m}$ of $\widehat{\xi}_0$ with metric rescaled by 
    $\nu_m(X_{m,i})^{\beta'}$, mass measure by $\nu_m(X_{m,i})$, and label sets in $\widehat{\xi}_m$ relabelled by the increasing bijection 
    $\mathbb N\setminus\{1\} \rightarrow \Lambda_{m,i}\setminus\{m+1\}$. 
\end{itemize} 
Then the tree growth process $(\mathcal F_m, m \geq 1)$ is a Ford tree growth process of index $\beta' \in (0,1)$. 
\end{prop}

It will be useful to represent two-colour trees in the space $l^1(\mathbb N_0^2)$ as follows. We denote by $e_{a,b}$, $a,b\geq 0$, the unit coordinate 
vectors. We will use $e_{k,0}$, $k\geq 0$, to embed a given stable tree $(\mathcal T,d,\rho,\mu)$, using $e_{k,0}$ to embed $\Sigma_k$, $k\geq 0$. Indeed, from now on we 
assume $(\mathcal T,d,\rho,\mu)=(\mathcal T,d,0,\mu)\in\mathbb T^{\rm emb}_{\rm w}$ is this \em embedded stable tree\em, with embedded leaves $\Sigma_k$, $k\geq 0$. We 
will use $e_{m,i}$, 
$i\ge 1$, $m\ge 1$, to embed the $m$th branch of the $i$th red component, so the last step of Algorithm \ref{twocolourmass} is:
\begin{enumerate}\item[3.] let $([0,L_k],[0,L_kB_k^\prime],\mu_k^+)$ be an independent $\beta$-mixed string of beads in the notation of \eqref{twoclb}; denote by $M_k$ 
  the size (number of leaves) of $\mathcal R_k^{(I_k)}$; define the scale factor $c=\mu_k^*(J_k^*)$ and set
   \begin{align*}&\mathcal{T}_{k+1}^*:=\cT_k^*\cup\left(J_k^*+]0,L_kB_k^\prime c^\beta]e_{M_k+1,I_k}\right)\cup\left(J_k^*+L_kB_k^\prime c^\beta e_{M_k+1,I_k}+]0,L_k(1-B_k^\prime)c^\beta]e_{k+1,0}\right)\\
      &\mathcal{R}_{k+1}^{(I_k)}:=\mathcal{R}_k^{(I_k)}\!\cup\!\left(J_k^*+]0,L_kB_k^\prime c^\beta]e_{M_k+1,I_k}\right),\quad \cR_{k+1}^{(i)}\!:=\!\cR_k^{(i)}, i \!\neq\! I_k,\qquad \mu_{k+1}^*:=\mu_k^*-c\delta_{J_k^*}+\lambda_k^+\\
    &\mbox{where }\begin{array}{rll}\lambda_k^+(J_k^*+]c^\beta s,c^\beta t]e_{M_k+1,I_k})=&\!\!\!c\mu_k^+(]s,t]),\qquad &0\le s<t\le L_kB_k^\prime,\\[0.2cm]
     \lambda_k^+(J_k^*+L_kB_k^\prime c^\beta+]c^\beta s,c^\beta t]e_{M_k+1,I_k})=&\!\!\!c\mu_k^+(L_kB_k^\prime+]s,t]),\qquad &0\le s<t\le L_k(1-B_k^\prime).\end{array}
   \end{align*}  
\end{enumerate}
We will now formulate a modification of Algorithm \ref{twocolourmass} starting from a given stable tree. 
Let $(\mathcal T, \mu)$ be a stable tree of index $\beta \in (0,1/2]$ and $(\Sigma_k,k \geq 0)$ an i.i.d. sequence of leaves sampled from $\mu$. Consider the 
sequence of reduced weighted $\mathbb R$-trees $(\mathcal T_k, \mu_k, k \geq 0)$ where $\mu_k$ captures the masses of the connected components of 
$\mathcal T \setminus \mathcal T_k$ projected onto $\mathcal T_k$ as in \eqref{findimmarg1}. Let $(v_i, i \geq 1)$ be the sequence of branch points of $\mathcal T$ in 
order of appearance in $(\mathcal T_k, k \geq 0)$, and denote by 
$(\mathcal S_j^{(i)}, j \geq 1)$ 
the subtrees of $\mathcal T \setminus \mathcal T_{k^{(i)}}$ rooted at $v_i$, $i \geq 1$, where $k^{(i)}=\inf\{k \geq 0\colon v_i \in \mathcal T_k\}$ and where indices are 
assigned in increasing order of least leaf labels $\min\{\ell\ge k^{(i)}\colon\Sigma_\ell\in\mathcal S_j^{(i)}\}$, $j\ge 1$. 
For $i,j \geq 1$, set $P_j^{(i)}:=\mu(\mathcal S_j^{(i)})$, \vspace{-0.1cm} 
\begin{equation}\label{pididef}
  D^{(i)} :=\lim_{n \rightarrow \infty} \left(1-\sum_{j\in[n]} P_j^{(i)}/P^{(i)}\right)^{1-\beta} (1-\beta)^{\beta-1}n^{\beta},\qquad \mbox{where }P^{(i)}:=\sum_{j \geq 1} P_j^{(i)}.\vspace{-0.1cm}
\end{equation} 
This yields an i.i.d. sequence of $(1-\beta)$-diversities $(D^{(i)}, i \geq 1)$ with $D^{(i)} \sim {\rm ML}(1-\beta, -\beta)$, cf. Theorem \ref{masspart} and 
\eqref{alphadiv}. 
In the following algorithm, we build i.i.d. Ford trees in the branch points of the stable tree $(\mathcal T, \mu)$ from i.i.d. labelled $(\beta', 1-\beta')$-strings of 
beads $\widehat{\xi}_k, k\geq 0$, for $\beta'=\beta/(1-\beta)$. To do so, we consider two separate mass measures: the measures $(\widehat{\mu}_k,k\ge 0)$, that equal  
$\mu$ on (shifted) subtrees of the stable tree, and the measures $\widehat{\nu}_k$ on the Ford trees, which, restricted to each Ford tree separately, play the role 
of the mass measures $\nu_{m}$, $m \geq 1$, in the construction in Proposition \ref{labelledbeadford}.

\begin{algorithm}[Algorithm \ref{twocolourmass} with subtrees from a given stable tree]  \label{depstruc1} 
We construct a sequence of weighted $\infty$-marked $\mathbb R$-trees  
$\big(\widehat{\mathcal T}_k,\big(\widehat{\mathcal R}_k^{(i)}, i \geq 1\big), \widehat{\mu}_k, \widehat{\nu}_k, 
 \big(\widehat{\Sigma}_n^{(k)}, n \geq 0 \big), k \geq 0\big)$ 
embedded in $l^1(\mathbb N_0^2)$, each equipped with an infinite leaf sequence $(\widehat{\Sigma}_n^{(k)}, n \geq 0)$ and an additional finite measure $\widehat{\nu}_k$ as follows.

\begin{itemize}

\item[0.] Let $(\widehat{\mathcal T}_0, (\widehat{\mathcal R}_0^{(i)}, i \geq 1), \widehat{\mu}_0, \widehat{\nu}_0,  (\widehat{\Sigma}_n^{(0)}, n \geq 0 )) = (\mathcal T, (\{\rho\}, i \geq 1), \mu, 0, (\Sigma_n, n \geq 0))$ be a stable tree.
\end{itemize}
Given $(\widehat{\mathcal T}_j,(\widehat{\mathcal R}_j^{(i)}, i \geq 1), \widehat{\mu}_j, \widehat{\nu}_j, (\widehat{\Sigma}_n^{(j)}, n \geq 0))$, $0 \leq j \leq k$,
let $\widehat{r}_k=\#\{i\ge 1\colon\cR_k^{(i)}\neq\{\rho\}\}$;
\begin{itemize}
\item[1.-2.] let $\widehat{J}_k \in \widehat{\mathcal T}_k$ be the closest point to the leaf $\widehat{\Sigma}_{k+1}^{(k)}$ in 
  $\mathcal R(\widehat{\mathcal T}_k, \widehat{\Sigma}_1^{(k)}, \ldots, \widehat{\Sigma}_{k}^{(k)})$; if $\widehat{J}_k\in\widehat{\cR}_k^{(i)}$ for some $i\in[\widehat{r}_k]$, 
  set $I_k=i$, otherwise let $I_k=\widehat{r}_k+1$; denote by $M_k\ge 0$ the size of $\cR_k^{(I_k)}$;
\item[3.] let $\widehat{\xi}_k$ be an independent labelled $(\beta', 1-\beta')$-string of beads; if $M_k\ge 1$, define the scale factor $\widehat{c}=\widehat{\nu}_k(\widehat{J}_k)$, otherwise set $\widehat{c}=1$; write as $([0,K_k],\nu_k,\sum_{j\ge 1}\delta_{(X_{k,j},\Lambda_{k,j})})$ 
  the string of beads $\widehat{\xi}_k$ with metric rescaled by $\widehat{c}^{\beta^\prime}(P^{(I_k)})^{\beta} (D^{(I_k)})^{\beta'}$ and mass measure rescaled by $\widehat{c}$, where $P^{(I_k)}$ and $D^{(I_k)}$ are as in (\ref{pididef});
  denote by $\mathcal S_{k,j}$, $j\in\{0,1,2,\ldots;\infty\}$, the 
  connected components of $\widehat{\mathcal T}_k\setminus\{\widehat{J}_k\}$, where $\mathcal S_{k,\infty}$ contains the root and the other components are ordered by least 
  label; let $X_{k,0}:=K_k$ and set
  $$\widehat{\mathcal T}_{k+1}:=\mathcal S_{k,\infty}\cup\mathcal S_{k,0}\cup\left(\widehat{J}_k+[0,K_k]e_{M_{k+1},I_k}\right)\cup\bigcup_{j\ge 0}\left(X_{k,j}e_{M_{k+1},I_k}+\mathcal S_{k,j+1}\right).$$
  if $M_k=0$, let $\widehat{\mathcal R}_{k+1}^{(I_k)}=\widehat{J}_k+[0,K_k]e_{M_{k+1},I_k}$, otherwise add this shifted string to $\widehat{\mathcal R}_{k}^{(I_k)}$
  to form $\widehat{\mathcal R}_{k+1}^{(I_k)}$; retain the other
  marked components, just shifted by the appropriate $X_{k,j}e_{M_{k+1},I_k}$ if $\widehat{\mathcal R}_k^{(i)}\subset\mathcal S_{j,k}$. Finally, let $\widehat{\mu}_{k+1}$ 
  denote the mass measure obtained from $\widehat{\mu}_k$ by appropriate shifting, and similarly for $\widehat{\nu}_{k+1}$, just with $\nu_k$ shifted onto
  $\widehat{J}_k+[0,K_k]e_{M_{k+1},I_k}$ replacing $\widehat{\nu}_k(\widehat{J}_k)\delta_{\widehat{J}_k}$. 
\end{itemize}

\end{algorithm}


\begin{remark} \label{pidi} Note that the scaling factor $(C^{(i)})^{-1}:=(P^{(i)})^{\beta} (D^{(i)})^{\beta/(1-\beta)}$ can be rewritten as 
$$(C^{(i)})^{-1}=\lim_{n \rightarrow \infty} \left( P^{(i)}- \sum_{j\in[n]} P_j^{(i)}\right)^{\beta}(1-\beta)^{-\beta}n^{\beta^2/(1-\beta)} =\lim_{n \rightarrow \infty} \left( \sum_{j\geq n+1} P_j^{(i)}\right)^{\beta}(1-\beta)^{-\beta}n^{\beta^2/(1-\beta)},$$
or, alternatively, using \eqref{alphadiv}, as $ (C^{(i)})^{-1} = \lim \limits_{j \rightarrow \infty} \left(j \Gamma(\beta)\right)^{\beta/(1-\beta)}  \left(P_j^{(i)\downarrow}\right)^{\beta}.$
\end{remark}

The following result follows directly from the construction in Algorithm \ref{depstruc1} and Proposition \ref{labelledbeadford}.

\begin{prop} In the setting of Algorithm \ref{depstruc1}, there exists a sequence of i.i.d. Ford CRTs $(\widehat{\mathcal F}_i, i \geq 1)$ of index $\beta'=\beta/(1-\beta)$ which is independent of the stable tree $(\mathcal T, \mu)$ such that, for all $i \geq 1$, 
$$\lim_{k \rightarrow \infty} \widehat{\mathcal R}_k^{(i)}=:\widehat{\mathcal R}^{(i)}= \left(C^{(i)}\right)^{-1} \widehat{\mathcal F}_i \quad \text{ a.s. w.r.t. to the Gromov-Hausdorff topology}.$$
\end{prop}
 
We will now prove that the sequence of reduced $\infty$-marked $\mathbb R$-trees constructed in Algorithm \ref{depstruc1} and the sequence of trees constructed in Algorithm \ref{twocolourmass} are equal in distribution.

\begin{prop} \label{algsame} Let $(\widehat{\mathcal T}_k, (\widehat{\mathcal R}_k^{(i)}, i \geq 1), \widehat{\mu}_k, (\widehat{\Sigma}_n^{(k)},n\ge 0), k \geq 0)$ and 
  $(\mathcal T_k^*, (\mathcal R_k^{(i)}, i \geq 1), \mu_k^*, k \geq 0)$ be as in 
  Algorithms \ref{depstruc1} and \ref{twocolourmass}, respectively,
  $\widehat{\pi}_k\colon\widehat{\mathcal T}_k\rightarrow \mathcal R(\widehat{\mathcal T}_k, \widehat{\Sigma}_0^{(k)}, \ldots, \widehat{\Sigma}_k^{(k)})$ the
  projection as in \eqref{pm}. Then, 
  \begin{equation} 
    \left(\mathcal R\left(\widehat{\mathcal T}_k, \widehat{\Sigma}_0^{(k)}, \ldots, \widehat{\Sigma}_k^{(k)}\right),\left(\widehat{\mathcal R}_k^{(i)}, i \geq 1\right), 
		  (\widehat{\pi}_k)_*\widehat{\mu}_k, k \geq 0 \right) 
    \,{\buildrel d\over =}\,  \left({\mathcal T}_k^*,\left({\mathcal R}_k^{(i)}, i \geq 1\right), \mu_k^*, k \geq 0 \right). \label{algsameq} 
  \end{equation}
  Furthermore, $(P_j^{(x)}, j \geq 1)$ with $P_j^{(x)}:=\widehat{\mu}_k\left(\mathcal S_j^{(x)}\right)/\sum_{\ell \geq 1} \widehat{\mu}_k\left(\mathcal S_\ell^{(x)}\right)$, $j\ge 1$, are i.i.d. 
  ${\rm GEM}(1-\beta, -\beta)$ for all $x \in \mathcal R(\widehat{\mathcal T}_k, \widehat{\Sigma}_0^{(k)}, \ldots, \widehat{\Sigma}_k^{(k)})$ with 
  $(\widehat{\pi}_k)_*\widehat{\mu}_k(x) >0$, where $(\mathcal S_j^{(x)}, j \geq 1)$ are the connected components of 
  $\widehat{\mathcal T}_k \setminus  \mathcal R(\widehat{\mathcal T}_k, \widehat{\Sigma}_0^{(k)}, \ldots, \widehat{\Sigma}_k^{(k)})$ rooted at 
  $x \in  \mathcal R(\widehat{\mathcal T}_k, \widehat{\Sigma}_0^{(k)}, \ldots, \widehat{\Sigma}_k^{(k)})$, ranked in increasing order of least leaf labels.
\end{prop}

The following is a direct consequence of Proposition \ref{algsame}.

\begin{corollary} \label{algsamecor} In Algorithm \ref{twocolourmass}, the tree growth processes 
  $ \left(C^{(i)} \mathcal R_{k_m^{(i)}}^{(i)}, m \geq 1 \right)$, $i \geq 1$, are i.i.d.\ Ford tree growth processes of index $\beta'=\beta/(1-\beta)$ independent of 
  the stable tree $({\mathcal T}, {\mu})=\lim_{k \rightarrow \infty} (\widetilde{\mathcal T}_k, \widetilde{\mu}_k)$ of Corollary \ref{convstable}, where the
  scaling factors $(C^{(i)})^{-1}=(P^{(i)})^{\beta} (D^{(i)})^{\beta/(1-\beta)}$, $i \geq 1$, are as in Remark \ref{pidi}.  
\end{corollary}

To prove Proposition \ref{algsame}, we will need a strong form of coagulation-fragmentation duality. 

\begin{lemma} \label{coag1} Let $P=(P_i,i\ge1)\sim{\rm GEM}(\alpha,\theta)$ with $\alpha$-diversity $S$, and 
  $\widehat{\xi}=([0,\widehat{K}],\widehat{\mu},\widehat{\cP}\!=\!\sum_{j\ge 1}\!\delta_{(X_j,\widehat{\Lambda}_j)})$ an 
  independent labelled $(\beta',\theta/\alpha)$-string of beads. Use $([0,\widehat{K}],\widehat{\mu},\widehat{\cP})$ to coagulate $(P_i,i\ge1)$ into 
  $\mu(\{X_j\}):=\sum_{i\in\widehat{\Lambda}_{j}}P_i$, with relative part sizes $Q_m^{(j)}:=P_{\pi_j(m)}/\mu(\{X_j\})$, $m\ge 1$, labelled by the increasing bijection
  $\pi_j\colon\bN\rightarrow\widehat{\Lambda}_j$, $j\ge 1$. Then 
  \begin{itemize}\item the string of beads $([0,S^{\beta'}\widehat{K}],\mu)$ is an $(\alpha\beta',\theta)$-string of beads,
	\item the sequence of fragments $(Q_m^{(j)},m\ge 1)$ has a ${\rm GEM}(\alpha,-\alpha\beta')$ distribution, for each $j\ge 1$,
    \item the string $([0,S^{\beta'}\widehat{K}],\mu)$ and the fragments $(Q_m^{(j)},m\ge 1)$ of $\mu(\{X_j\})$, $j\ge 1$, are independent.
  \end{itemize} 
\end{lemma}

\begin{proof} This is an enriched instance of coagulation-fragmentation duality, see e.g. \cite[Section 5.5]{5}. We use a combinatorial approach, with notation
  $(x)_{n\uparrow\gamma}=x(x+\gamma)\cdots(x+(n-1)\gamma)$ and using known distributions of (ordered and unordered) Chinese restaurant partitions \cite{5,1}. 
  Fix $n\ge 1$. 

  What is the probability that an ordered $(\beta',\theta/\alpha)$-coagulation groups the tables of an unordered $(\alpha,\theta)$-Chinese restaurant partition of $[n]$
  into $m$ groups $(n_{1,1},\ldots,n_{1,k_1}),\ldots,(n_{m,1},\ldots,n_{m,k_m})$? If we denote by $\ell$ the number of new right-most groups opened, and $(\gamma)_{j\uparrow\delta}:=\gamma(\gamma+\delta)\cdots(\gamma+(j-1)\delta)$, then it is
  $$\frac{(\theta+\alpha)_{k_1+\cdots+k_m-1\uparrow\alpha}\prod_{i\in[m]}\prod_{j\in[k_i]}(1-\alpha)_{n_{ij}\uparrow 1}}
         {(1+\theta)_{n-1\uparrow 1}}
    \frac{(\beta')^{m-\ell-1}(\theta/\alpha)^\ell\prod_{i\in[m]}(1-\beta')_{k_i-1\uparrow 1}}
         {(1+\theta/\alpha)_{k_1+\cdots+k_m-1\uparrow1}}.$$
  What is the probability that an unordered $(\alpha,-\alpha\beta')$-fragmentation of an ordered $(\alpha\beta',\theta)$-Chinese restaurant partition of $[n]$ yields 
  $m$ tables further split into $(n_{1,1},\ldots,n_{1,k_1}),\ldots,(n_{m,1},\ldots,n_{m,k_m})$? If we denote by $\ell$ the number of new right-most tables, then it is 
  $$\frac{(\alpha\beta')^{m-\ell-1}\theta^\ell\prod_{i\in[m]}(1-\alpha\beta')_{n_{i,1}+\cdots+n_{i,k_i}-1\uparrow 1}}
	     {(1+\theta)_{n-1\uparrow 1}}
	\prod_{i\in[m]}\frac{(\alpha-\alpha\beta')_{k_i-1\uparrow\alpha}\prod_{j\in[k_i]}(1-\alpha)_{n_{ij}-1\uparrow 1}}
	                  {(1-\alpha\beta')_{n_{i,1}+\cdots+n_{i,k_i}-1\uparrow 1}}.$$
  Elementary cancellations show that these two expressions are equal for all $n\ge 1$. Since these structured partitions can be constructed in a consistent way, as 
  $n\ge 1$ varies, the statement of the lemma merely records different aspects of the limiting arrangement, either asymptotic frequencies in size-biased order of least
  labels coagulated by a labelled strings of beads, or respectively a string of beads with blocks further fragmented, with fragments in size-biased order of least labels.
\end{proof}

The following result can be proved using the same method.
\begin{lemma} \label{coag2}
  Let $P=(P_i,i\ge1)\sim{\rm GEM}(\alpha,\theta)$ and, for $\alpha \in (0,1), \theta >0$, let  $\widehat{\Lambda}=(\widehat{\Lambda}_1,\ldots,\widehat{\Lambda}_r)$ be an 
  independent ${\rm Dirichlet}(\theta_1/\alpha,\ldots,\theta_r/\alpha)$ partition of $\mathbb{N}$ with $\sum_{i\in[r]}\theta_i=\theta$. Use 
  $(\widehat{\Lambda}_1,\ldots,\widehat{\Lambda}_r)$ to coagulate $(P_i,i\ge1)$ into $R_j:=\sum_{i\in\widehat{\Lambda}_{j}}P_i$, with
  relative part sizes $Q_m^{(j)}:=P_{\pi_j(m)}/R_j$, $m\ge 1$, labelled by the increasing bijection $\pi_j\colon\bN\rightarrow\widehat{\Lambda}_j$, $j\in[r]$.  
  Then 
  \begin{itemize}\item the vector $(R_1,\ldots,R_r)$ of aggregate masses has a $ {\rm Dirichlet}(\theta_1,\ldots,\theta_r)$ distribution,
    \item the sequence of fragments $(Q_m^{(j)},m\ge 1)$ has a ${\rm GEM}(\alpha,\theta_j)$ distribution, for each $j\in[r]$,
    \item the vector $(R_1,\ldots,R_r)$ and the fragments $(Q_m^{(j)},m\ge 1)$ of $R_j$, $j \in [r]$, are independent.  
  \end{itemize} 
\end{lemma}

In the context of Algorithm \ref{depstruc1}, it is useful to adopt the following terminology. Consider a branch point of the reduced stable tree and the 
associated Ford tree. The (unordered) $(\alpha,\theta)$-Chinese restaurant behind $P$ partitions the total branch point mass into subtrees (unmarked tables) which 
carry leaf labels of the stable tree (unmarked customers). A transition $k\rightarrow k+1$ of the algorithm spreads the subtrees over a new string of beads of the Ford
tree. The ordered structures $\widehat{\xi}$ and $\widehat{\Lambda}$, respectively, partition the leaf labels of the Ford tree (marked customers) into marked tables 
(whose sizes are captured by $\widehat{\nu}$ for each marked component separately). The coagulation takes subtrees as marked customers and so coagulates those unmarked 
tables that are listed in the same marked table to form a partition of unmarked customers (leaves of the stable tree) into marked tables. 
The further partition into unmarked tables within each marked table is then a fragmentation of the unmarked customers (leaf labels of the stable tree).  

\begin{proof}[Proof of Proposition \ref{algsame}] As the families of weighted discrete $\infty$-marked $\mathbb R$-trees in \eqref{algsameq}, suitably represented, are consistent and at step $k$ uniquely determine the trees at steps $0, \ldots, k-1$, it suffices to show that for fixed $k \geq 0$ 
\begin{equation}  \left(\mathcal R\left(\widehat{\mathcal T}_k, \widehat{\Sigma}_0^{(k)}, \ldots, \widehat{\Sigma}_k^{(k)}\right),\left(\widehat{\mathcal R}_k^{(i)}, i \geq 1\right), (\widehat{\pi}_k)_*\widehat{\mu}_k \right) \,{\buildrel d\over =}\,  \left({\mathcal T}_k^*,\left({\mathcal R}_k^{(i)}, i \geq 1\right), \mu_k^* \right). \label{algsameqq} \end{equation}
We will prove \eqref{algsameqq} by induction on $k$, showing that the LHS follows the characterisation of the distribution of the two-colour tree on the RHS given in Proposition \ref{starstrings}. The case $k=0$ follows from Proposition \ref{betastring} in combination with Corollary \ref{GEMS}.

For general $k \geq 0$, we obtain the shape $T_k$ of a stable tree $\mathcal T_k$ reduced to the first $k+1$ leaves from the stable tree growth processes with masses naturally embedded in Algorithm \ref{depstruc1}, and conditionally given its shape with $\ell$ branch points $v_1, \ldots, v_\ell$ of degrees $d_1, \ldots, d_\ell$, a Dirichlet$(\beta, \ldots, \beta, m_1+(1-2\beta), \ldots, m_\ell+(1-2\beta))$ mass split between edges and branch points as in Proposition \ref{masssel2} where $m_i:=d_i-2, i \in [\ell]$. We further obtain rescaled independent $(\beta, \beta)$-strings of beads on the branches of the stable tree, i.e.\ the unmarked branches of $\mathcal R(\widehat{\mathcal T}_k, \widehat{\Sigma}_0^{(k)}, \ldots, \widehat{\Sigma}_k^{(k)})$, cf. Theorem \ref{masssel2} and Proposition \ref{betastring}.

From the stick-breaking representation \eqref{v1v2} of GEM($\cdot, \cdot$) and Algorithm \ref{masssel}, the relative masses of the subtrees of $\mathcal T \setminus \mathcal T_k$ rooted at $v_i$ indexed in increasing order of smallest leaf labels form a vector with distribution GEM$(1-\beta, m_i(1-\beta)+(1-2\beta))$, independently for each branch point, $i\in[\ell]$. 

From the independent Ford tree growth processes via labelled strings of beads built from the $(\widehat{\xi}_k,k\ge 0)$ in Algorithm \ref{depstruc1}, we have the shapes of conditionally independent Ford trees with $m_1, \ldots, m_\ell$ leaves, and for each Ford tree conditionally given the shape, independently a ${\rm Dirichlet}(\beta^\prime,\ldots,\beta^\prime,1-\beta^\prime,\ldots,1-\beta^\prime)$ partition of $\mathbb{N}$ obtained by relabelling the edge-partition of labels $\mathbb{N}\setminus[m_i]$ by the increasing bijection $\mathbb{N}\setminus[m_i]\rightarrow\mathbb{N}$. These partitions are further split on each internal edge by a labelled $(\beta^\prime,\beta^\prime)$-string of beads, and on each external edge by a labelled $(\beta^\prime,1-\beta^\prime)$-string of beads, again all labelled by $\mathbb{N}$ and obtained by increasing bijections from $\mathbb N$ to the label sets of the edges.

We apply Lemma \ref{coag2} with $P$ as the ${\rm GEM}(1-\beta,m_i(1-\beta)+(1-2\beta))$ split into further subtree masses of the $i$th marked component
  and $\widehat{\Lambda}$ as the ${\rm Dirichlet}(\beta^\prime,\ldots,\beta^\prime,1-\beta^\prime,\ldots,1-\beta^\prime)$ partition of marked Ford labels in the $i$th component. 
  We note that we eventually place subtrees in their size-biased order in $P$ into the further Ford leaves of the $i$th component. Therefore, the coagulation of Lemma
  \ref{coag2} produces a ${\rm Dirichlet}(\beta,\ldots,\beta,1-2\beta,\ldots,1-2\beta)$ mass split onto the edges and independent ${\rm GEM}(1-\beta,\beta)$ and ${\rm GEM}(1-\beta,1-2\beta)$ sequences of fragments of these edge masses.

 We apply Lemma \ref{coag1} for each edge, with $P$ as the ${\rm GEM}(1-\beta,\beta)$ or ${\rm GEM}(1-\beta,1-2\beta)$ sequence of fragments and with the labelled
  $(\beta^\prime,\beta^\prime)$- or $(\beta^\prime,1-\beta^\prime)$-string of beads as $\widehat{\xi}$, independent. Again, we note that we eventually place subtrees in
  their size-biased order in $P$ according to the positions of the labels in the labelled string of beads. Therefore, the coagulation of Lemma \ref{coag1} produces a mass
  split according to a $(\beta,\beta)$- or $(\beta,1-2\beta)$-string of beads, respectively.

  We obtain two-colour shapes as needed for the distribution of the RHS of \eqref{algsameqq} characterised in Proposition \ref{starstrings}. Conditionally given the two-colour shape, we obtain independent Dirichlet splits onto edges that combine to a
  ${\rm Dirichlet}(\beta,\ldots,\beta,1-2\beta,\ldots,1-2\beta)$ split, with parameters $\beta$ for unmarked and marked internal edges and $1-2\beta$ for marked external edges.
  Again conditionally given the two-colour shape, we obtain, independently of the Dirichlet splits, for each unmarked and marked internal edge an independent 
  $(\beta,\beta)$-string of beads, and for each marked external edge a $(\beta,1-2\beta)$-string of beads. If we arrange the edges in the tree shape suitably by depth first
  search and sort the Dirichlet vectors and the vectors of strings accordingly, their joint conditional distribution does not depend on the two-colour shape, so the 
  two-colour shape, the overall Dirichlet split and the strings of beads are jointly independent. 

  Finally, Algorithm \ref{depstruc1} scales the strings of beads. We can write $(P^{(i)})^\beta(D^{(i)})^{\beta'}=(D_{m_i}^{(i)})^{\beta'}(P^{(i)}_{(m_i)})^\beta$,
  where $D_{m_i}^{(i)}$ is the $(1-\beta)$-diversity of $P$ in the application of Lemma \ref{coag2} above, independent of the total mass 
  $P^{(i)}_{(m_i)}=\sum_{j\ge m_i+1}P^{(i)}_j$ on the $i$th component, which is further split according to the Dirichlet distribution found above, as required.
  Altogether, the distribution is the same as in Proposition \ref{starstrings}.
\end{proof}

\begin{proof}[Proof of Theorem \ref{fordembb} and \eqref{subfordgrowth} in Theorem \ref{Mainresult1}] This is a direct consequence of Proposition \ref{algsame} and Corollary \ref{algsamecor}.\end{proof}

In Theorem \ref{fordembb}, we identified the tree growth processes $(\mathcal R_k^{(i)}, k \geq 1)$, $i \geq 1$, as consistent families of tree growth processes which obey the growth rules of a Ford tree growth process of index $\beta'=\beta/(1-\beta)$. Rescaling these processes to obtain i.i.d. sequences of Ford trees requires knowledge of the scaling factor which is incorporated in the limiting stable tree. It is, however, possible to approximate this scaling factor using the tree constructed up to step $k$ only. We are further able to obtain i.i.d. marked subtree growth processes obeying the Ford growth rules (but with wrong starting lengths) applying suitable scaling.

\begin{theorem}[Embedded Ford trees]\label{embford} Let $(\mathcal T_k^*, (\mathcal R_k^{(i)}, i \geq 1), \mu_k^*, k \geq 0)$ as in Algorithm \ref{twocolourmass}.
 \begin{itemize} \item[{\rm (i)}] The normalised tree growth processes in the components, with projected $\mu$-masses, are i.i.d.: 
\begin{equation} \left(\mathcal G_m^{(i)}, \mu_m^{(i)}, m\geq 1\right)= \left( \mu_{k_1^{(i)}}^* \left(\mathcal R_{k_1^{(i)}}^{(i)}\right)^{-\beta}\mathcal R_{k_m^{(i)}}^{(i)}, \mu_{k_1^{(i)}}^* \left(\mathcal R_{k_1^{(i)}}^{(i)}\right)^{-1} \mu_{k}^* \restriction_{\mathcal R_k^{(i)}}, m \geq 1\right), \quad i \geq 1. \label{embford1} \end{equation}
 
\item[{\rm (ii)}] The processes $\big( \mu_{k_1^{(i)}}^* \big(\mathcal R_{k_1^{(i)}}^{(i)}\big)^{-\beta}\mathcal R_{k_m^{(i)}}^{(i)}, m \geq 1\big)$, without 
  $\mu$-masses are i.i.d. Ford tree growth processes of index $\beta'=\beta/1-\beta$ as in Algorithm \ref{Ford}, $i\geq 1$, but starting from 
  {\rm ML}$(\beta, 1-2\beta)$, not {\rm ML}$(\beta', 1-\beta')$. 
\item[{\rm (iii)}] For $i \geq 1$, define 
  $C_m^{(i)}:=\left(1-\beta\right)^{\beta} m^{-\beta^2/(1-\beta)} \mu_{k_m^{(i)}}^*\big(\mathcal R^{(i)}_{k_m^{(i)}}\big)^{-\beta}$. The processes 
  $(C_m^{(i)}\mathcal R_{k_m^{(i)}}^{(i)}, m \geq 1)$ with scaling constant depending on $m$, $i \geq 1$, are i.i.d., $\lim_{m \rightarrow \infty} C_m^{(i)}= \left(H^{(i)}\right)^{-\beta/(1-\beta)}\mu_{k_1^{(i)}}^*\big(\mathcal R_{k_1^{(i)}}^{(i)}\big)^{-\beta}$ a.s., where $H^{(i)} \sim {\rm ML}(1-\beta, 1-2\beta)$, and 
$\lim_{m \rightarrow \infty} C_m^{(i)} \mathcal R_{k_m^{(i)}}^{(i)}=\mathcal F^{(i)}$ a.s. 
in the Gromov-Hausdorff topology where $(\mathcal F^{(i)}, i \geq 1)$ are i.i.d. Ford CRTs of index $\beta'$. 
\end{itemize}
\end{theorem}

\begin{proof} See Section \ref{sec71} in the appendix. \end{proof}


\section{Continuum tree asymptotics} \label{sec5}

In this section, we use embedding to show the convergence of the constructions in Theorems \ref{Mainresult2} and \ref{branchrepl}.

\subsection{Embedding of the two-colour line-breaking construction into a binary compact CRT} \label{Embedding}

In \cite{RW} we constructed CRTs recursively based on recursive distribution equations as reviewed by Aldous and Bandyopadhyay 
\cite{14}. This method applied to a $\beta$-mixed string of beads yields a compact CRT $(\mathcal T^*, \mu^*)$ in which we can embed the two-colour line-breaking 
construction. Let us briefly recall the recursive construction of $(\mathcal T^*, \mu^*)$ from {\cite[Proposition 4.12]{RW}} including some useful notation. We only
outline the constructions without going into the mathematical details for which we refer to \cite{RW}. 

For $\beta \in (0,1/2]$, consider a sequence of independent strings of beads $({\xi}_{\mathbf i}$, $\mathbf{i} \in \mathbb{U})$, 
$${\xi}_{\mathbf i}=\left([0,{L}_{\mathbf i}], \sum_{j \geq 1} {P}_{\mathbf{i}j}\delta_{{X}_{\mathbf{i}j}}\right), \quad {\mathbf i} \in \mathbb{U},$$
where ${\xi}_{\varnothing}$ is a $(\beta, \beta)$-string of beads independent of the $\beta$-mixed strings of beads  
${\xi}_{\mathbf i}, \mathbf{i} \in \mathbb{U}\setminus \{\varnothing\}$, and $\mathbb U:=\bigcup_{n \geq 0} \mathbb N^n$ is the infinite Ulam-Harris tree. 
Let $(\check{\mathcal T}_0, \check{\mu}_0)= {\xi}_{\varnothing}$, and for $n \geq 0$, conditionally given $(\check{\mathcal T}_n, \check{\mu}_n)$ with 
$\check{\mu}_n=\sum_{\mathbf{i}j\in\bN^{n+1}}\check{P}_{\mathbf{i}j}\delta_{\check{X}_{\mathbf{i}j}}$, attach to each $\check{X}_{\mathbf{i}j}$ an isometric copy of the 
string of beads $\xi_{\mathbf{i}j}$ 
\begin{itemize} 
  \item with metric rescaled by $\check{\mu}_n(\check{X}_{\mathbf{i}j})^\beta$, and mass measure rescaled by $\check{\mu}_n(\check{X}_{\mathbf{i}j})$, 
  \item so that the atom ${P}_{\mathbf{i}jk}\delta_{{X}_{\mathbf{i}jk}}$ of $\xi_{\mathbf{i}j}$ is scaled to become an atom of $\check{\mathcal T}_{n+1}$ denoted by 
    $\check{P}_{\mathbf{i}jk}\delta_{\check{X}_{\mathbf{i}jk}}$, $k\geq 1$,
\end{itemize}
for all ${\mathbf i}j \in \mathbb N^{n+1}$ respectively. Denote the resulting tree by $(\check{\mathcal T}_{n+1}, \check{\mu}_{n+1})$.

By construction, $(\check{\mathcal T}_{n}, \check{\mu}_{n})$ only carries mass in the points $\check{X}_{{\mathbf i}j}, {\mathbf{i}j} \in \mathbb N^{n+1}$, i.e. $\check{\mu}_{n}(\check{\mathcal T}_{n} \setminus \check{\mathcal T}_{n-1})=0$ for $n \geq 0$. Note that, for any $\check{X}_{i_1 i_2 \cdots i_{n+1}} \in \check{\mathcal T}_{n}$, $n \geq 0$, 
$$ \check{\mu}_{n} \left(\check{X}_{i_1 i_2 \cdots i_{n+1}}\right)=\check{P}_{i_1i_2\cdots i_{n+1}}={P}_{i_1} {P}_{i_1 i_2} \cdots {P}_{i_1 i_2 \cdots i_{n+1}}.$$ 
This induces a recursive description of the trees $(\check{\mathcal T}_{n}, \check{\mu}_{n}, n \geq 0)$ via the strings of beads 
$({\xi}_{\mathbf i}, \mathbf{i} \in \mathbb{U})$.

\begin{theorem}[{\cite[Proposition 4.12]{RW}}] \label{reccon} Let $\beta \in (0,1/2]$ and $(\check{\mathcal T}_{n}, \check{\mu}_n, n \geq 0)$ as above. Then there exists a compact CRT $(\mathcal T^*, \mu^*)$ such that 
$$ \lim \limits_{n \rightarrow \infty}\left(\check{\mathcal T}_n, \check{\mu}_n\right)=\left(\mathcal T^*, \mu^*\right) \text{ a.s. } $$ with respect to the Gromov-Hausdorff-Prokhorov topology.
\end{theorem}

We will show that the increasing sequence $(\mathcal T_k^*, k \geq 0)$ of compact $\mathbb R$-trees from Algorithm \ref{twocolourmass} converges a.s. to a tree with the 
same distribution as $\mathcal T^*$. To do this and handle the marked components, we will embed the sequence of weighted $\infty$-marked $\mathbb R$-trees 
$(\mathcal T_k^*, (\mathcal R_k^{(i)}, i \geq 1), \mu_k^*, k \geq 0)$ into a given $(\mathcal T^*, \mu^*)$.  

%

Note that the strings of beads ${\xi}_{\mathbf i}$, $\mathbf{i} \in \mathbb U \setminus \{\varnothing\}$, are $\beta$-mixed strings of beads as used in Algorithm 
\ref{twocolourmass} but are not elements of the space of (equivalence classes of) weighted $1$-marked $\mathbb R$-trees $\mathbb T_{\rm w}^{[1]}$, as there is no marked 
component. As we would like to embed into $(\mathcal T^*, \mu^*)$ the two-colour line breaking construction which carries colour marks on $\beta$-mixed strings of beads, 
we need to determine $I_1 =[0,K_1]\subset I=[0,K]$ such that $(I, I_1, \lambda) \sim \nu_{\beta}^{[1]}$ given some ${\xi}=(I=[0,K], \lambda) \sim \nu_{\beta}$, where 
$\nu_{\beta}$ and $\nu_{\beta}^{[1]}$ were introduced at the beginning of Section \ref{sec4} as distributions on one-branch trees in $\bT_{\rm w}$ and $\bT_{\rm w}^{[1]}$,
respectively. The existence of the conditional distribution of the point of the colour 
change $K_1$ given ${\xi}$ is stated in the following lemma.

\begin{lemma}\label{condis} Let ${\xi} \sim \nu_\beta$. Then there exists a unique probability kernel $\kappa$ from $\mathbb T_{\rm w}$ to $\mathbb R$ such that
\begin{equation} \mathbb P\left(K_1 \in \cdot \lvert {\xi} \right)=\kappa \left({\xi}, \cdot\right) \quad \text{a.s.}. \label{conddistr} \end{equation} 
\end{lemma}

\begin{proof}
This is a special case of Theorem 6.3 in \cite{32}, since $\mathbb R$ is a Borel space.
\end{proof}

Given the weighted $\mathbb R$-tree $(\mathcal T^*, \mu^*)$, we will obtain a sequence of weighted $\infty$-marked $\mathbb R$-trees 
\begin{equation*} \left(\overline{\mathcal T}_k^*, \left(\overline{\mathcal R}_k^{(i)}, i \geq 1\right), \overline{\mu}_k^*, k \geq 0\right) 
\end{equation*} 
with the same distribution as $({\mathcal T}_k^*, ({\mathcal R}_k^{(i)}, i \geq 1), {\mu}_k^*, k \geq 0)$ as an increasing sequence of subsets 
$\overline{\mathcal T}_k^* \subset \mathcal T^*$, $k \geq 0$, where the mass measure $\overline{\mu}_k^*$ captures the masses of the connected components of 
$\mathcal T^* \setminus \overline{\mathcal T}_k^*$ projected onto $\overline{\mathcal T}_k^*$, $k \geq 0$. The recursive structure 
${\xi}_{\mathbf i}, {\mathbf i} \in \mathbb U$, provides the i.i.d.\ strings of beads needed in Algorithm \ref{twocolourmass}, which the colour change kernel 
\eqref{conddistr} turns into i.i.d. $1$-marked strings of beads.

\begin{algorithm}[Two-colour embedding] \label {twocemb} \rm Let $\beta\in(0,1/2]$. We embed into the tree $({\cT}^*,{\mu}^*)$ of Theorem \ref{reccon} weighted $\infty$-marked $\bR$-trees $(\overline{\cT}_k^*,(\overline{\cR}_k^{(i)},i\ge 1),\overline{\mu}_k^*)$, $k\ge 0$, as follows.
  \begin{enumerate}\setcounter{enumi}{-1}\item Let $(\overline{\cT}_0^*,\overline{\mu}_0^*)={\xi}_\varnothing$ be the initial $(\beta,\beta)$-string of beads; let $\overline{r}_0=0$ and $\overline{\cR}_0^{(i)}=\{\rho\}$, $i\ge 1$.
  \end{enumerate}
  Given $(\overline{\cT}_j^*,(\overline{\cR}_j^{(i)},i\ge 1),\overline{\mu}_j^*)$ with $\overline{\mu}_j^*=\sum_{x\in\overline{\cT}_j^*}\overline{\mu}_j^*(x)\delta_x$, $0\le j\le k$, let $\overline{r}_k=\#\{i\ge 1\colon\overline{\cR}_k^{(i)}\neq\{\rho\}\}$;  
  \begin{enumerate}\item select an edge $\overline{E}_k^*\subset\overline{\cT}_k^*$ with probability proportional to its mass $\overline{\mu}_k^*(\overline{E}_k^*)$; if 
      $\overline{E}_k^*\subset\overline{\cR}_k^{(i)}$ for some $i\in[\overline{r}_k]$, let $\overline{I}_k=i$; otherwise, i.e.\ if $\overline{E}_k^*\subset\overline{\cT}_k^*\setminus\bigcup_{i\in[\overline{r}_k]}\overline{\cR}_k^{(i)}$, let $\overline{r}_{k+1}\!=\!\overline{r}_k\!+\!1$, $\overline{I}_k\!=\!\overline{r}_{k+1}$;  
    \item if $\overline{E}_k^*$ is an external edge of $\overline{\cR}_k^{(i)}$, perform $(\beta,1-2\beta)$-coin tossing sampling on $\overline{E}_k^*$ to determine $\overline{J}_k^*\in \overline{E}_k^*$;  otherwise, i.e.\ if $\overline{E}_k^*\subset\overline{\cT}_k^*\setminus\bigcup_{i\in[\overline{r}_k]}\overline{\cR}_k^{(i)}$ or if $\overline{E}_k^*$ is an internal edge of $\overline{\cR}_k^{(i)}$, sample $\overline{J}_k^*$ from the normalised mass measure on $\overline{E}_k^*$;
    \item let $\mathbf{j}\in\mathbb U$ such that $\overline{J}_k^*=\check{X}_{\mathbf{j}}$ and
      $\overline{\mu}_k^*(\overline{J}_k^*)=\check{P}_{\mathbf{j}}$; sample
      a point $\overline{\Omega}_k$ from $\kappa(\xi_{\mathbf{j}},\cdot)$; to form $(\overline{\cT}_{k+1}^*,\overline{\mu}_{k+1}^*)$, 
      remove $\overline{\mu}_k^*(\overline{J}_k^*)\delta_{\overline{J}_k^*}$ from $\overline{\mu}_k^*$ and add to $\overline{\cT}_k^*$ the scaled copy of the string
      of beads $\xi_{\mathbf{j}}$ with $\overline{\Omega}_k$ embedded in $\cT^*$; set $\overline{\cR}_{k+1}^{(\overline{I}_k)}=\overline{\cR}_k^{(\overline{I}_k)}\cup[[\overline{J}_k^*,\overline{\Omega}_k]]$ and 
      $\overline{\cR}_{k+1}^{(i)}=\overline{\cR}_k^{(i)}$, $i\neq \overline{I}_k$. 
  \end{enumerate}
\end{algorithm}

The proof of the following statement can be found in the Appendix \ref{appen1}, together with similar proofs.

\begin{prop} \label{oversamedist} The sequences of trees constructed in Algorithm \ref{twocolourmass} and Algorithm \ref{twocemb}
have the same distribution, i.e. 
$({\mathcal T}_k^*,({\mathcal R}_k^{(i)}, i \geq 1), {\mu}_k^*, k \geq 0 )
\,{\buildrel d \over =}\,  (\overline{\mathcal T}_k^*, (\overline{\mathcal R}_k^{(i)}, i \geq 1), \overline{\mu}_k^*, k \geq 0).$
\end{prop}

\subsection{Convergence of two-colour trees, and the proof of Theorem \ref{Mainresult2}} 

Theorem \ref{mainresult} and Corollary \ref{algsamecor} demonstrate that the two-colour line-breaking construction naturally combines the stable tree growth process, and infinitely many rescaled subtree growth processes that build rescaled independent Ford CRTs. We can show that the tree growth process $(\mathcal T_k^*, k \geq 0)$ converges to a compact CRT with the same distribution as the CRT $(\mathcal T^*, \mu^*)$ constructed in the beginning of Section \ref{Embedding}, using the embedding of Algorithm \ref{twocemb} and Proposition \ref{oversamedist}.

\begin{prop}[Convergence of $(\mathcal T_k^*, \mu_k^*, k \geq 0)$] \label{propo2} Let $({\mathcal T}_k^*, {\mu}_k^*, k \geq 0)$ be the sequence of weighted $\mathbb R$-trees from Algorithm \ref{twocolourmass}. Then, there is a compact CRT $(\mathcal T^*, \mu^*)$ such that  \begin{equation}
\lim \limits_{k \rightarrow \infty} d_{\rm GHP} \left( \left({\mathcal T}_k^*, {\mathcal \mu}_k^*\right), \left({\mathcal T}^*, \mu^*\right) \right) =0 \quad \text{a.s..} \label{alld1} \end{equation}
\end{prop}

\begin{proof} We prove the claim for the sequence of weighted $\mathbb R$-trees $(\overline{\mathcal T}_k^*, \overline{\mu}_k^*, k \geq 0)$  embedded in a given  $(\mathcal T^*, \mu^*)$ as in Section \ref{Embedding}. Then \eqref{alld1} will follow from Proposition \ref{oversamedist}.

By Theorem \ref{mainresult} and Corollary \ref{convstable}, we can couple a stable tree growth process $(\widetilde{\cT}_k,\widetilde{\mu}_k)\rightarrow(\cT,\mu)$ with 
$(\overline{\cT}_k^*,\overline{\mu}_k^*,k\geq 0)$ in such a way that $\widetilde{\mu}_k$ is a push-forward of $\overline{\mu}_k^*$. In particular, we have
\begin{equation}\max\{\overline{\mu}_k^*(x),x\in\overline{\cT}_k^*\}\le\max\{\widetilde{\mu}_k(x),x\in\widetilde{\cT}_k\}\rightarrow 0\qquad\mbox{a.s..}
  \label{noatoms}
\end{equation}
On the other hand, $\overline{\mu}_k^*$ is the pushforward of $\mu^*$ under the projection map $\overline{\pi}_k^*\colon\cT^*\rightarrow\overline{\cT}_k^*$. Now assume,
for contradiction that $\overline{\bigcup_{k\ge 0}\overline{\cT}_k^*}\neq\cT^*$. Since all leaves are limit points of $\cT^*\setminus{\rm Lf}(\cT^*)$ and by Theorem \ref{reccon}, $\cT^*$ is a
CRT, there is $x\in\cT^*\setminus\overline{\bigcup_{k\ge 0}\overline{\cT}_k^*}$ such that the subtree of $\cT^*$ above $x$ has positive mass 
$c:=\mu^*(\cT_x^*)>0$. Since $\overline{\bigcup_{k\ge 0}\overline{\cT}_k^*}$ is path-connected, $\cT_x^*\cap\overline{\bigcup_{k\ge 0}\overline{\cT}_k^*}=\varnothing$,
and hence all $\overline{\mu}_k^*$ must have an atom greater than $c$, which contradicts \eqref{noatoms}. 

We conclude that $\overline{\bigcup_{k\ge 0}\overline{\cT}_k^*}=\cT^*$. Since $\cT^*$ is compact and the union is increasing in $k\ge 0$, this implies GH-convergence.
The convergence in the GHP sense follows since the mass measure $\overline{\mu}_k^*$ is the projection of $\mu^*$ onto $\overline{\mathcal T}_k^*$, see the proof of \cite[Corollary 23]{1} for details of this argument. 
\end{proof}

\begin{corollary}[Convergence of two-colour trees] \label{contwoc} Let $(\mathcal T_k^*, (\mathcal R_k^{(i)}, i \geq 1 ), \mu_k^*, k \geq 0)$ be the two-colour tree growth process from Algorithm \ref{twocolourmass} for some $\beta \in (0,1/2]$. Then there exist a compact CRT $(\mathcal T^*, \mu^*)$, an i.i.d. sequence $(\mathcal F^{(i)}, i \geq 1)$ of Ford CRTs of index $\beta'=\beta/(1-\beta)$ and scaling factors $(C^{(i)}, i \geq 1)$ as in Corollary \ref{algsamecor} with 
$ \lim_{k \rightarrow \infty} d_{\rm GHP}^\infty  \big( \big( \mathcal T_k^*, \big(\mathcal R_k^{(i)}, i \geq 1\big), \mu_k^* \big) , \big(\mathcal T^*, \big( \big(C^{(i)}\big)^{-1} \mathcal F^{(i)}, i \geq 1 \big), \mu^* \big) \big) = 0$ a.s..
 \end{corollary}

\begin{proof}This is a direct consequence of Proposition \ref{propo2} and Corollary \ref{algsamecor}.
\end{proof}

It will be convenient to use the representation of 
Algorithm \ref{depstruc1}. We note the following consequences of the construction, in the light of the Proposition \ref{propo2}. 

\begin{corollary}\label{corproj} In the setting of Algorithm \ref{depstruc1}
  \begin{enumerate}
    \item[(i)] the closure $\widehat{\mathcal{T}}$ in $l^1(\mathbb N_0^2)$ of the increasing union 
      $\bigcup_{k\ge 0}\mathcal{R}(\widehat{\mathcal{T}}_k,\widehat{\Sigma}_0^{(k)},\ldots,\widehat{\Sigma}_k^{(k)})$ is compact;
    \item[(ii)] the natural projection of $\widehat{\mathcal T}$ onto the subspace spanned by $e_{k,0}$, $k\ge 0$, is the stable tree $\mathcal T$;
    \item[(iii)] the natural projection of $\widehat{\mathcal T}$ onto the subspace spanned by $e_{m,i}$, $m\ge 1$, scaled by the scaling factor $C^{(i)}$ of
      Remark \ref{pidi}, is a Ford CRT for each $i\ge 1$.
  \end{enumerate} 
\end{corollary} 
\begin{proof} (i) It follows from Propositions \ref{algsame} and \ref{propo2}, that the closure $\widehat{\mathcal{T}}$ in $l^1(\mathbb N_0^2)$ of the increasing union 
  is compact. (ii) holds by construction since all steps of Algorithm \ref{depstruc1} preserve this projection property for the trees $\widehat{\mathcal T}_k$, $k\ge 0$.
  (iii) holds by Corollary \ref{algsamecor} since the scaled projections of 
  $\mathcal{R}(\widehat{\mathcal{T}}_k,\widehat{\Sigma}_0^{(k)},\ldots,\widehat{\Sigma}_k^{(k)})$ are Ford tree growth processes whose $m$th growth step is for 
  $k=k_m^{(i)}$, $m\ge 1$, $i\ge 1$.
\end{proof} 

These two corollaries imply Theorem \ref{Mainresult2}.

\subsection{Branch point replacement in a stable tree, and the proof of Theorem \ref{branchrepl}} \label{bprepl}

The aim of this section is to replace branch points of the stable tree by rescaled independent Ford CRTs. Let us denote the independent Ford tree growth processes underlying  Corollary \ref{corproj}(iii) by $({\mathcal F}^{(i)}_m,m\ge 1)$, and the Ford CRTs with leaf labels
by $(\mathcal{F}^{(i)},\Omega^{(i)}_m,m\ge 1)$, $i\ge 1$, all embedded in the appropriate coordinates. Now fix $i\ge 1$, and focus on the $m$th subtree of the $i$th branch point of $\mathcal{T}$, suppose $\Sigma_n$ 
is its smallest label. In Algorithm \ref{depstruc1}, each insertion into the $i$th marked component shifts some subtrees of the $i$th branch point, and the subtree we 
consider stops being shifted at the $m$th insertion. 

The branch point replacement algorithm can be viewed as a change of order of the insertions of Algorithm \ref{depstruc1}. The $k$th step of Algorithm \ref{depstruc1}
gets $\Sigma_k$ into its final position $\widehat{\Sigma}_k^{(k)}$ by inserting one branch of a marked component. The $i$th step of the branch point replacement 
algorithm gets the smallest labelled leaf of all subtrees of the $i$th branch point into their final positions by making all insertions into the $i$th component. This 
amounts to shifting the $m$th subtree of the $i$th branch point by $\Omega^{(i)}_m$, $m\ge 1$.

\begin{algorithm}[Branch point replacement in the stable tree] \label{algbprepl}
We construct a sequence of weighted $i$-marked $\mathbb R$-trees $(\mathcal B^{(i)},(\mathcal{R}^{(1)},\ldots,\mathcal{R}^{(i)}),\mu^{(i)})$.
Let $(\mathcal B^{(0)}, \mu^{(0)})=(\mathcal T, \mu)$ be the embedded stable tree with leaves $\Sigma_n^{(0)}\!=\!\Sigma_n, n\!\ge\! 0$. 
For $i \!\geq\! 1$, conditionally given $(\mathcal B^{(i-1)},(\mathcal{R}^{(1)},\ldots,\mathcal{R}^{(i-1)}),\mu^{(i-1)},(\Sigma_n^{(i-1)},n\ge 0))$, shift the connected
components $\mathcal S_m^{(i)}$, $m\in\{0,1,2,\ldots;\infty)$, of $\mathcal B^{(i-1)}\setminus v_i^{(i-1)}$ of the $i$th branch point $v_i^{(i-1)}$: 
\begin{equation*} \mathcal B^{(i)}:= \mathcal S_\infty^{(i)}\cup\mathcal S_0^{(i)}\cup\left(v_i^{(i-1)}+\left(C^{(i)}\right)^{-1}\mathcal{F}^{(i)}\right)\cup\bigcup_{m\ge 1}\left(\left(C^{(i)}\right)^{-1}\Omega_m^{(i)}+\mathcal S_m^{(i)}\right)
\end{equation*}
where $\mathcal{F}^{(i)}$ is the independent Ford CRT with labelled Ford leaves $(\Omega_m^{(i)},m\ge 1)$. Take as $\mu^{(i)}$ the measure $\mu^{(i-1)}$ shifted with each 
of the connected components and set 
$\mathcal R^{(i)}:=\left(v_i^{(i-1)}+\left(C^{(i)}\right)^{-1}\mathcal{F}^{(i)}\right)$.
\end{algorithm}



\begin{theorem}[Branch point replacement] The $\mathbb R$-trees 
  $(\mathcal B^{(i)},(\mathcal{R}^{(1)},\ldots,\mathcal{R}^{(i)},\{0\},\{0\},\ldots),\mu^{(i)})$, of Algorithm \ref{algbprepl} 
  converge in $(\mathbb{T}_{\rm w}^{\infty},d_{\rm GHP}^{\infty})$ to a limit with the same distribution as in Corollary \ref{contwoc}, i.e.
  \begin{equation*} \lim_{i \rightarrow \infty} d_{\rm GHP}^\infty  \left( \left( \mathcal B^{(i)}, \left(\mathcal R^{(1)},\ldots,\mathcal R^{(i)},\{0\},\ldots\right), \mu^{(i)}\right) , \left(\mathcal T^*, \left( \left(C^{(i)}\right)^{-1} \mathcal F^{(i)}, i \geq 1 \right), \mu^* \right) \right) = 0 \quad \text{a.s.}. \end{equation*}
\end{theorem}

\begin{proof} By construction, the trees spanned by the first $k$ leaves are the same in Algorithms \ref{depstruc1} and \ref{algbprepl}: 
%
\begin{equation} \left( \mathcal R\left(\widehat{\mathcal T}_k, \Sigma^{(k)}_0, \ldots, \Sigma_k^{(k)} \right), \left(\widehat{\mathcal R}_k^{(i)}, i \geq 1\right) , \widehat{\mu}_k^*, k \geq 0\right) \,=\, \left(\mathcal B^{(k)}_k, \left(\mathcal U_k^{(i)}, i \geq 1 \right), \lambda_k, k \geq 0\right) \label{brepldis} \end{equation}
where $\mathcal B_k^{(k)}:=\mathcal R(\mathcal B^{(k)}, \Sigma_0^{(k)}, \ldots, \Sigma_k^{(k)} )$, $\mathcal U_k^{(i)}:=\mathcal R^{(i)} \cap \mathcal B_k^{(k)}$, and $\lambda_k=(\pi_k^{\mathcal B})_*\mu^{(k)}$ denotes the projected mass measure.

%
%
%

By Proposition \ref{algsame} and Corollary \ref{contwoc}, we have convergence of reduced trees to the claimed limit. In particular, for all $\varepsilon>0$, there is
$k_0\ge 0$ such that for all $k\ge k_0$,
$$d_{\rm GHP}^\infty  \left( \left(\mathcal B^{(k)}_k, \left(\mathcal U_k^{(i)}, i \geq 1 \right), \lambda_k\right), \left(\widehat{\mathcal T}, \left( \left(C^{(i)}\right)^{-1} \mathcal F^{(i)}, i \geq 1 \right), \widehat{\mu} \right) \right)<\varepsilon/3.$$
But this is only possible if all connected components of $\widehat{\mathcal{T}}\setminus\mathcal B^{(k)}_k$ have height less than $2\varepsilon/3$. By construction, the 
components of $\mathcal B^{(k)}\setminus\mathcal B^{(k)}_k$ are bounded in height by the corresponding components of height less than $2\varepsilon/3$. Since 
$\widehat{\mu}$
and $\mu^{(k)}$ have the same projection onto $\widehat{\mathcal T}_k=\mathcal B^{(k)}_k$, we conclude that also
$$d_{\rm GHP}^\infty  \left( \left(\mathcal B^{(k)}_k, \left(\mathcal U_k^{(i)}, i \geq 1 \right), \lambda_k\right), \left(\mathcal B^{(k)}, \left(\mathcal R^{(1)}, \ldots,\mathcal R^{(k)},\{0\},\ldots \right), \mu^{(k)} \right) \right)<2\varepsilon/3.$$
By the triangle inequality, this completes the proof.
\end{proof}

This formalises and proves Theorem \ref{branchrepl}.

\section{Discrete two-colour tree growth processes} \label{Sec6}


Marchal \cite{18} introduced a tree growth model related to the stable tree. Specifically, he built a sequence of discrete trees $(T_n, n \geq 0)$, which we view as rooted $\mathbb R$-trees with unit edge lengths, equipped with the graph distance, i.e. the distance between two vertices $x,y \in T_n$ is the number of edges between $x$ and $y$. 


\begin{algorithm}[Marchal's algorithm] \label{Marchal}
Let $\beta \in (0,1/2]$. We grow discrete trees $ T_n$, $n \geq 0$, as follows. 

\begin{itemize}
\item[0.]Let $T_0$ consist of a root $\rho$ and a leaf $\Sigma_0$, connected by an edge. \end{itemize}
Given $T_n$, with leaves $\Sigma_0,\ldots,\Sigma_n$, 
\begin{itemize}
\item[1.] distribute a total weight of $n+\beta$ by assigning $(d-3)(1-\beta)+1-2\beta$ to each vertex of degree $d \geq 3$ and $\beta$ to each edge of $T_n$; select a vertex or an edge in $T_n$ at random according to these weights;
\item[2.] if an edge is selected, insert a new vertex, i.e. replace the selected edge by two edges connecting the new vertex to the vertices of the selected edge; proceed with the new vertex as the selected vertex;
\item[3.] in all cases, add a new edge from the selected vertex to a new leaf $\Sigma_{n+1}$ to form $T_{n+1}$.
\end{itemize}
\end{algorithm}
\noindent Strengthening a result by Marchal \cite{18}, Curien and Haas \cite{11} showed that the sequence of trees $(T_n, n \geq 0)$ has the stable tree $\mathcal T$ of index $\beta$ as its a.s.\ scaling limit, in the following strong sense: 
\begin{equation*}
\lim \limits_{n \rightarrow \infty} n^{-\beta} T_n=\mathcal T \quad \text{a.s. in the Gromov-Hausdorff topology}.
\end{equation*}

The trees $(\mathcal F_m, m \geq 1)$ of a Ford tree growth process can also be obtained as scaling limits of a discrete tree growth process, the so-called Ford alpha-model. Both Marchal's model related to the stable tree and Ford's alpha-model are contained as special cases in the alpha-gamma-model studied in \cite{10}. 
\begin{definition}[The alpha-gamma-model] Let $\alpha \in [0,1]$ and $\gamma \in (0,\alpha]$. We grow discrete trees $T_n, n \geq 1$:
\begin{itemize}
\item[0.] Let $T_1$ consist of a root $\rho$ and a leaf $\Sigma_1$, connected by an edge. 
\end{itemize}
Given $T_n$, with leaves $\Sigma_1,\ldots,\Sigma_n$, 
\begin{itemize}
\item[1.] distribute a total weight of $n-\alpha$ by assigning $(d-2)\alpha -\gamma$ to each vertex of $T_n$ of degree $d \geq 3$, $1-\alpha$ to each external edge of $T_n$, and $\gamma$ to each internal edge of $T_n$; select a vertex or an edge in $T_n$ at random according to these weights;
\item[2.] if an edge is selected, insert a new vertex, i.e. replace the selected edge by two edges connecting the new vertex to the vertices of the selected edge; proceed with the new vertex as the selected vertex;
\item[3.] in all cases, add a new edge from the selected vertex to a new leaf $\Sigma_{n+1}$ to form $T_{n+1}$.
\end{itemize}
\end{definition}

Note that the case $\gamma=1-\alpha=\beta$ gives Marchal's model, Algorithm \ref{Marchal}, while the case $\gamma=\alpha=\beta'$ was introduced by Daniel Ford in his 
thesis \cite{9} and is referred to as Ford's alpha-model. 
In the latter, branch points get assigned weight zero after their creation, i.e. the trees in Ford's alpha model are binary.




 
\begin{lemma}[Convergence of reduced trees] Let $(T_n, n \geq 1)$ be an alpha-gamma tree-growth process for some $\alpha \in (0,1)$ and $\gamma \in (0, \alpha]$. For $k \geq 1$, consider the reduced tree $\mathcal R \left( T_n, \Sigma_1, \ldots, \Sigma_k \right)$ spanned by the root and the first $k$ leaves, equipped with the graph distance on $T_n$, i.e.\ for any edge $a \rightarrow b$ in $\mathcal R \left( T_n, \Sigma_1, \ldots, \Sigma_k \right)$  the number of edges between $a$ and $b$ in $T_n$.
Then there exists an $\mathbb R$-tree $\mathcal R_k$ such that
\begin{equation*} \lim_{n \rightarrow \infty} n^{-\gamma} \mathcal R \left( T_n, \Sigma_0, \ldots, \Sigma_k\right)= \mathcal R_k \quad \text{a.s.} \end{equation*}
in the Gromov-Hausdorff topology. Furthermore, conditionally given that $T_k$ has a total of $k+\ell$ edges, i.e. that $T_k$ has $\ell$ branch points, the edge lengths of $\mathcal R_k$ are given by $L_k V_k^\gamma D_k$ where
\begin{equation*} D_k \sim \rm{Dirichlet}\left( ({1-\alpha})/{\gamma}, \ldots, ({1-\alpha})/{\gamma}, 1, \ldots, 1\right) \end{equation*}
with a weight of $(1-\alpha)/\gamma$ for each external edge, and weight $1$ for each internal edge, and 
\begin{equation*}
L_k \sim {\rm{ML}}(\gamma,\ell\gamma+k(1-\alpha)), \quad V_k \sim {\rm{Beta}}\left(k(1-\alpha)+\ell\gamma, (k-1)\alpha-\ell\gamma\right)
\end{equation*}
are conditionally independent. 
\end{lemma}

Note that in the stable case, the total length is a $V_k \sim {\rm{Beta}}((k+\ell)(1-\alpha), (k-1-\ell)\alpha-\ell)$ proportion of $L_k \sim {\rm{ML}}(1-\alpha,(k+\ell)(1-\alpha))$, and is uniformly distributed amongst the $k+\ell$ edges. In Ford's model, we have $\ell=k-1$, and we distribute the \enquote{full} length $L_k \sim {\rm{ML}}(\alpha, k-\alpha)$ according to a Dirichlet variable $D_k$ with a parameter of $1/\alpha-1$ for each external edge and parameter $1$ for each internal edge.

In a similar manner, we can obtain the two-colour trees $(\mathcal T_k^*, (\mathcal R_k^{(i)}, i \geq 1))$, $k \geq 0$, as a.s. scaling limits of the following discrete tree growth process in the space of $\infty$-marked $\mathbb R$-trees with unit edge lengths. 

\begin{definition}[The discrete two-colour model] \label{deftwoc} 
Let $\beta \in (0,1/2]$. We grow discrete two-colour trees $(T^*_n, (R_n^{(i)}, i \geq 1))$, $n \geq 0$, as follows.
\begin{itemize}
\item[0.] Let $T_0$ consist of a root $\rho$ and a leaf $\Sigma_0$ connected by an edge, let $R_0^{(i)}=\{\rho\}$, $i \geq 1$, and $r_0=0$.
\end{itemize}
Given $(T^*_n, (R_n^{(i)}, i \geq 1))$, with leaves $\Sigma_0,\ldots,\Sigma_n$ and $r_n=\#\{i\geq 1\colon R_n^{(i)}\neq \{\rho\}\}$, 
\begin{itemize}
\item[1.] distribute a total weight of $n+\beta$ by assigning
$\beta$ to each unmarked and each internal marked edge of $T_n$, and $1-2\beta$ to each external marked edge of $T_n$; select an edge in $T_n$ at random according to these weights;
\item[2.] if the selected edge is unmarked, replace it by two unmarked edges connecting the new vertex to the vertices of the selected edge and set $I_n=r_n+1$; if the selected edge is a marked edge of  $R_n^{(i)}$ for some $i\geq 1$, replace it by two marked edges and set $I_n=i$; proceed with the new vertex as the selected vertex;
\item[3.] add a new degree-2 vertex, connect it to the selected vertex by a marked edge, and to a new leaf $\Sigma_{n+1}$ by an unmarked edge; add the marked edge to $R_n^{(I_n)}$ to form $R_{n+1}^{(I_n)}$; set $R_{n+1}^{(i)}=R_n^{(i)}$ for $i\neq I_n$. \pagebreak
\end{itemize}
\end{definition}

\begin{prop}[Convergence of the discrete two-colour model] \label{distwoccon} Consider the discrete two-colour tree growth process $(T_n^*, (R_n^{(i)}, i \geq 1), n \geq 0)$ from Definition \ref{deftwoc}, which we view as a sequence of $\infty$-marked $\mathbb R$-trees with unit edge lengths. For all $k \geq 0$, let
$\mathcal R( T_n^*, (R_n^{(i)},i \geq 1)  , \Sigma_0, \ldots, \Sigma_k)$
denote the reduced tree spanned by the root $\rho$ and the leaves $\Sigma_0, \ldots, \Sigma_k$. Then \begin{equation*} \lim_{n \rightarrow \infty} n^{-\beta}\mathcal R \left(T_n^*, \left(R_n^{(i)}, i \geq 1 \right) , \Sigma_0, \ldots, \Sigma_k \right) = \left(\mathcal T_k^*, \left(\mathcal R_k^{(i)}, i \geq 1 \right) \right) \quad \text{a.s.} \end{equation*} 
with respect to the distance $d_{GH}^\infty$ defined in \eqref{GHK}, where $(\mathcal T_k^*,(\mathcal R_k^{(i)}, i \geq 1),k\ge 0)$ is as in Algorithm \ref{twocolour}.

Conditionally given that $T_k^*$ has $r_k$ marked components $R_k^{(i)} \neq \{\rho\}$ with $d_1-2, \ldots, d_{r_k}-2$ leaves, the distribution of the edge lengths of $(\mathcal T_k^*, (\mathcal R_k^{(i)}, i \geq 1 ) ) $ is given by $S_k^* D_k$ where $S_k^* \sim {\rm{ML}}(\beta, \beta+k)$ and 
\begin{equation*} D_k \sim \rm{Dirichlet}\left( 1, \ldots, 1, 1/\beta-2, \ldots, 1/\beta-2\right)
\end{equation*}
with weight $1$ for each unmarked edge and each internal marked edge, and weight $1/\beta-2$ for each external marked edge, are conditionally independent.
\end{prop}

The proof of Proposition \ref{distwoccon} is based on exactly the same techniques as the proof of the corresponding result for the alpha-gamma model, cf. \cite[Propositions 21 and 22]{10}, and the result for $(\alpha, \theta)$-tree growth processes, cf. \cite[Proposition 14]{1}. We omit the details.

\begin{remark} One can obtain the mass measures $\mu_k^*$, $k \geq 0$, as the scaling limits of the empirical measures on the leaves of $T_n$, projected onto the reduced trees, using the same methods as in \cite{1}. In particular, each edge equipped with limiting relative projected subtree masses is a rescaled $(\beta, \theta)$-string of beads where $\theta=\beta$ for internal marked and unmarked edges, and $\theta=1-2\beta$ for external marked edges. It can be shown directly that these strings of beads are independent of each other and of the mass split on $\mathcal T_k^*$, which has distribution ${\rm{Dirichlet}}( \beta, \ldots, \beta, 1-2\beta, \ldots, 1-2\beta)$, with parameter $\beta$ for each internal marked and unmarked edge, and parameter $1-2\beta$ for each external marked edge of $\mathcal T_k^*$, as in Proposition \ref{starstrings}.
\end{remark}

\appendix

\section{Appendix} \label{appen}

We present the proofs postponed from earlier parts of this paper.

\subsection{Coupling proofs of Proposition \ref{masssel2}, Theorem \ref{mainresult} and Proposition \ref{oversamedist}}\label{appen1}

The proofs of Proposition \ref{masssel2}, Theorem \ref{mainresult} (i) and (ii), and of Proposition \ref{oversamedist} are based on coupling arguments and are quite similar to one another. We present the proof of Theorem \ref{mainresult}(i) first.

\begin{proof}[Proof of Theorem \ref{mainresult}(i)]
Recall the constructions of $(\mathcal T_k^*, \mu_k^*)$ in Algorithm \ref{twocolourmass} and $(\widetilde{\mathcal T}_k, \widetilde{\mu}_k)$ in \eqref{ttilde}. We couple $(\mathcal T_k, \mu_k, k \geq 0)$ to $(\mathcal T_k^*, \mu_k^*, k \geq 0)$ and identify the distribution as required for Algorithm \ref{masssel}:

\begin{itemize}

\item We couple the initial $(\beta, \beta)$-strings of beads to be equal 
$(\mathcal T_0, \mu_0)=(\widetilde{\mathcal T}_0, \widetilde{\mu}_0)=(\mathcal T_0^*, \mu_0^*)$.

\item Supposing that $(\mathcal T_k, \mu_k)=(\widetilde{\mathcal T}_k, \widetilde{\mu}_k)$ for some $k \geq 0$, set $J_k:=\widetilde{J}_k=[J_k^*]_{\sim}$, $\xi_k=\xi_k^{(2)}$, and 
\begin{equation*} Q_k:=\left(1-\gamma_k\right) {\mu_k^*({J_k^*})}/{\widetilde{\mu}_k(\widetilde{J}_k)},
\end{equation*}
where we recall that $(1-\gamma_k) \sim \rm{Beta}(\beta, 1-2\beta)$ is the independent scaling factor for $\xi_k^{(2)}$ in the construction of a $\beta$-mixed string of 
beads from $\xi_k^{(1)}$, $\xi_k^{(2)}$ and $\gamma_k$, as at the beginning of Section \ref{sec4}. If the selected atom $J_k^*$ is an element of a marked component,
$Q_k$ is the proportion of the mass of $J_k^*$ added to this marked component in the form of a rescaled independent $(\beta,\beta)$-string of beads $\xi_k^{(2)}$, 
while a proportion of $1-Q_k$ is split into an unmarked rescaled $(\beta, 1-2\beta)$-string of beads $\xi_k^{(1)}$.
\end{itemize}
Since $\widetilde{J}_k$ was sampled from $\widetilde{\mu}_k$, $J_k$ is sampled from $\mu_k$, as required for Algorithm \ref{masssel}.
It remains to check that the scaling factor $Q_k {\widetilde{\mu}_k}({\widetilde{J}_k})$ induced by Algorithm \ref{twocolourmass}, applied to the $(\beta, \beta)$-string of beads $\xi_k=\xi^{(2)}_k$ that is used in the attachment procedure, is as needed for Algorithm \ref{masssel}. We work conditionally given the event that $\widetilde{\mathcal T}_k$ has $\ell$ branch points $\widetilde{v}_j$ of sizes $d_j={\rm deg}(\widetilde{v}_j, \widetilde{\mathcal T}_k)$, $j \in [\ell]$, respectively.

\begin{itemize} 
\item If $\widetilde{J}_k \neq \widetilde{v}_i$ for $i \in [\ell]$, then $J_k=\widetilde{J}_k=J_k^*$, and a new branch point $\widetilde{J}_k$ of degree deg$(\widetilde{J}_{k}, \widetilde{\mathcal T}_{k+1})=3=1+\text{deg}(\widetilde{J}_{k}, \widetilde{\mathcal T}_{k})$ is created. The mass ${\mu}_k^*(J_k^*)=\widetilde{\mu}_k(\widetilde{J}_k)$ is split by the independent random variable $\gamma_k \sim \text{Beta}(1-2\beta, \beta)$ into a branch point weight $\widetilde{\mu}_{k+1}(\widetilde{J}_{k})=\gamma_k \widetilde{\mu}_k(\widetilde{J}_k)$ and the isometric copy of the $(\beta, \beta)$-string of beads $\xi_k^{(2)}=\xi_k$, scaled by $\widetilde{\mu}_k(\widetilde{J}_k)(1-\gamma_k)=\widetilde{\mu}_k(\widetilde{J}_k)Q_k$ where $Q_k\sim \text{Beta}(\beta, 1-2\beta)$
is conditionally independent of $\xi_k$ and $(\mathcal T_k, \mu_k, J_k)$ given ${\rm deg}(J_k, \mathcal T_k)=2$, as required.

\item If $\widetilde{J}_k = \widetilde{v}_i$ of degree deg$(\widetilde{v}_i,\widetilde{\mathcal T}_k)=d_i$ for some $i \in [\ell]$, we first select an edge $E_k^*$ of $\mathcal R_k^{(i)}$ from $\mu_k^*$ restricted to ${\mathcal R_k^{(i)}}$. Conditionally given that $E_k^*$ has been selected, we choose $J_k^* \in E_k^*$ according to $(\beta, \theta)$-coin tossing sampling, where $\theta=\beta$ if $E_k^*$ is an internal edge of $\mathcal R_k^{(i)}$, and $\theta=1-2\beta$ otherwise. By Proposition \ref{cointoss} and Proposition \ref{Diri}(iii)-(iv), conditionally given $J_k^* \in E_k^*$, the relative mass split in $\mathcal R_k^{(i)}$ is
\begin{equation*} {\rm Dirichlet}\left(\beta, \ldots, \beta, 1-2\beta, \ldots, 1-2\beta, \beta, 1-\beta, \theta\right) \end{equation*}
with parameter $\beta$ for each non-selected internal edge of $\mathcal R_k^{(i)}$, $1-2\beta$ for each non-selected external edge of $\mathcal R_k^{(i)}$, $\beta$ for the part of $E_k^*$ closer to the root, $\theta$ for the other part of $E_k^*$, and $1-\beta$ for the atom $J_k^*$. In any case (i.e. no matter if $E_k^*$ is internal or external), we get by Proposition \ref{Diri}(i)-(ii) that, conditionally given $\widetilde{J}_k=\widetilde{v}_i$, 
\begin{equation*} \mu_k^*(J_k^*)/\widetilde{\mu}_k(\widetilde{J}_k) \sim \text{Beta} \left(1-\beta, (d_i-2)(1-\beta)\right) \end{equation*}
is independent of $\widetilde{\mu}_k(\widetilde{J}_k)$, as the internal relative mass split in $\mathcal R_k^{(i)}$ is independent of its total mass, see Proposition \ref{starstrings} and Proposition \ref{Diri}(ii). Overall, still conditionally given $\widetilde{J}_k = \widetilde{v}_i$, we have that
\begin{equation*} \mu_k^*\left(J_k^*\right)\left(1-\gamma_k\right)= \left(1-\gamma_k\right) \left( \mu_k^*\left(J_k^*\right) \widetilde{\mu}_k \left(\widetilde{J}_k\right)^{-1}\right) \widetilde{\mu}_k \left(\widetilde{J}_k\right) =  Q_k\widetilde{\mu}_k\left({\widetilde{J}_k}\right) \end{equation*}
where $Q_k \sim \text{Beta}(\beta, d_i(1-\beta)-1)$, as is easily checked using Proposition \ref{Diri}(i)-(iii). Note that $Q_k$ is also conditionally independent of $\widetilde{\mu}_k(\widetilde{J}_k)$ given $\widetilde{J}_k=\widetilde{v}_i$ and ${\rm deg}(\widetilde{v}_j,\widetilde{\mathcal T}_k)=d_i$. This is due to the fact that the mass split within $\mathcal R_k^{(i)}$, and the mass split between the edges of $\widetilde{\mathcal T}_k$ and its branch points are conditionally independent given there are $\ell$ branch points $\widetilde{v}_j$ with ${\rm deg}(\widetilde{v}_j, \widetilde{\mathcal T}_k)=d_j$, $j \in [\ell]$. \qedhere
\end{itemize} 
\end{proof}

\begin{proof}[Proof of Proposition \ref{oversamedist}] Similarly to the proof of Theorem \ref{mainresult}(i), let us couple so that the initial weighted $\infty$-marked $\mathbb R$-trees coincide, i.e. let $(\mathcal T_0^*, (\mathcal R_0^{(i)}, i \geq 1), \mu^*_0):=(\overline{\mathcal T}^*_0, (\overline{\mathcal R}_0^{(i)}, i\geq 1), \overline{\mu}^*_0)$. Then, $(\mathcal T_0^*, \mu_0^*)$ is a $(\beta, \beta)$-string of beads, and $\mathcal R_0^{(i)}=\{\rho\}$ for all $i \geq 1$, as required for  Algorithm \ref{twocolourmass}.

Supposing that $(\mathcal T^*_k, (\mathcal R_k^{(i)}, i \geq 1), \mu_k^*)=(\overline{\mathcal T}^*_k, (\overline{\mathcal R}_k^{(i)}, i\geq 1), \overline{\mu}^*_k)$ for 
some $k\geq 0$, set $J_k^*:=\overline{J}_k^*$, $I_k: =\overline{I}_k$, and if $\overline{J}_k^*=\check{X}_{\mathbf{i}j}$, take as $(E_k^+, \mu_k^+)$ the scaled copy of 
$\xi_{\mathbf{i}j}$ embedded in $\mathcal T^*$ and $R_k^+=[[\overline{J}_k^*, \overline{\Omega}_k]]$. We need to check that the induced update step from 
$(\mathcal T_k^*,(\mathcal R_k^{(i)}, i \geq 1), \mu_k^*)$ to $(\mathcal T_{k+1}^*,(\mathcal R_{k+1}^{(i)}, i \geq 1), \mu_{k+1}^*)$ is as required in Algorithm 
\ref{twocolourmass}. Selecting $\overline{J}_k^*$ in Algorithm \ref{twocemb}, we first select an edge $\overline{E}_k^*$ of $\mathcal T_k^*$  proportionally to 
$\overline{\mu}_k^*(\overline{E}_k^*)$, and perform $(\beta, 1-2\beta)$-coin tossing if $\overline{E}_k^*$ is an external marked edge, and uniform sampling from 
$\overline{\mu}_k \restriction_{\overline{E}_k^*}$ otherwise, and since $\mu_k^*=\overline{\mu}_k^*$, this means that $J_k^*$ is sampled precisely as required for Algorithm 
\ref{twocolourmass}, and in particular we have $\mu_k^*(J_k^*)=\overline{\mu}_k^*(\overline{J}_k^*)$. Furthermore, $(E_k^+, R_k^+,\mu_k^+)$ is an independent $\beta$-mixed string of beads, as it is obtained from ${\xi}_{\mathbf{i}j}$ and the transition kernel $\kappa({\xi}_{\mathbf{i}j}, \cdot)$ of Lemma \ref{condis}. 
Therefore, 
\begin{equation*} \left(\left(\mathcal T^*_k, \left(\mathcal R_k^{(i)}, i \geq 1\right), \mu_k^*\right), \left(\mathcal T^*_{k+1}, \left(\mathcal R_{k+1}^{(i)}, i \geq 1\right), \mu_{k+1}^*\right)\right)
\end{equation*}
has the same distribution as 
$\big((\overline{\mathcal T}^*_k, (\overline{\mathcal R}_k^{(i)}, i \geq 1), \mu_k^*), (\overline{\mathcal T}^*_{k+1}, (\overline{\mathcal R}_{k+1}^{(i)}, i\geq 1), \overline{\mu}^*_{k+1})\big)$,
which proves Proposition \ref{oversamedist}, as both Algorithm \ref{twocolourmass} and Algorithm \ref{twocemb} specify Markov chains.
\end{proof}

\begin{proof}[Proof of Proposition \ref{masssel2}] Construction \eqref{findimmarg1} and Algorithm \ref{masssel} use the same notation. To avoid confusion in this proof, 
we denote the sequence of trees of \eqref{findimmarg1} by $({\mathcal T}'_k, {\mu}'_k, k \geq 0)$. We will couple the construction of $(\mathcal T_k, \mu_k, k \geq 0)$ 
of Algorithm \ref{masssel} to the given sequence $({\mathcal T}'_k, {\mu}'_k, k \geq 0)$, specifically identifying the sequences $(J_k, k \geq 0)$ of attachment points, 
and $(Q_k, k \geq 0)$ of update random variables. 

The coupling is as follows. Set $(\mathcal T_0, \mu_0)=({\mathcal T}'_0, {\mu}'_0)$, and, given $(\mathcal T_k, \mu_k)=({\mathcal T}'_k, {\mu}'_k)$ for some $k \geq 0$, set $J_k:={J}'_k$ where 
\begin{equation*} {J}'_k:=\text{arg}\inf\left\{d\left(\rho, x\right)\colon x \in {\mathcal T}'_{k+1} \setminus {\mathcal T}'_k \right\}, 
\end{equation*}
let $Q_k=1-{\mu}'_{k+1}({J}'_k)/{\mu}'_{k}({J}'_k)$, and 
$\xi_k:=\big( {\mu}'_{k+1}({\mathcal T}'_{k+1} \setminus {\mathcal T}'_{k})^{-\beta} {\mathcal T}'_{k+1} \setminus {\mathcal T}'_{k},\ {\mu}'_{k+1}({\mathcal T}'_{k+1} \setminus {\mathcal T}'_{k})^{-1} {\mu}'_{k+1}\restriction_{{\mathcal T}'_{k+1} \setminus {\mathcal T}'_{k}}\big)$.

By Proposition \ref{betastring}, $({\mathcal T}'_0, {\mu}'_0)$ is a $(\beta, \beta)$-string of beads, as required in Algorithm \ref{masssel}. Now assume that $(\mathcal T_k, \mu_k)=({\mathcal T}'_k, {\mu}'_k)$ for some $k \geq 0$ with the distribution claimed in Proposition \ref{masssel2}. Denote the connected components of $\mathcal T \setminus {\mathcal T}'_k$ by $\overline{\mathcal S}_j^{(i)\downarrow}$, $j \ge 1$, $i \ge 1$, completed by their root vertices $\rho_i \in  {\mathcal T}'_k$, $i \ge 1$, respectively. Note that ${\mu}'_k(\rho_i)=\sum_{j \ge 1} {\mu}(\overline{\mathcal S}_j^{(i)\downarrow})$.

Since we sample $\Sigma_{k+1}$ from the mass measure $\mu$ on $\mathcal T$, the conditional probability that $\Sigma_{k+1} \in \overline{\mathcal S}_j^{(i)\downarrow}$, 
given $(\mathcal T, \mu)$, $({\mathcal T}'_k, {\mu}'_k)$ and $(\overline{\mathcal S}_j^{(i)\downarrow}, j \ge 1, i \ge 1)$, is 
$\mu( \overline{\mathcal S}_j^{(i)\downarrow})={\mu}_k'(\rho_i)({\mu}(\overline{\mathcal S}_j^{(i)\downarrow})/{\mu}'_k(\rho_i))$, i.e.\ we can sample ${J}'_k$ in two 
steps: first, select one of the atoms $\rho_i$ of ${\mathcal T}'_k$ proportionally to ${\mu}'_k(\rho_i)$, and second, select one of the components 
$\overline{\mathcal S}_j^{(i)\downarrow}$ with root $\rho_i$ proportionally to relative mass $\mu(\overline{\mathcal S}_j^{(i)\downarrow})/{\mu}'_k(\rho_i)$. 
By Theorem \ref{masspart}(ii) and Proposition \ref{usefulPD}, we further note that conditionally given 
$(\mathcal{T}_k^\prime,\mu_k^\prime)$ with $\mu_k^\prime=\sum_{i\ge 1}\mu_k^{\prime}(\rho_i)\delta_{\rho_i}$, we have  
$(\mu(\mathcal S_j^{(i)})/{\mu}'_k(\rho_i), j \geq 1)^{\downarrow} \sim {\rm PD}(1-\beta,(d_i-3)(1-\beta)+1-2\beta)$ with $d_i={\rm deg}(\rho_i,\mathcal{T}_k^\prime)$, 
$i\ge 1$, independently.

We have ${J}'_k=\rho_i$ with probability ${\mu}'_k(\rho_i)$, and hence $J_k$ is sampled from $\mu_k$, as required in Algorithm \ref{masssel}. 
By Theorem \ref{masspart}(iii), the weighted $\mathbb R$-trees 
$$\left(\mu\left(\overline{\mathcal S}_j^{(i)\downarrow}\right)^{-\beta}\overline{\mathcal S}_j^{(i)\downarrow},
        \mu\left(\overline{\mathcal S}_j^{(i)\downarrow}\right)^{-1}\mu \restriction_{\overline{\mathcal S}_j^{(i)\downarrow}}\right),\qquad j\ge 1,\ i\ge 1,$$ 
are independent copies of $(\mathcal T, \mu)$, i.e. conditionally given $\Sigma_{k+1} \in \overline{\mathcal S}_j^{(i)\downarrow}$, the sampling procedure of 
$\Sigma_{k+1} \in \overline{\mathcal S}_j^{(i)\downarrow}$ from 
$\mu(\overline{\mathcal S}_j^{(i)\downarrow})^{-1}\mu \restriction_{\overline{\mathcal S}_j^{(i)\downarrow}}$ is like sampling $\Sigma_0 \in \mathcal T$ from $\mu$. 
Hence, $\xi_k$ is an independent $(\beta, \beta)$-string of beads, as required in Algorithm \ref{masssel}.

Let us consider the distribution of $Q_k$. Conditionally given ${\rm deg}({J}'_k, {\mathcal T}'_k)=2$, $\Sigma_{k+1}$ is a leaf of a connected component 
$\overline{\mathcal S}_j^{(i)\downarrow}$ of $\mathcal T \setminus {\mathcal T}'_k$ with root $\rho_i={J}'_k$, which is chosen independently and proportionally to 
relative mass $\mu(\overline{\mathcal S}_j^{(i)\downarrow})/{\mu}'_k(\rho_i)$. As noted above, the relative mass partition above ${J}'_k$ is 
PD$(1-\beta,-\beta)$, i.e. by Proposition \ref{usefulPD}, $Q_k \sim {\rm Beta}(\beta, 1-2\beta)$, as required in Algorithm \ref{masssel}.

Conditionally given ${\rm deg}({J}'_k, {\mathcal T}'_k)=d$ for some $d \geq 3$, $\Sigma_{k+1}$ is a leaf of a connected component 
$\overline{\mathcal S}_j^{(i)\downarrow}$ of $\mathcal T \setminus {\mathcal T}'_k$ with root $\rho_i={J}'_k$. Then  
the relative mass partition of the connected components $\mathcal T \setminus {\mathcal T}'_k$ with root $\rho_i$ is PD$(1-\beta,(d-3)(1-\beta)+1-2\beta)$ where we note 
that ${J}'_k$ must have been selected $d-2$ times up step $k$ in order to obtain ${\rm deg}({J}'_k, {\mathcal T}'_k)=d$. Therefore, by Proposition \ref{usefulPD}, 
conditionally given ${\rm deg}({J}'_k, {\mathcal T}'_k)=d$, $Q_k \sim {\rm Beta}(\beta, (d-3)(1-\beta)+1-2\beta)$, as required in Algorithm \ref{masssel}. Also, by Proposition \ref{usefulPD}, $Q_k$ is conditionally independent of ${\mu}'_k({J}'_k)$ given ${\rm deg}({J}'_k, {\mathcal T}'_k)=d$. The mass split in 
$({\mathcal T}'_{k+1}, {\mu}'_{k+1})$ is easily found from Proposition \ref{Diri}, cf. the proof of Proposition \ref{starstrings} for a similar elementary Dirichlet
argument.\end{proof}

\begin{proof}[Proof of Theorem \ref{mainresult}(ii)]

Recall that the ingredients in Algorithm \ref{GH} to construct the sequence on the RHS of \eqref{emstlb2} are the Mittag-Leffer Markov chain $(S_k, k \geq 0)$, attachment points $(J_k, k \geq 0)$, and 
i.i.d.\ random variables $B_k$, $k \geq 0$, with $B_1 \sim {\rm Beta}(1, 1/\beta-2)$. We recover these ingredients from the random variables incorporated in the construction of the LHS of \eqref{emstlb2} via the following coupling. 
\begin{itemize}
\item Set $S_0=S_0^*$, i.e. $S_0 \sim {\rm ML}(\beta, \beta)$ is the length of the initial $(\beta, \beta)$-string of beads $(\mathcal T_0^*, \mu_0^*)=(\widetilde{\mathcal T}_0, \widetilde{\mu}_0)$. For $k \geq 0$, set $S_k$ equal to the total length of $\mathcal T_k^*$, i.e. $S_k=S_k^*$.

\item Set $(J_k, k \geq 0)=(\widetilde{J}_k, k \geq 0)$.

\item Set $(B_k, k \geq 0)=(B_k^*, k \geq 0)$, where $B_k^*$ denotes the length split between the unmarked and the marked part of the independent $\beta$-mixed string of beads $(E_k^+, R_k^+, \mu_k^+)$ built from $\xi_k^{(1)}$, $\xi_k^{(2)}$ and $\gamma_k$. By Remark \ref{betamixedlength}, $(B_k, k \geq 0)$ is an i.i.d. sequence with $B_1 \sim {\rm Beta}(1, 1/\beta-2)$, as required.
\end{itemize}

We will show that 
\begin{align} \left( \widetilde{\mathcal T}_{k}, \left( \widetilde{W}_{j}^{(i)}, 0 \leq j \leq k, i \geq 1\right) \right)  \,{\buildrel d \over =}\,   \left({\mathcal T}_{k}, \left( {W}_j^{(i)}, 0 \leq j \leq k, i \geq 1\right) \right) \label{toshow} \end{align}
for all $k \geq 0$, which implies \eqref{emstlb2} as the families of trees $(\widetilde{\mathcal T}_k, k \geq 0)$ and $(\mathcal T_k, k \geq 0)$ are consistent, i.e.\ given the tree $\mathcal T_k$ at step $k$, we can recover the previous steps $\mathcal T_{k-1}, \ldots, \mathcal T_0$ of the tree sequence. 

We prove \eqref{toshow} by induction on $k$. For $k=0$ the claim is trivial. Suppose that \eqref{toshow} holds up to $k$.
In the tree growth process $(\widetilde{\mathcal T}_k, k \geq 0)$ edge and branch point selection is based on masses, whereas in $(\mathcal T_k, k \geq 0)$ edges are selected based on length and branch points based on weights. We first prove the correspondence of the selection rules, where we work conditionally given the shape of the tree $\widetilde{\mathcal T}_k=\mathcal T_k$, in particular conditionally given that $\mathcal T_k^*$ has $\ell$ marked components $\mathcal R_k^{(i)} \neq \{\rho\}$, $i \in [\ell]$, of sizes $d_i-2$, $i \in  [\ell]$, respectively, or, in other words, that $\widetilde{\mathcal T}_k$ has $\ell$ branch points $\widetilde{v}_i, i \in [\ell]$, of degrees $d_i, i \in [\ell],$ respectively, and a total of $k+\ell+1$ edges. 
By (i) and Proposition \ref{masssel2}, the total mass split in $\widetilde{\mathcal T}_k$ is
\begin{equation}
\left(\widetilde{\mu}_k\left(E_k^{(1)}\right), \ldots, \widetilde{\mu}_k\left(E_k^{(k+\ell+1)}\right), \widetilde{\mu}_k\left(\widetilde{v}_1\right), \ldots, \widetilde{\mu}_k\left(\widetilde{v}_\ell\right)\right)   \sim \text{Dirichlet}\left(\beta, \ldots, \beta, w(d_1), \ldots, w(d_\ell)\right) \label{massintikay}
\end{equation}
where $w(d_i)=(d_i-3)(1-\beta)+1-2\beta$ for $i \in [\ell]$. We denote the edge lengths and the branch point weights in $\widetilde{\mathcal T}_k$ by 
\begin{equation} \widetilde{L}_k=\left(\widetilde{ L}_k^{(1)}, \ldots, \widetilde{ L}_k^{(k+\ell+1)} \right), \quad \widetilde{W}_k=\left(\widetilde{W}_k^{(1)}, \ldots, \widetilde{W}_k^{(\ell)}\right), 
\end{equation} 
and use corresponding notation in $\mathcal T_k$. We will show that the joint distributions of edge lengths, weights and the selected attachment points $\widetilde{J}_k$ and $J_k$ in $\widetilde{\mathcal T}_k$ and $\mathcal T_k$, respectively, are the same in Algorithm \ref{twocolourmass} and Algorithm \ref{GH}, i.e. 
for any $k \geq 0$, and any continuous and bounded function $f\colon \mathbb R^{k+2\ell+1} \rightarrow \mathbb R$,
\begin{align}
\mathbb E \left [ f\left( \widetilde{L}_k, \widetilde{W}_k\right) \mathbbm{1}_{\{\widetilde{J} _k\in E_k^{(j)}\}} \right] &= \mathbb E \left [ f\left( {L}_k, {W}_k\right) \mathbbm{1}_{\{{J} _k\in E_k^{(j)}\}} \right]\qquad\mbox{for any $j \in [k+\ell+1]$,}  \label{equalexp1}\\
\mbox{and}\qquad\mathbb E \left [ f\left( \widetilde{L}_k, \widetilde{W}_k\right) \mathbbm{1}_{\{\widetilde{J}_k = v_j \}} \right] &= \mathbb E \left [ f\left( {L}_k, {W}_k\right) \mathbbm{1}_{\{{J}_k=v_j\}} \right]\qquad\ \mbox{for any $j \in [\ell]$.} 
\label{equalexp2} \end{align}
 Then, together with the coupling, this completes the induction step. It remains to prove \eqref{equalexp1} and \eqref{equalexp2}.

\begin{itemize} 

\item \textit{Proof of \eqref{equalexp1}}.
Fix some $j \in [k+\ell+1]$, and consider the LHS of \eqref{equalexp1} first. Conditioning on $\widetilde{J}_k \in E_k^{(j)}$, and using the mass split \eqref{massintikay} and Proposition \ref{Diri}(iv), we obtain
\begin{equation*}
\mathbb E \left [ f\left( \widetilde{L}_k, \widetilde{W}_k \right) \mathbbm{1}_{\{\widetilde{J} _k\in E_k^{(j)}\}} \right] = \frac{\beta}{k+\beta} \mathbb E \left [ f\left( \widetilde{L}_k, \widetilde{W}_k \right) \middle| {\widetilde{J} _k\in E_k^{(j)}} \right].
\end{equation*}
By Proposition \ref{Diri}(iv) and \eqref{massintikay}, conditionally given ${\widetilde{J} _k\in E_k^{(j)}}$, the distribution of the mass split 
\begin{equation} \left(X_k^{(1)}, \ldots, X_k^{(j-1)}, X_k^{(j)}, X_k^{(j+1)}, \ldots, X_k^{(k+\ell+1)}, X_k^{(k+\ell+2)}, \ldots, X_k^{(k+2\ell+1)}\right) \label{xmasses} \end{equation}  with $X_k^{(i)}=\widetilde{\mu}_k(E_k^{(i)})$ for $i \in [k+\ell+1]$ and $X_k^{(i)}=\widetilde{\mu}_k(\widetilde{v}_{i-(k+\ell+1)})$ for $i \in [k+2\ell+1]\setminus [k+\ell+1]$ is 
\begin{equation}
\text{Dirichlet}\left(\beta, \ldots, \beta, 1+\beta, \beta, \ldots, \beta, w(d_1), \ldots, w(d_\ell)\right). \label{xnotation}
\end{equation}
 
Furthermore, still conditionally given ${\widetilde{J} _k\in E_k^{(j)}}$, $\widetilde{J}_k$ is an atom of mass $\widetilde{\mu}_k(\widetilde{J}_k)=:U_k^{(j)} X_k^{(j)}$ sampled from the rescaled independent $(\beta, \beta)$-string of beads related to $E_k^{(j)}$, splitting $E_k^{(j)}$ into two edges $E_{k+1}^{(j)}$ and $E_{k+1}^{(k+\ell+3)}$ of masses $\widetilde{\mu}_k(E_{k+1}^{(j)})=:U_k^{(-)} X_k^{(j)}$ and $\widetilde{\mu}_{k}(E_{k+1}^{(k+\ell+3)})=:U_k^{(+)} X_k^{(j)}$, respectively. By Proposition \ref{cointoss}, the relative mass split on $E_k^{(j)}$ is given by \begin{align*} \left(U_k^{(-)}, U_k^{(j)}, U_k^{(+)} \right) \sim \text{Dirichlet}\left(\beta, 1-\beta, \beta\right),
\end{align*}
and is independent of $X_k^{(j)}=\widetilde{\mu}_k(E_k^{(j)})$, since, by (i) and Proposition \ref{masssel2}, the $(\beta,\beta)$-string of beads
\begin{equation*} \left(\left(X_k^{(j)} \right)^{-\beta} E_k^{(j)}, \left( X_k^{(j)}\right)^{-1} \widetilde{\mu}_k \restriction_{E_k^{(j)}}\right) \end{equation*}
is independent of the scaling factor $X_k^{(j)}$. We obtain the refined mass split
\begin{equation}
\left(\overline{X}_k^{(1)}, \ldots, \overline{X}_k^{(j-1)}, \overline{X}_k^{(-)},  \overline{X}_k^{(j)},\overline{X}_k^{(+)}, \overline{X}_k^{(j+1)}, \ldots, \overline{X}_k^{(k+2\ell+1)}\right) \label{xbar} \end{equation}
where $\overline{X}_k^{(i)}=X_k^{(i)}$, $i \in [k+2\ell+1] \setminus \{j, +, -\}$ and $\overline{X}_k^{(-)}= U_k^{(-)}X_k^{(j)}$, $\overline{X}_k^{(j)}= U_k^{(j)}X_k^{(j)}$ and $\overline{X}_k^{(+)}= U_k^{(+)}X_k^{(j)}$. By Proposition \ref{Diri}(iii), the distribution of \eqref{xbar} is
\begin{equation*}\text{Dirichlet}\left(\beta, \ldots, \beta, \beta,1-\beta, \beta, \beta, \ldots, \beta, w(d_1), \ldots, w(d_\ell)\right).
\end{equation*}
Furthermore, the atom $\widetilde{J}_k$ induces the two rescaled independent $(\beta, \beta)$-strings of beads
\begin{equation*}
\left(\left(X_k^{(-)}\right)^{-\beta} E_{k+1}^{(j)},\left(X_k^{(-)}\right)^{-1} \widetilde{\mu}_k \restriction_{E_{k+1}^{(j)}}\right), \quad \left(\left(X_k^{(+)} \right)^{-\beta} E_{k+1}^{(k+\ell+3)},\left(X_k^{(+)}\right)^{-1} \widetilde{\mu}_k \restriction_{E_{k+1}^{(k+\ell+3)}}\right)
\end{equation*}
where $X_k^{(i)}=U_k^{(i)}X_k^{(j)}, i \in \{-, +\}$, i.e. the lengths of the edges $E_{k+1}^{(j)}$ and $E_{k+1}^{(k+\ell+3)}$ are given by \begin{equation*} \widetilde{L}_{k+1}^{(j)}=\left(U_k^{(-)}X_k^{(j)}\right)^{\beta} M_k^{(-)}, \quad \widetilde{L}_{k+1}^{(k+\ell+3)}=\left(U_k^{(+)}X_k^{(j)}\right)^{\beta} M_k^{(+)}, \end{equation*} respectively, where $M_k^{(i)} \sim \text{ML}(\beta,\beta)$, $i \in \{-,+\}$, are independent, see Proposition \ref{cointoss}. Conditionally given $\widetilde{J}_k \in E_k^{(j)}$, by \eqref{toshow} and Corollary \ref{Massesaslengths}, the weights $\widetilde{W}_k^{(i)}$ of $\widetilde{\mathcal T}_k$ are therefore
\begin{equation*}  \widetilde{W}_k^{(i-(k+\ell+1))}=\left(X_k^{(i)}\right)^\beta M_k^{(i)}, \quad i \in [k+2\ell+1]\setminus [k+\ell+1], \end{equation*}
where the lengths $\widetilde{L}_k^{(i)}$ are  
\begin{equation*} \widetilde{L}_k^{(i)}= \begin{cases} \left(X_k^{(i)}\right)^\beta M_k^{(i)},  & i \in [k+\ell+1] \setminus [j], \\ \left(U_k^{(-)}X_k^{(j)}\right)^\beta M_k^{(-)}+\left(U_k^{(+)}X_k^{(j)}\right)^\beta M_k^{(+)}, & i=j,
\end{cases} \end{equation*}
for independent random variables 
\begin{equation*} M_k^{(i)} \sim \begin{cases} \text{ML}(\beta, \beta), & i \in [k+\ell+1] \setminus \{j\} \cup \{-, +\},\\ 
\text{ML}(\beta, w(d_{i-(k+\ell+1)})), & i \in [k+2\ell+1] \setminus [k+\ell+1]. \end{cases} \end{equation*}
Also note that, by the definition of $(S_k^*, k \geq 0)$ and the attachment procedure,
\begin{equation*}
S_{k+1}^*-S_k^*= \widetilde{\mu}_k\left(\widetilde{J}_k\right)^\beta M_k^* = \left(U_k^{(j)}X_k^{(j)}\right)^{\beta} M_k^* \end{equation*}
where $M_k^*\sim {\rm ML}(\beta,1-\beta)$ is the length of the attached, independent $\beta$-mixed string of beads. We conclude by Proposition \ref{DIRML} and Proposition \ref{Diri}(i)-(ii) that 
$S_k^*=A_k^*S_{k+1}^* \sim {\rm ML}(\beta, k+\beta)$
where $S_{k+1}^* \sim {\rm ML}\left(\beta, k+1+\beta\right)$ and $A_k^* \sim {\rm Beta}(k/\beta+2, 1/\beta-1)$ are independent, and that, conditionally given $\widetilde{J}_k \in E_k^{(j)}$, we have $(\widetilde{L}_k, \widetilde{W}_k) = S_{k+1}^*  A_k^*  {\widetilde{Z}}_k$ where 
\begin{align*} \widetilde{Z}_k=\left({\widetilde{Z}}_k^{(1)}, \ldots, {\widetilde{Z}}_k^{(j-1)},{\widetilde{Z}}_k^{(j)},{\widetilde{Z}}_k^{(j+1)}, \ldots, {\widetilde{Z}}_k^{(k+\ell+1)}, {\widetilde{Z}}_k^{(k+\ell+2)}, \ldots, {\widetilde{Z}}_k^{(k+2\ell+1)}\right) \end{align*}
is independent of $S_{k+1}^*$ and $A_k^*$, and has a ${\rm Dirichlet}(1, \ldots, 1,2,1,\ldots,1, w(d_1)/\beta, \ldots, w(d_\ell)/\beta)$ distribution. Hence,
\begin{align*}
\mathbb E \left [ f\left( \widetilde{L}_k, \widetilde{W}_k \right) \mathbbm{1}_{\{\widetilde{J} _k\in E_k^{(j)}\}} \right] \nonumber =\frac{\beta}{k+\beta} \mathbb E \left [ f\left(S_k^* {\widetilde{Z}}_k\right) \middle| \widetilde{J}_k \in E_k^{(j)} \right].
\end{align*}

We now consider the RHS of \eqref{equalexp1}. We condition on $J_k \in E_k^{(j)}$, and apply Lemma \ref{GH1} and Proposition \ref{Diri}(iv) to obtain
\begin{align*}
\mathbb E &\left[ f\left( L_k^{(1)}, \ldots, L_k^{(k+\ell+1)}, W_k^{(1)}, \ldots, W_k^{(\ell)}\right) \mathbbm{1}_{\{J_k \in E_k^{(j)}\}}\right] =\frac{\beta}{k+\beta} \mathbb E \left[ f\left( S_k Z_k\right) \middle| {J}_k \in E_k^{(j)} \right] 
\end{align*}
where $S_k \sim \text{ML}(\beta, k+\beta)$ is independent of 
\begin{align*} Z_k=\left(Z_k^{(1)}, \ldots, Z_k^{(j-1)}, Z_k^{(j)}, Z_k^{(j+1)}, \ldots, Z_k^{(k+\ell+1)},  Z_k^{(k+\ell+2)}, \ldots,  Z_k^{(k+2\ell+1)} \right)  \end{align*}
and $Z_k \sim \text{Dirichlet}(1, \ldots,1,2,1,\ldots, 1,w(d_1)/\beta, \ldots, w(d_\ell)/\beta)$. Hence, we conclude \eqref{equalexp1}.

\vspace{0.5cm}

\item \textit{Proof of \eqref{equalexp2}}. Consider now the LHS of \eqref{equalexp2}. We follow the lines of the proof of \eqref{equalexp1}. Conditionally given $\widetilde{J}_k = \widetilde{v}_j$, the mass split \eqref{xmasses} has distribution
\begin{align}
\text{Dirichlet}\left(\beta, \ldots, \beta, w(d_1), \ldots, w(d_{j-1}),1+w(d_j), w({d_j+1}),\ldots, w(d_\ell)\right). \label{massintikay2}
\end{align}
By (i), the mass $\widetilde{\mu}_k(\widetilde{v}_j)$ is split by an independent vector $\left(Q_k, 1-Q_k\right) \sim {\rm Dirichlet}\left(\beta, w(d_j)+1-\beta\right)$ into a new unmarked $(\beta, \beta)$-string of beads 
\begin{equation*} \left( \widetilde{\mu}_{k+1}\left(E_{k+1}^{(k+l+2)}\right)^{-\beta}  E_{k+1}^{(k+\ell+2)}, \widetilde{\mu}_{k+1}\left(E_{k+1}^{(k+\ell+2)}\right)^{-1} \widetilde{\mu}_{k+1}\restriction_{E_{k+1}^{(k+\ell+2)}} \right) \end{equation*}
attached to $\widetilde{v}_j$ and the mass $\widetilde{\mu}_{k+1}(\widetilde{v}_j)$, i.e. 
$\widetilde{\mu}_{k+1}(E_{k+1}^{(k+\ell+2)})=Q_k\widetilde{\mu}_{k}(\widetilde{v}_j)$, $\widetilde{\mu}_{k+1}(\widetilde{v}_j) = (1-Q_k) \widetilde{\mu}_{k}(\widetilde{v}_j)$.

By (i), Proposition \ref{starstrings} and the induction hypothesis, we get $\widetilde{W}_{k+1}^{(i-(k+\ell+2))}=\widetilde{\mu}_{k+1}(\widetilde{v}_{i-(k+\ell+2)})^{\beta}M_{k+1}^{(i)}$ for $i \in [k+2\ell+2]\setminus [k+\ell+2], i \neq j+(k+\ell+2)$ and $\widetilde{L}_{k+1}^{(i)}=\widetilde{\mu}_{k+1}\left(E_{k+1}^{(i)}\right)^{\beta} M_{k+1}^{(i)}$ for $i \in [k+\ell+2]$ where the random variables $M_{k+1}^{(i)} \sim {\rm ML}(\beta, \beta)$ for $i  \in [k+\ell+2]$, $M_{k+1}^{(i)} \sim {\rm ML}(\beta, w(d_{i-(k+\ell+2)}))$ for $i \in [k+2\ell+2] \setminus [k+\ell+2], i \neq j+(k+\ell+2)$ and $M_{k+1}^{(i)} \sim {\rm ML}(\beta, w(d_{i-(k+\ell+2)})+(1-\beta))$ for $i=j+(k+\ell+2)$, are independent of the mass split \eqref{massintikay2}. By Proposition \ref{starstrings},
$\widetilde{W}_{k}^{(j)} = B_k^*  \widetilde{W}_{k+1}^{(j)}$, 
where $B_k^* \sim {\rm Beta}((d_j-2)(1/\beta-1) , 1/\beta-2)$ is independent of $\widetilde{W}_{k+1}^{(j)}$. Note that $\widetilde{W}_{k+1}^{(i)}=\widetilde{W}_{k}^{(i)}$ for $i \neq j$, $\widetilde{L}_{k+1}^{(i)}=\widetilde{L}_{k}^{(i)}$ for $i \in [k+\ell+1]$, and hence, 
$S_{k+1}^*-S_k^*=\widetilde{L}_{k+1}^{(k+\ell+2)}+ \left(\widetilde{W}_{k+1}^{(j)}-\widetilde{W}_{k}^{(j)}\right)$.
By Proposition \ref{DIRML}, $S_k^*=A_k^*S_{k+1}^*$ where $S_{k+1}^* \sim {\rm ML}\left(\beta, k+1+\beta\right)$ and $A_k^* \sim {\rm Beta}(k/\beta+2, 1/\beta-1)$ are independent. The rest of the proof of \eqref{equalexp2} is analogous to the proof of \eqref{equalexp1}.\vspace{-0.6cm}
\end{itemize}
\end{proof}

\subsection{Proof of Theorem \ref{embford}} \label{sec71}

We first consider the evolution of marked subtrees $(\mathcal R_k^{(i)},k \geq 1)$, $i \geq 1$. Recall the notation in Algorithm \ref{twocolourmass}. Given that $\mathcal R_k^{(i)}$ has size $m$, i.e. $k_m^{(i)} \leq k \leq k_{m+1}^{(i)}-1$, we denote the edges and the edge lengths of $\mathcal R_k^{(i)}$ by
\begin{equation}
E_{m,i}=\left(E_{m,i}^{(1)}, \ldots, E_{m,i}^{(2m-1)}\right), \quad L_{m,i}=\left(L_{m,i}^{(1)}, \ldots, L_{m,i}^{(2m-1)}\right),
\end{equation}
respectively, where we note that $\mathcal R_{k}^{(i)}$ is a binary tree, i.e. it has $2m-1$ edges for $k_m^{(i)} \leq k \leq k_{m+1}^{(i)}-1$. Recall that $E_{m,i}^{(j)}$ is an internal edge of $\mathcal R_k^{(i)}$ if $1 \leq j \leq m-1$, and an external edge if $m \leq j \leq 2m-1$.

\begin{lemma}[Mass split in marked subtrees]\label{mms} Let $(\mathcal T_k^*, (\mathcal R_k^{(i)}, i \geq 1), \mu_k^*, k \geq 0)$ be as in Algorithm \ref{twocolourmass}, and fix some $i \geq 1$. Then, for $m \geq 1$, conditionally given $k_m^{(i)}=k$, the relative mass split in $\mathcal R_k^{(i)}$ given by
\begin{equation}
\mu_k^{*}\left( \mathcal R_{k}^{(i)}\right)^{-1} \left(\mu_k^{*}\left(E_{m,i}^{(1)}\right), \ldots, \mu_k^{*}\left(E_{m,i}^{(m-1)}\right),\mu_k^{*}\left(E_{m,i}^{(m)}\right), \ldots, \mu_k^{*}\left(E_{m,i}^{(2m-1)}\right) \right)  \label{msms}
\end{equation}
has a $\rm{Dirichlet}\left(\beta, \ldots, \beta, 1-2\beta, \ldots, 1-2\beta \right)$ distribution and is independent of $\mu_k^*(\mathcal R_k^{(i)})$ and of the mass split in $\mathcal T_{k}^* \setminus \mathcal R_k^{(i)}$. Furthermore, for $j \in [2m-1]$,
\begin{equation}
\left(  \mu_k^*\left( E_{m,i}^{(j)} \right)^{-\beta}
E_{m,i}^{(j)}, \mu_k^*\left( E_{m,i}^{(j)} \right)^{-1} \mu_k^* \restriction_{E_{m,i}^{(j)}} \right) \label{mssob1}
\end{equation}
is a $(\beta, \theta)$-strings of beads, where $\theta=\beta$ for $j \in [m-1]$ and $\theta=1-2\beta$ for $j \in [2m-1] \setminus [m-1]$. The strings of beads \eqref{mssob1} are independent of each other and of the mass split in $\mathcal R_k^{(i)}$ given by \eqref{msms}. 
Conditionally given that $k_{m+1}^{(i)}=k'$,
\begin{equation}
\mu_{k'}^*\left(\mathcal R_{k'}^{(i)}\right)=\left(1-Q_m^{(i)}\right)\mu_{k}^*\left(\mathcal R_{k}^{(i)}\right) \label{mssob3}
\end{equation}
where $Q_m^{(i)} \sim {\rm{Beta}}(\beta, m(1-\beta)+1-2\beta)$ is independent of $\mathcal R_{k'}^{(i)}$ normalised to unit mass.
\end{lemma}

\begin{proof} This is a direct consequence of Proposition \ref{starstrings}, and Proposition \ref{Diri}(ii). To see \eqref{mssob3}, note that $\mathcal R_{k'}^{(i)} \setminus \mathcal R_k^{(i)}=E_{m+1, i}^{(2m)}$ and that $\mu_{k'}^*(E_{m+1,i}^{(2m)})=\gamma_k\mu_k^*(J_k^*)$ where $\gamma_k \sim {\rm Beta}(1-2\beta, \beta)$ is independent, and apply Proposition \ref{Diri}(i)-(ii). \end{proof}

\begin{corollary}[Length split in marked subtrees] \label{lms} In the setting of Lemma \ref{mms}, let $\widetilde{S}_{m,i}\!=\!\sum_{j\in[2m-1]} L_{m,i}^{(j)}$ denote the total length of $\mathcal R_{k_m^{(i)}}^{(i)}$, $m \geq 1$. Then, conditionally given $k_m^{(i)}=k$,
\begin{equation}
\left( L_{m,i}^{(1)}, \ldots, L_{m,i}^{(m-1)}, L_{m,i}^{(m)}, \ldots, L_{m,i}^{(2m-1)}\right)= \mu^*_k\left(\mathcal R_{k}^{(i)} \right)^{\beta} S_{m,i} \cdot \left(Z_{m,i}^{(1)}, \ldots, Z_{m,i}^{(m-1)},Z_{m,i}^{(m)}, \ldots, Z_{m,i}^{(2m-1)} \right)\label{nl1}
\end{equation}
where $\mu^*_{k}(\mathcal R_{k}^{(i)} )$, $S_{m,i}\sim {\rm{ML}}(\beta, (m-1)(1-\beta)+1-2\beta)$ and
\begin{equation*} \left(Z_{m,i}^{(1)}, \ldots, Z_{m,i}^{(m-1)},Z_{m,i}^{(m)}, \ldots, Z_{m,i}^{(2m-1)} \right) \sim {\rm{Dirichlet}}\left( 1, \ldots, 1, {1}/{\beta}-2, \ldots, {1}/{\beta}-2\right)\end{equation*} are independent.
In particular, $\widetilde{S}_{m,i}=\mu_{k}^*(\mathcal R_k^{(i)})^\beta S_{m,i}$. Furthermore, for $m \geq 1$, 
\begin{equation} \widetilde{S}_{m,i}=B_{m,i} \widetilde{S}_{m+1,i} \label{claimed}
\end{equation} where $B_{m,i} \sim {\rm Beta}(m(1/\beta-1), 1/\beta-2)$ and $\widetilde{S}_{m+1,i}$ are independent, i.e. the sequence of lengths of each marked subtree 
is a Markov chain with the same transition rule as the Mittag-Leffler Markov chain with parameter $\beta/(1-\beta)$ starting from 
${\rm ML}(\beta/(1-\beta), (1-2\beta)/(1-\beta))$.
\end{corollary}

\begin{proof} Fix $i \geq 1$, and set
$X_j= \mu_{k_m^{(i)}}^*(E_{m,i}^{(j)})$, $j \in [2m-1]$, so that $\sum_{j\in[2m-1]}X_j=\mu_{k_m^{(i)}}^*(\mathcal R_{k_m^{(i)}}^{(i)})$. 
By Lemma \ref{mms}, the edge lengths $L_{m,i}^{(j)}$, $j \in [2m-1]$, are given by $L_{m,i}^{(j)}=X_j^\beta M_m^{(j)}$ where $M_m^{(j)} \sim \text{ML}(\beta, \beta)$ for $j \in [m-1]$, $M_m^{(j)} \sim \rm{ML}(\beta, 1-2\beta)$ for $j \in [2m-1] \setminus [m-1]$, $\sum_{j\in[2m-1]}X_j$ and \begin{equation*} \left(\sum_{j\in[2m-1]} X_j\right)^{-1} \left(X_1, \ldots, X_{m-1}, X_m, \ldots X_{2m-1}\right) \sim {\text{Dirichlet}}\left(\beta, \ldots, \beta, 1-2\beta, \ldots, 1-2\beta\right) \end{equation*} are independent.
We apply Proposition \ref{DIRML} with $n=2m-1$, $\theta_j = \beta$ for $j \in [m-1]$ and $\theta_j=1-2\beta$ for $j \in [2m-1] \setminus [m-1]$ to the vector \begin{align}
\left( L_{m,i}^{(1)}, \ldots, L_{m, i}^{(2m-1)}\right) =\left(\sum_{j\in[2m-1]}X_j\right)^\beta \left( \left(\frac{X_1}{\sum_{j\in[2m-1]}X_j}\right)^\beta M_m^{(1)}, \ldots,\left(\frac{X_{2m-1}}{\sum_{j\in[2m-1]}X_j}\right) ^\beta M_m^{(2m-1)}\right)  \label{lab1}
\end{align}
Then $\theta=(m-1)(1-\beta)+1-2\beta$, and hence \eqref{nl1} follows. 

To see \eqref{claimed}, recall that $E_{m+1,i}^{(2m)}=\mathcal R_{k_{m+1}^{(i)}}^{(i)} \setminus \mathcal R_{k_{m}^{(i)}}^{(i)}$. By \eqref{lab1} for $m+1$, and Proposition \ref{Diri}(i)-(ii),
$\mu_{k_m^{(i)}}^*\big(\mathcal R_{k_m^{(i)}}^{(i)}\big)^{\beta}S_{m,i}=B_{m,i}
\mu_{k_{m+1}^{(i)}}^*\big(\mathcal R_{k_{m+1}^{(i)}}^{(i)}\big)^{\beta} S_{m+1,i}$ 
where the variables $S_{m+1,i}\sim {\rm ML}(\beta, m(1-\beta)+1-2\beta)$, $B_{m,i} \sim \text{Beta}(m(1/\beta-1), 1/\beta -2 )$ and  $\mu_{k_{m+1}^{(i)}}^*(\mathcal R_{k_{m+1}^{(i)}}^{(i)})$, are independent, i.e. $\widetilde{S}_{m,i}=B_{m,i}\widetilde{S}_{m+1,i}$.
\end{proof}

\begin{proof}[Proof of Theorem \ref{embford}]

(i) Consider a space $\mathbb T_{[m]}$ of weighted discrete $\mathbb R$-trees $(\mathcal T, \mu)$ with $m$ leaves labelled by $[m]$ and mass measure $\mu$ of total mass $\mu(\mathcal T) \in (0,1]$, $m\geq 1$, see e.g. \cite[Section 3.3]{1} for a formal introduction. We define transition kernels $\kappa_m$ from $\mathbb T_{[m]}$ to $\mathbb T_{[m+1]}$, $m \geq 1$: given any $(\mathcal T, \mu) \in \mathbb T_{[m]}$, 
\begin{itemize}
\item select an edge $E$ of $\mathcal T$ according to the normalised mass measure $\mu(\mathcal T)^{-1} \mu$; given $E$, select an atom $J$ of $\mu \restriction_E$ according to $(\beta, \theta)$-coin tossing sampling where $\theta=\beta$ if $E$ is internal, and $\theta=1-2\beta$ if $E$ is external; this determines a selection probability $p_m(x)$ for each atom $x \in \mathcal T$; 
\item  given $J$, let $\gamma \sim {\rm Beta}(1-2\beta, \beta)$ be independent, and attach to $J$ an independent $(\beta, 1-2\beta)$-string of beads with mass measure rescaled by $\gamma\mu(J)$ and metric rescaled by $(\gamma\mu(J))^{\beta}$, and label the new leaf by $m+1$.
\end{itemize}

We use the convention that if no atom is selected, we apply a scaling factor of $0$. Note that, in our setting with $(\beta, \beta)$-strings of beads on internal edges and $(\beta,1-2\beta)$-strings of beads on external edges, this does not happen almost surely. Denote by $\kappa_m((\mathcal T, \mu), \cdot)$ the distribution of the resulting tree. We further consider the kernel $\kappa_{0}(\cdot)=\kappa_0((\{\rho\}, \delta_\rho), \cdot)$ taking the singleton tree $\{\rho\}$ of mass $1$, and associating a $(\beta, 1-2\beta)$-string of beads with $\{\rho\}$. We will show that each process in \eqref{embford1} evolves according to the transition kernels $\kappa_m, m \geq 1$, starting from an independent $(\beta,1-2\beta)$-string of beads whose distribution is given by $\kappa_0(\cdot)$.

More formally, for $\ell \geq 1$ and some $m_i \geq 1$, $i \in [\ell]$, we will show that 
\begin{align} \mathbb E &\left[\prod_{i \in [\ell]} f_i \left( \left( \mathcal G_m^{(i)}, \mu_m^{(i)}\right), m \in [m_i] \right) \right] \nonumber \\ & \hspace{-0.2cm} = \prod_{i \in [\ell]} \int \int \cdots \int f_i \left( R_1, \ldots,  R_{m_i}\right) \kappa_{m_i-1}\left( R_{m_i-1},d  R_{m_i}\right) \cdots \kappa_1\left( R_1, d  R_2\right) \kappa_0 \left(d  R_1\right) \label{product} \end{align}
for any bounded continuous functions $f_i \colon \mathbb T_{[1]} \times \cdots \times \mathbb T_{[m_i]} \rightarrow \mathbb R$, $i \in [\ell]$.

We first show the equation \eqref{product} for $\ell=1$. For notational convenience, we write $( \mathcal G_m, \mu_m)=( \mathcal G_m^{(1)}, \mu_m^{(1)})$ and $f=f_1$. We further use the notation $\xi_{\beta, \beta}$ and $\xi_{\beta, 1-2\beta}$ for $(\beta, \beta)$- and $(\beta, 1-2\beta)$-strings of beads, respectively, and recall that we denote by $p_m(x)$ the selection probability of $x \in \mathcal T$ for $\mathcal T \in \mathbb T_{[m]}$ using the edge selection rule in combination with coin tossing sampling, as described above. $B_{\beta, 1-2\beta}(\cdot)$ denotes the density of Beta$(\beta, 1-2\beta)$.
We obtain, 
\begin{align*}
\mathbb E \left[ f \left( \mathcal G_1, \ldots, \mathcal G_{m_1} \right) \right] 
  =& \sum \limits_{k_1^{(1)}=1, k_2^{(1)}, \ldots, k_{m_1}^{(1)}} \int \limits_{\xi_0}  \sum_{v \in \mathcal \xi_0} \mu_0 \left(v\right) \int \limits_{x_1} {B}_{\beta,1-2\beta}(x_1) \\ 
   &\int \limits_{\xi_1} \left(1-\mu_0\left(v\right)\left(1-\overline{x}_1\right)\right)^{k_2^{(1)}-k_1^{(1)}-1} \mu_0\left(v\right)\left(1-\overline{x}_1\right) \sum \limits_{w_1 \in R_1} p_1\left(w_1\right) \int \limits_{x_2}  {B}_{\beta,1-2\beta}(x_2) \\ 
   & \int \limits_{\xi_2}  \cdots  \left(1-\mu_0\left(v\right)\prod \limits_{i \in [m_1-1]}\left(1-\overline{x}_i\right)\right)^{k_{m_1}^{(1)}-k_{m_1-1}^{(1)}-1} \mu_0\left(v\right)\prod \limits_{i \in [m_1-1]}\left(1-\overline{x}_i\right)  \\ 
   & \qquad\qquad\sum \limits_{w_{m_1-1} \in R_{m_1-1}} p_{m_1-1}\left(w_{{m_1}-1}\right) \int \limits_{x_{m_1}}  {B}_{\beta,1-2\beta}(x_{m_1}) \int \limits_{\xi_{m_1}} f\left( R_1, \ldots, R_{m_1} \right)  \\ 
   & \mathbb P\left(\xi_{\beta, 1-2\beta} \in d\xi_{m_1}\right)dx_{m_1} \cdots \mathbb P\left(\xi_{\beta, 1-2\beta} \in d\xi_2\right) dx_2 \mathbb P\left(\xi_{\beta, 1-2\beta} \in d\xi_1\right)dx_1  \mathbb P\left(\xi_{\beta, \beta} \in d\xi_0\right)\end{align*}
where 
\begin{itemize}
\item $\mu_0$ is the mass measure of $\xi_0$;
\item $R_1=\xi_1$ with mass measure $\mu_1^{(1)}$ is the initial string of beads, and, for $m \geq 2$, $R_m$ with mass measure $\mu_m^{(1)}$ is created by attaching to $w_{m-1}\in R_{m-1}$ the string of beads $\xi_m$ rescaled by the proportion $x_{m-1}$ of the mass of $w_{m-1}$;
\item the sequence $(\overline{x}_i, i \geq 1)$ is defined by $\overline{x}_1=x_1$, $\overline{x}_i = 1-\frac{\mu_{i-1}^{(1)}(w_{i-1})}{\mu^{(1)}_{i-1}(R_{i-1})}(1-x_{i})$, $i=2, \ldots, m_1$;
\item the integrals are taken over the whole ranges of $x_i \in [0,1]$ and the subspaces of $\xi_i\in\mathbb T_{\rm w}$ that correspond to strings of beads.
\end{itemize}

Note that $\mu_0\left(v\right)\prod_{i \in [m-1]} (1-\overline{x}_i)$ 
is the relative remaining mass of the first marked component after $m$ transition steps have been carried out in this component.

We can move the sum over $k_1^{(1)}, \ldots, k_{m_1}^{(1)}$ inside the integrals, and note that there is only one term which depends on $k_{m_1}^{(1)}$. Moving the sum over $k_{m_1}^{(1)}$ in front of this factor, we obtain 
\begin{equation*}
\sum \limits_{k_{m_1}^{(1)} \geq k_{m_1-1}^{(1)}+1 } \left(1-\mu_0\left(v\right) \prod \limits_{i \in [m_1-1]} \left(1-\overline{x}_i\right)\right)^{k_{m_1}^{(1)} - k_{m_1-1}^{(1)}-1}\mu_0\left(v\right) \prod \limits_{i \in [m_1-1]} \left(1-\overline{x}_i\right) =1
\end{equation*}
as this is the sum over the probability mass function of a geometric random variable (there are infinitely many insertions into the first marked component almost surely). We can proceed inductively and sum the corresponding geometric probabilities over $k_1^{(1)}, \ldots, k_{m_1-1}^{(1)}$ to obtain
\begin{align*}
\mathbb E \left[ f \left( \mathcal G_1, \ldots, \mathcal G_{m_1} \right) \right] =
  &\int \limits_{\xi_0}  \sum_{v \in \mathcal \xi_0} \mu_0\left(v\right) \int \limits_{x_1} {B}_{\beta,1-2\beta}(x_1) \int \limits_{\xi_1} \sum \limits_{w_1 \in R_1} p_1\left(w_1\right)  \int \limits_{x_2}  {B}_{\beta,1-2\beta}(x_2) \\
  & \int \limits_{\xi_2}  \cdots \sum \limits_{w_{m_1-1} \in R_{m_1-1}}  p_{m_1-1}\left(w_{{m_1}-1}\right) \int \limits_{x_{m_1}}  {B}_{\beta,1-2\beta}(x_{m_1}) \int \limits_{\xi_{m_1}}   f\left(R_1, \ldots, R_{m_1}\right) \\
  & \mathbb P\left(\xi_{\beta, 1-2\beta} \in d\xi_{m_1}\right)dx_{m_1} \cdots \mathbb P\left(\xi_{\beta, 1-2\beta} \in d\xi_2\right) dx_2 \mathbb P\left(\xi_{\beta, 1-2\beta} \in d\xi_1\right)dx_1 \mathbb P\left(\xi_{\beta, \beta} \in d\xi_0\right).
\end{align*}

We can now take the sum $ \sum_{v \in \xi_0} \mu_0(v)=1$ and the outer integral, as the inner terms are independent of $\mu_0(v)$ and $\xi_0$. This results in  
\begin{align*}
\mathbb E \left[ f \left( \mathcal G_1, \ldots, \mathcal G_{m_1} \right) \right] =
  & \int \limits_{x_1} {B}_{\beta,1-2\beta}(x_1) \int \limits_{\xi_1} \sum \limits_{w_1 \in R_1} p_1\left(w_1\right) \int \limits_{x_2}  {B}_{\beta,1-2\beta}(x_2) \\
  &  \int \limits_{\xi_2}  \cdots\sum \limits_{w_{m_1-1} \in R_{m_1-1}}  p_{m_1-1}\left(w_{{m_1}-1}\right)  \int \limits_{x_{m_1}}  {B}_{\beta,1-2\beta}(x_{m_1})  \int \limits_{\xi_{m_1}}  f\left(R_1, \ldots, R_{m_1}\right)  \\
  & \mathbb P\left(\xi_{\beta, 1-2\beta} \in d\xi_{m_1}\right)dx_{m_1} \cdots  \mathbb P\left(\xi_{\beta, 1-2\beta} \in d\xi_2\right) dx_2 \mathbb P\left(\xi_{\beta, 1-2\beta} \in d\xi_1\right)dx_1. 
\end{align*}
We recognise the definition of the transition kernels $\kappa_m, m \geq 1$, and rewrite this integral in the form \begin{align*}
\mathbb E &\left[f \left( \mathcal G_1, \ldots, \mathcal G_{m_1} \right) \right]  = \int \limits \int \limits \cdots \int \limits f\left(R_1, \ldots, R_m\right)   \kappa_{m-1}\left(R_{m-1}, dR_m\right) \cdots \kappa_1\left(R_1, dR_2\right) \kappa_0\left(dR_1\right) \end{align*}
 
To see \eqref{product} in the general setting, we express the left-hand side in terms of the distribution of $(\mathcal T_0^*, \mu_0^*)$ and the two-colour transition kernels, which can be described via Algorithm \ref{twocolourmass}, as a sum over $k_j^{(i)}, j\in [m_i], i \in [\ell]$. Then we can proceed as follows.
\begin{itemize}
\item First integrate out irrelevant transitions which affect components $i \geq \ell+1$ and parts of earlier transitions such as unmarked strings of beads after the creation of the $\ell$th component. These transitions do not affect the marked components $i \in [\ell]$.

\item Move the sums over $k_{m_\ell}^{(\ell)}, \ldots, k_2^{(\ell)}$ inside the integrals. Notice that there is only one term depending on $k_{m_\ell}^{(\ell)}$, i.e. we obtain the sum over $k_{m_\ell}^{(\ell)} \geq k_{m_\ell-1}^{(l)}$, $k_{m_\ell}^{(l)} \neq k_j^{(i)}, j \in [m_i], i \in [\ell-1]$ of the probabilities of selecting the $\ell$th marked component at step $k_{m_\ell}^{(\ell)}$, skipping indices $k_j^{(i)}$ of insertions into other marked components $i \in [\ell-1]$, i.e.
\begin{align*}
\sum \limits_{k_{m_\ell}^{(\ell)} \geq k_{m_\ell-1}^{(\ell)}+1, k_{m_\ell}^{(l)} \neq k_j^{(i)}, j \in [m_i], i \in [\ell-1]} &\left( 1- \mu_{k_1^{(\ell)}-1}\left(v_\ell\right) \prod_{r \in [m_\ell-1]}\left(1-\overline{x}_r^{(\ell)}\right)\right)^{ k(m,\ell)}   \\&\mu_{k_1^{(\ell)}-1}\left(v_\ell\right) \prod_{r \in [m_\ell-1]}\left(1-\overline{x}_r^{(\ell)}\right),
\end{align*}
where $k(m,\ell):=k_{m_\ell}^{(\ell)}-k_{m_\ell-1}^{(\ell)}-\#\{k_{m_\ell -1}^{(\ell )} < k < k_{m_\ell }^{(\ell )}\colon k=k_j^{(i)}, j \in [m_i], i \in [\ell -1]\}$, 
and where the sequences $({x}_{i}^{(\ell )}, i \geq 1)$ and $(\overline{x}_{i}^{(\ell )}, i \geq 1)$ are defined as $(x_i, i \geq 1)$ and $(\overline{x}_i, i \geq 1)$, respectively. Note that
\begin{equation*}
\mu_{k_1^{(\ell )}-1}\left(v_\ell \right) \prod_{r \in [m-1]}\left(1-\overline{x}_r^{(\ell )}\right) \end{equation*}
is the mass of the $\ell $th marked component after $m$ transition steps have been carried out in this component. As we have a sum over the probability mass function of a geometric random variable, no matter when insertions into components $i \in [\ell -1]$ happen, this sum is $1$. We can proceed inductively down to $k_2^{(\ell )}$.

\item The sum over the insertion point $v_\ell $ is just a sum over the bead selection probabilities $$\mu_{k_1^{(\ell )}-1}(v_\ell ), \quad k_1^{(\ell )} \geq k_1^{(\ell -1)}+1,$$ which sum to the probability of creating the $\ell $th component (no matter what the sizes of the other components are at this step). The sum over $k_1^{(\ell )}$ is not geometric but it is a sum over the probabilities of success in a Bernoulli sequence with increasing success probability. This sum is again $1$ (as we will open the $\ell $th marked component with probability one).

\item We can put the integrals over the ingredients for the $\ell $th subtree growth process in front of the other integrals, as they do not depend on anything else.

\item Inductively, for $j=\ell -1, \ldots, 1$, repeat these steps to lose all sums over insertion times $k_i^{(j)}$ and first insertion points $v_i$, $i \in [\ell ]$. 

\item Finally, the integrand of the outer integral over the distribution of $\xi_0$ is constant, so the integral can be dropped. We obtain precisely the product form of the right-hand side \eqref{product}. \qedhere
\end{itemize}

(ii) Note that, by Lemma \ref{mms} (and Proposition \ref{starstrings}), for each $i$ and $k=k_m^{(i)}-1$ for some $m \geq 1$, we are in the situation of Lemma \ref{equivalence} with $n=2m-1$, $\theta_1=\cdots=\theta_{m-1}=\beta$, $\theta_m =\cdots=\theta_{2m-1}=1-2\beta$, $\alpha=\beta$. We recover Algorithm \ref{Ford} with index $\beta'=\beta/(1-\beta)$ and the \enquote{wrong} starting length ${\rm ML}(\beta,1-2\beta)$, cf. Corollary \ref{lms}.

\medskip

(iii) First, note that, by Corollary \ref{lms}, the lengths of the trees $C_m^{(i)} \mathcal R_{k_m^{(i)}}^{(i)}$ do not depend on $\mu_{k_m^{(i)}}^*(\mathcal R_{k_m^{(i)}}^{(i)})$. Fix some $i \geq 1$ and recall from Lemma \ref{mms} that there are independent random variables $Q_m^{(i)}\sim {\rm Beta}(\beta, m(1-\beta)+1-2\beta)$ such that 
\begin{equation*}
\mu_{k_{m+1}^{(i)}}^*\left(\mathcal R_{k_{m+1}^{(i)}}^{(i)}\right)=\left(1-Q_m^{(i)}\right)\mu_{k_{m}^{(i)}}^*\left(\mathcal R_{k_{m}^{(i)}}^{(i)}\right), \quad m \geq 1.
\end{equation*}
Define $P_1^{(i)}:=Q_1^{(i)}\sim {\rm Beta}(\beta, 2-3\beta)$, and, for $m \geq 1$, define $P_m^{(i)}:= \overline{Q}_1^{(i)} \overline{Q}_2^{(i)} \cdots  \overline{Q}_{m-1}^{(i)} Q_m^{(i)}$, 
where $\overline{Q}=1-Q$ for any random variable $Q$. Note that $P_m^{(i)}$ is the proportion of the mass of $\mu_{k_1^{(i)}}^{*}(\mathcal R_{k_1^{(i)}}^{(i)})$ attached to the $(m+1)$st leaf of the $i$th marked component. 

We recognise the stick-breaking construction \eqref{v1v2} of a PD$(1-\beta, 1-2\beta)$ vector $(P_m^{(i)}, m \geq 1)^{\downarrow}$, and obtain the corresponding $(1-\beta)$-diversity $H^{(i)}$ by
\begin{equation}
H^{(i)}=\lim_{m \rightarrow \infty} \left(1-\sum_{j\in[m]} P_j^{(i)}\right)^{1-\beta} \left(1-\beta\right)^{-(1-\beta)}m^{\beta} \sim {\rm ML}\left(1-\beta, 1-2\beta \right),
\end{equation}
as in \eqref{alphadiv}. Now fix some $m_0 \geq 1$ and let $k \geq k_{m_0}^{(i)}$. We consider the reduced tree

\begin{equation}
\mathcal R\left(C_m^{(i)} \mathcal R_{k_m^{(i)}}^{(i)}, \Omega_1^{(i)}, \ldots, \Omega_{m_0}^{(i)} \right) \label{red516} \end{equation}
spanned by the root $v_i$ and the leaves $\Omega_1^{(i)}, \ldots, \Omega_{m_0}^{(i)}$ of $\mathcal R_k^{(i)}$. Recall from (i), Corollary \ref{lms} and Proposition \ref{Forddis} that the shape and the Dirichlet$(1, \ldots, 1, 1/\beta-2, \ldots, 1/\beta-2)$ length split between the edges $E_{m_0,i}^{(i)}, \ldots, E_{m_0,i}^{(2m_0-1)}$ of $\mathcal R_{k_{m_0}^{(i)}}^{(i)}$ are as required for the reduced tree associated with a Ford CRT of index $\beta'$.
Scaling by $C_m^{(i)}$ only affects the total length of the reduced tree \eqref{red516}. We will show that the total length of \eqref{red516} scaled by $C_m^{(i)}$ converges a.s. to some $S_{m_0}' \sim {\rm ML}(\beta', m_0-\beta')$, which is the total length of the reduced tree spanned by the root and the first $m_0$ leaves of a Ford CRT of index $\beta'$, i.e. that 
\begin{equation*} \lim_{m \rightarrow \infty}  {\rm Leb}\left(\mathcal R\left(C_m^{(i)} \mathcal R_{k_m^{(i)}}^{(i)}, \Omega_1^{(i)}, \ldots, \Omega_{m_0}^{(i)} \right) \right) = S_{m_0}' \sim {\rm ML}\left(\beta', m_0-\beta'\right) \end{equation*}
where we will use that 
$$C_m^{(i)}:= \left(1-\beta\right)^{\beta}m^{-\beta^2/(1-\beta)}  \mu_{k_m^{(i)}}^*\left(\mathcal R_{k_m^{(i)}}^{(i)}\right)^{-\beta}=\left(1-\sum_{j\in[m]} P_j^{(i)}\right)^{-\beta} \left(1-\beta\right)^{\beta}m^{-\beta^2/(1-\beta)}  \mu_{k_1^{(i)}}^*\left(\mathcal R_{k_1^{(i)}}^{(i)}\right)^{-\beta} \label{uhusd},$$
since $1-\sum_{j\in[m]} P_j^{(i)}=\mu_{k_m^{(i)}}^*(\mathcal R^{(i)}_{k_m^{(i)}})/\mu_{k_1^{(i)}}^*(\mathcal R^{(i)}_{k_1^{(i)}})$. Hence $\lim \limits_{m \rightarrow \infty} C_i(m)=(H^{(i)})^{-\beta/(1-\beta)}\mu_{k_1^{(i)}}^*(\mathcal R_k^{(i)})^{-\beta}$ a.s.. Note that $H^{(i)}$ is independent of $\mu_{k_1^{(i)}}^*(\mathcal R_k^{(i)})$ as it only depends on the sequence of independent random variables $(Q_i, i \geq 1)$ which is independent of $\mu_{k_1^{(i)}}^*(\mathcal R_k^{(i)})$.

The shape of $\mathcal R^{(i)}_{k_{m}^{(i)}}$ has the same distribution as the shape of $\mathcal F_m$ where $(\mathcal F_m, m \geq 1)$ is a Ford tree growth process of index $\beta'$. In particular, we already know that the number of edges $N_m+2m_0-1, m \geq m_0$, of the reduced trees \eqref{red516} as a subset of $\mathcal R_{k_m^{(i)}}^{(i)}$ behaves like the number of tables in a $(\beta', m_0-\beta')$-CRP, started at $m_0$, i.e.\ by \eqref{tables},
\begin{equation}
\lim_{m \rightarrow \infty} \left(m-m_0\right)^{-\beta/(1-\beta)} N_m =\lim_{m \rightarrow \infty} m^{-\beta/(1-\beta)} N_m =S'_{m_0}  \quad\text{ a.s. }  \label{scaling1}
\end{equation}
where $S'_{m_0}  \sim {\rm ML}(\beta', m_0-\beta')$. By Lemma \ref{mms}, we conclude that, in the limit, the length of $\mathcal R_{k_{m_0}^{(i)}}^{(i)}$ is given by \begin{equation} \lim_{m \rightarrow \infty} \mu^*_{k_m^{(i)}}\left(\mathcal R_{k_m^{(i)}}^{(i)}\right)^{\beta} \sum_{j\in[N_m]} \left(X'_j\right)^{\beta} M^{(j)} = \mu_{k_{m_0}^{(i)}}^*\left(\mathcal R_{k_{m_0}^{(i)}}^{(i)} \right)^\beta S_{m_0,i} =\widetilde{S}_{m_0,i}\label{scaling2} \end{equation} where $X':=(X_1', \ldots, X_{m-1}', X_m', \ldots, X_{2m-1}') \sim {\rm Dirichlet}(\beta, \ldots, \beta, 1-2\beta, \ldots, 1-2\beta)$,  $\mu^*_{k_m^{(i)}}(\mathcal R_{k_m^{(i)}}^{(i)})$, $(N_m, m \geq 1)$ and the i.i.d. random variables $M^{(j)} \sim {\rm ML}(\beta, \beta), j \geq 1$, are independent.. Note that we do not consider the lengths of the $m_0$ external edges leading to the leaves of the reduced tree \eqref{red516} and the initial $m_0-1$ internal edges, which does not affect the asymptotics. 
We will use the representation of a Dirichlet vector $X'\sim {\rm Dirichlet}(\beta, \ldots, \beta, 1-2\beta, \ldots, 1-2\beta)$ in terms of independent Gamma variables, i.e.
\begin{equation*} X'\,{\buildrel d \over=}\,Y^{-1}\left(Y_1, \ldots,Y_{m-1}, Y_1', \ldots, Y'_{m}\right) \end{equation*}
for independent i.i.d. sequences $(Y_j, j \geq 1)$, $(Y_j', j \geq 1)$ with $Y_1 \sim {\rm Gamma}(\beta, 1)$, $Y_1 \sim {\rm Gamma}(1-2\beta,1)$, and $Y=\sum_{j\in[m-1]} Y_j + \sum_{j\in[m]}Y_j' \sim {\rm Gamma}((m-1)(1-\beta)+ 1-2\beta,1)$. By \eqref{scaling2}, \begin{align*}
C_m^{(i)} \widetilde{S}_{m_0,i} &= C_m^{(i)} \mu^*_{k_m^{(i)}}\left(\mathcal R_{k_m^{(i)}}^{(i)}\right)^{\beta} \left(\sum_{j\in[N_m+(m_0-1)]} \left(X'_j\right)^{\beta} M^{(j)} + \sum_{j=0}^{m_0-1} \left(X'_{j+m}\right)^{\beta} \overline{M}^{(j)} \right)\end{align*}
where $\overline{M}_j$, $j \geq 1$, are i.i.d. with $\overline{M}^{(1)} \sim {\rm ML}(\beta, 1-2\beta)$ and independent of $X'$ and $N_m$, $m \geq 1$, and hence $C_m^{(i)} \widetilde{S}_{m_0,i}$ has the same distribution as \begin{align}
\frac{N_m\left(1-\beta\right)^{\beta}}{m^{\beta/(1-\beta)}} \left(m^{-1} \left(\sum_{j\in[m-1]} Y_j + \sum_{j\in[m]}Y_j' \right)\right)^{-\beta} \left(N_m^{-1} \left(\sum_{j\in[N_m+(m_0-1)]} Y_j^{\beta}M^{(j)}+\sum_{j\in[m_0]}Y_j'^{\beta}\overline{M}^{(j)} \right)\right). \label{cim}
\end{align}

By the strong law of large numbers, we have $\lim_{m \rightarrow \infty} N_m^{-1} \sum_{j\in[N_m]} Y_j^{\beta}M_m^{(j)} = \mathbb E[Y_1^\beta M_m^{(j)}]=1$ a.s. since $N_m \rightarrow \infty$ a.s., $\mathbb E[Y_1^{\beta}]=\Gamma(2\beta)/\Gamma(\beta)$, and where we use the first moment of the Mittag-Leffler distribution \eqref{mlmoms}. Furthermore, note that $Y_j'':=Y_j+Y_j' \sim {\rm Gamma}(1-\beta, 1)$, $j \in [m-1]$, are i.i.d., and hence $m^{-1} (\sum_{j\in[m-1]}Y_j + \sum_{j\in[m]}Y_j') \rightarrow \mathbb{E}[Y_1'']= 1-\beta$ a.s..
By \eqref{scaling1}, we conclude that the expression in \eqref{cim} converges to $S_{m_0}'$ a.s. where $S_{m_0}' \sim {\rm ML}(\beta', m_0-\beta')$. We already know that $\mathcal R^{(i)}_{k_m^{(i)}}$ and the scaling factor $C_m^{(i)}$ converge almost-surely, and hence, by Proposition \ref{Forddis}, \begin{equation*} \lim_{m \rightarrow \infty} C_m^{(i)} \mathcal R_{k_{m_0}^{(i)}}^{(i)}=\mathcal F_{m_0}^{(i)} \quad \text{ a.s. } \end{equation*} for $(\mathcal F_m^{(i)}, m\geq 1)$ are i.i.d. Ford tree growth processes of index $\beta'$, i.e. (ii) follows as $m_0 \rightarrow \infty$. \qedhere \end{proof}


\bibliographystyle{abbrv}
\bibliography{BinEmb}	 

\nocite{*}

\end{document}